\date{} 
\title{Imaginary geometry II: reversibility of $\SLE_\kappa(\rho_1;\rho_2)$ for $\kappa \in (0,4)$}
\author{Jason Miller and Scott Sheffield}
\documentclass[12pt,naturalnames]{article}
\usepackage{amsmath}
\usepackage{amssymb}
\usepackage{amsthm}
\usepackage{amsfonts}
\usepackage{graphicx}
\usepackage{tabularx}
\usepackage{caption}
\usepackage{enumerate}
\usepackage[normalem]{ulem}
\usepackage[protrusion=true,expansion=true]{microtype}
\usepackage{morefloats}
\usepackage{color}
\input epsf.tex

\def\@rst #1 #2other{#1}
\newcommand\MR[1]{\relax\ifhmode\unskip\spacefactor3000 \space\fi
  \MRhref{\expandafter\@rst #1 other}{#1}}
\newcommand{\MRhref}[2]{\href{http://www.ams.org/mathscinet-getitem?mr=#1}{MR#2}}

\newcommand{\giv}{\,|\,}

\usepackage[margin=1.3in]{geometry}
\allowdisplaybreaks

\newif\ifhyper\IfFileExists{hyperref.sty}{\hypertrue}{\hyperfalse}

\ifhyper\usepackage{hyperref}\fi

\newif\ifdraft
\drafttrue
\def\note#1/{\ifdraft {\bf [#1]}\fi}
\long\def\comment#1{}
\numberwithin{equation}{section}
\numberwithin{figure}{section}

\newtheorem{theorem}{Theorem}
\numberwithin{theorem}{section}

\newtheorem{lemma}[theorem]{Lemma}
\newtheorem{proposition}[theorem]{Proposition}

\theoremstyle{remark}\newtheorem{definition}[theorem]{Definition}
\theoremstyle{remark}\newtheorem{remark}[theorem]{Remark}

\usepackage{subfigure}
\usepackage{graphicx}

\newcommand{\C}{\mathbf{C}}
\newcommand{\D}{\mathbf{D}}
\newcommand{\E}{\mathbf{E}}

\newcommand{\N}{\mathbf{N}}

\newcommand{\p}{\mathbf{P}}
\newcommand{\Q}{\mathbf{Q}}
\newcommand{\R}{\mathbf{R}}

\newcommand{\h}{\mathbf{H}}

\newcommand{\Fg}{\mathfrak {g}}
\newcommand{\Fh}{\mathfrak {h}}

\newcommand{\CA}{\mathcal {A}}

\newcommand{\CC}{\mathcal {C}}

\newcommand{\CF}{\mathcal {F}}

\newcommand{\CK}{\mathcal {K}}
\newcommand{\CL}{\mathcal {L}}

\newcommand{\CR}{\mathcal {R}}
\newcommand{\CS}{\mathcal {S}}

\newcommand{\CV}{\mathcal {V}}
\newcommand{\CZ}{\mathcal {Z}}

\newcommand{\im}{{\rm Im}}
\newcommand{\re}{{\rm Re}}

\newcommand{\SLE}{{\rm SLE}}
\newcommand{\CLE}{{\rm CLE}}

\newcommand{\strip}{\CS}
\newcommand{\striptop}{\partial_U \CS}
\newcommand{\stripbot}{\partial_L \CS}

\newcommand{\vstrip}{\CV}
\newcommand{\vstripleft}{\partial_L \vstrip}
\newcommand{\vstripright}{\partial_R \vstrip}

\newcommand{\vhstrip}{\CV^+}
\newcommand{\vhstripleft}{\partial_L \vhstrip}
\newcommand{\vhstripright}{\partial_R \vhstrip}

\newcommand{\wh}{\widehat}
\newcommand{\wt}{\widetilde}
\newcommand{\ol}{\overline}
\newcommand{\ul}{\underline}
\newcommand{\one}{{\bf 1}}

\def\CLEkk#1/{$\mathrm{CLE}(#1)$}
\def\CLEk/{\CLEkk{\kappa}/}
\def\CLE/{$\mathrm{CLE}$}

\def\Ito/{It\^o}

\def \E {{\bf E}}

\def\hcap{{\rm hcap}}

\newcommand{\refrho}{\widehat}

\begin{document}
\maketitle

\begin{abstract}
Given a simply connected planar domain $D$, distinct points $x,y \in \partial D$, and $\kappa >0$, the {\em Schramm-Loewner evolution} $\SLE_\kappa$ is a random continuous non-self-crossing path in $\overline D$ from $x$ to $y$.
The $\SLE_\kappa(\rho_1;\rho_2)$ processes, defined for $\rho_1, \rho_2 > -2$, are in some sense the most natural generalizations of $\SLE_\kappa$.

When $\kappa \leq 4$, we prove that the law of the time-reversal of an $\SLE_\kappa(\rho_1;\rho_2)$ from $x$ to $y$ is, up to parameterization, an $\SLE_\kappa(\rho_2;\rho_1)$ from $y$ to $x$.
This assumes that the ``force points'' used to define $\SLE_\kappa(\rho_1;\rho_2)$ are immediately to the left and right of the $\SLE$ seed. A generalization to arbitrary (and arbitrarily many) force points applies whenever the path does not (or is conditioned not to) hit $\partial D \setminus \{x,y\}$.

The proof of time-reversal symmetry makes use of the interpretation of $\SLE_{\kappa}(\rho_1; \rho_2)$ as a ray of a random geometry associated to the Gaussian free field.  Within this framework, the time-reversal result allows us to couple two instances of the Gaussian free field (with different boundary conditions) so that their difference is almost surely constant on either side of the path.  In a fairly general sense, adding appropriate constants to the two sides of a ray reverses its orientation.
\end{abstract}

\newpage
\tableofcontents
\newpage

\setlength{\parskip}{0.25cm plus1mm minus1mm}

\medbreak {\noindent\bf Acknowledgments.}  We thank Oded Schramm, Wendelin Werner, David Wilson and Dapeng Zhan for helpful discussions.  We also thank several referees for many helpful comments.

\section{Introduction}

For each simply connected Jordan domain $D \subseteq \C$ and distinct pair $x,y \in \partial D$, the {\em Schramm-Loewner evolution} of parameter $\kappa > 0$ ($\SLE_\kappa$) describes the law of a random continuous path from $x$ to $y$ in $\overline D$.  This path is almost surely a continuous, simple curve when $\kappa \leq 4$.  It is almost surely a continuous, non-space-filling curve that intersects itself and the boundary when $\kappa \in (4,8)$, and it is almost surely a space-filling curve when $\kappa \geq 8$ \cite{RS05, LSW04}.  We recall the basic definitions of $\SLE_\kappa$ in Section~\ref{subsec::SLEoverview}.

While $\SLE_\kappa$ curves were introduced by Schramm in \cite{S0}, the following fact was proved only much more recently by Zhan in \cite{Z_R_KAPPA}: if $\eta$ is an $\SLE_\kappa$ process for $\kappa \in (0,4]$ from $x$ to $y$ in $D$ then the law of the time-reversal of $\eta$ is, up to monotone reparameterization, that of an $\SLE_\kappa$ process from $y$ to $x$ in $D$.  This is an extremely natural symmetry.

Since their introduction in \cite{S0}, it has been widely expected that $\SLE_\kappa$ curves would exhibit this symmetry for all $\kappa \leq 8$.\footnote{This was the final problem in a series presented by Schramm at ICM 2006 \cite{MR2334202}.}  One reason to expect this to be true is that $\SLE_\kappa$ has been conjectured and in some cases proved to arise as a scaling limit of discrete models that enjoy a discrete analog of time-reversal symmetry.  However, it is far from obvious from the definition of $\SLE_\kappa$ why such a symmetry should exist.

It was further conjectured by Dub\'edat that when $\kappa \leq 4$ the property of time-reversal symmetry is also enjoyed by the so-called $\SLE_\kappa(\rho_1;\rho_2)$ processes, which depend on parameters $\rho_1,\rho_2 > -2$, and which are in some sense the most natural generalizations of $\SLE_\kappa$.  We recall the definition of $\SLE_\kappa(\rho_1;\rho_2)$ in Section~\ref{subsec::SLEoverview}.  (The conjecture assumes the so-called {\em force points}, whose definition we recall in Section~\ref{subsec::SLEoverview},  are located immediately left and right of the $\SLE$ seed.)  Dub\'edat's conjecture was later proved in the special case of $\SLE_\kappa(\rho)$ processes (obtained by setting one of the~$\rho_i$ to~$0$ and the other to~$\rho$), under the condition that $\rho$ is in the range of values for which the curve almost surely does not intersect the boundary.  This was accomplished by Dub\'edat \cite{DUB_DUAL} and Zhan \cite{Z_R_KAPPA_RHO} using a generalization of the technique used to prove the reversibility of $\SLE_\kappa$ in \cite{Z_R_KAPPA}.

This article will prove Dub\'edat's conjecture in its complete generality using several new methods that we hope will be of independent interest.  In particular, we will give a new proof of the time-reversal symmetry of ordinary $\SLE_\kappa$ for $\kappa < 4$ (inspired by an unpublished argument of Schramm and second author), which is independent of the arguments in \cite{Z_R_KAPPA, DUB_DUAL, Z_R_KAPPA_RHO}.

One can generalize $\SLE_\kappa(\rho_1;\rho_2)$ theory to multiple force points, as we recall in Section~\ref{subsec::SLEoverview}; when there are more than two, we often use $\ul \rho$ to denote the corresponding vector of $\rho$ values. The time-reversal of $\SLE_\kappa(\ul{\rho})$ with multiple force points (or a force point not immediately adjacent to the $\SLE$ seed) need not be an $\SLE_\kappa(\ul{\rho})$, as illustrated in \cite{DUB_DUAL, Z_R_KAPPA_RHO}.  However, another result of the current paper is that in general the time-reversal of an $\SLE_\kappa(\ul{\rho})$ that does not (or is conditioned not to) hit the boundary is also an $\SLE_\kappa(\ul{\rho})$ that does not (or is conditioned not to) hit the boundary, with appropriate force points.  A similar result applies if the $\SLE_\kappa(\ul{\rho})$ hits the boundary only on one of the two boundary arcs connecting $x$ and $y$ (and there are no force points in the interior of that arc).

This paper is a sequel to and makes heavy use of a recent work of the authors \cite{MS_IMAG}, which in particular proves the almost sure continuity of general $\SLE_\kappa(\ul{\rho})$ traces, even those that hit the boundary. The results of the current paper have various applications to the theory of ``imaginary geometry'' described in \cite{MS_IMAG}, to Liouville quantum gravity, and to SLE theory itself.

In particular, they will play a crucial role in a subsequent work by the authors that will give the first proof of the time-reversal symmetry of $\SLE_\kappa$ and $\SLE_\kappa(\rho_1;  \rho_2)$ processes that applies when $\kappa \in (4,8)$ \cite{MS_IMAG3}.  Interestingly, we will find in \cite{MS_IMAG3} that when $\kappa \in (4,8)$ the $\SLE_\kappa(\rho_1;\rho_2)$ processes are reversible if and only if $\rho_i \geq \tfrac{\kappa}{2}-4$ for $i \in \{1,2\}$.  The threshold $\tfrac{\kappa}{2}-4$ is significant because, when $\kappa \in (4,8)$, the $\SLE_\kappa(\rho_1;\rho_2)$ curves almost surely hit every point on the entire left (resp.\ right) boundary of $D$ if and only if $\rho_1 \leq \tfrac{\kappa}{2}-4$ (resp.\ $\rho_2 \leq \tfrac{\kappa}{2}-4$).  Thus, aside from the critical cases, the ``non-boundary-filling'' $\SLE_\kappa(\rho_1;\rho_2)$ curves are the ones with time-reversal symmetry.

The time-reversal symmetries that apply when $\kappa \geq 8$ will be addressed in the fourth work of the current series \cite{MS_INTERIOR}. When $\kappa \geq 8$, we will see that one has time-reversal symmetry only for one special pair of $\rho_1, \rho_2$ values; however, in the $\kappa \geq 8$ context, it is possible to describe time-reversals of $\SLE_\kappa(\rho_1;\rho_2)$ processes more generally in terms of $\SLE_\kappa(\refrho{\rho}_1;\refrho{\rho}_2)$ processes for certain values of $\refrho{\rho}_1$ and $\refrho{\rho}_2$.  We will also show in \cite{MS_INTERIOR} that certain families of ``whole-plane'' variants of $\SLE_\kappa$ have time-reversal symmetry as well, generalizing a recent work of Zhan on this topic \cite{2010arXiv1004.1865Z}.

\subsection{Main results}

The following is our first main result:

\begin{theorem}
\label{thm::all_reversible}
Suppose that $\eta$ is an $\SLE_\kappa(\rho_1;\rho_2)$ process in a Jordan domain $D$ from $x$ to $y$, with $x,y \in \partial D$ distinct and weights $\rho_1,\rho_2 > -2$ corresponding to force points located at $x^-,x^+$, respectively.  The law of the time-reversal $\mathcal R(\eta)$ of $\eta$ is, up to reparameterization, an $\SLE_\kappa(\rho_2;\rho_1)$ process in $D$ from $y$ to $x$ with force points located at $y^-,y^+$, respectively.  Thus, the law of $\eta$ as a random set is invariant under an anti-conformal map that swaps $x$ and $y$.
\end{theorem}

The proof of Theorem~\ref{thm::all_reversible} has two main parts.  The first part is to establish the reversibility of $\SLE_\kappa(\rho)$ processes with a single force point located at the $\SLE$ seed, even when they hit the boundary.
This extends the one-sided result of \cite{DUB_DUAL, Z_R_KAPPA_RHO} to the boundary-intersecting regime.
The second part is to extend this result to $\SLE_\kappa(\rho_1;\rho_2)$ processes.

This second part of the proof will be accomplished using so-called \emph{bi-chordal} $\SLE$ processes to reduce the two-force-point problem to the single-force-point case.  The bi-chordal processes we use will be probability measures on pairs of non-crossing paths $(\eta_1,\eta_2)$ in $D$ from $x$ to $y$ with the property that the conditional law of $\eta_1$ given $\eta_2$ is an $\SLE_\kappa(\ul{\rho}^L)$ in the left connected component of $D \setminus \eta_2$ and the law of $\eta_2$ given $\eta_1$ is an $\SLE_\kappa(\ul{\rho}^R)$ in the right connected component of $D \setminus \eta_1$.  We use the superscript ``$R$'' to indicate that the force points associated with $\eta_1$ lie on the counterclockwise arc of $\partial D$ between the initial and terminal points of $\eta_1$.  Likewise, the superscript ``$L$'' indicates that the force points associated with $\eta_2$ lie on the clockwise arc of $\partial D$ between the initial and terminal points of $\eta_2$.  We will prove in a rather general setting that this information (about the conditional law of each $\eta_i$ given the other) completely characterizes the joint law of $(\eta_1,\eta_2)$, a result we consider independently interesting.  One can then use the imaginary geometry constructions from \cite{MS_IMAG} to explicitly produce processes in which each $\eta_i$ is a one-sided $\SLE_\kappa(\rho)$ when restricted to the complement of the other, but the marginal law of each path is an $\SLE_\kappa(\rho_1;\rho_2)$ process in the whole domain.  The time-reversal symmetry of the individual $\SLE_\kappa(\rho)$ processes can then be used to prove the time-reversal symmetry of $\SLE_\kappa(\rho_1;\rho_2)$.

Another important point for us will be to show the law of $\SLE_\kappa(\rho)$ is uniquely determined by certain type of domain Markov property, much like the one characterizes ordinary $\SLE_\kappa$ \cite{S0}.  A simple heuristic argument of a statement of this kind appears at the beginning of \cite[Section~8.3]{CONF_RES_CHORDAL}, which is where the $\SLE_\kappa(\rho)$ processes were first defined.  It is noted there that the SDE driving single-force-point $\SLE_\kappa(\rho)$ is the only one with this type of property.  
Our proof of the reversibility of $\SLE_\kappa(\rho)$ requires a particular precise characterization of the type indicated in \cite[Section~8.3]{CONF_RES_CHORDAL} (see also the discussion in \cite[Section~7.2]{WER_CONF_RES}).  That is, we will need that $\SLE_\kappa(\rho)$ is characterized by the following version of conformal invariance and the domain Markov property.  Let $c = (D,x,y;z)$ be a configuration which consists of a Jordan domain $D \subseteq \C$ and $x,y,z \in \partial D$ and $x \neq y$.  We let $\CC_L$ (resp.\ $\CC_R$) be the collection of configurations $c = (D,x,y;z)$ where $z$ lies on the clockwise (resp.\ counterclockwise) arc of $\partial D$ from $x$ to $y$.

\begin{definition}[Conformal Invariance]
\label{def::conf_invariance}
We say that a family $(\p_c : c \in \CC_q)$, $q \in \{L,R\}$, where $\p_c$ is a probability measure on continuous paths from $x$ to $y$ in $D$, is conformally invariant if the following is true.  Suppose that $c = (D,x,y;z)$, $c' = (D',x',y';z') \in \CC_q$, and $\psi \colon D \to D'$ is a conformal map with $\psi(x) = x'$, $\psi(y) = y'$, and $\psi(z) = z'$.  Then for $\eta \sim \p_c$, we have that $\psi(\eta) \sim \p_{c'}$, up to reparameterization.
\end{definition}

\begin{definition}[Domain Markov Property]
\label{def::domain_markov}
We say that a family $(\p_c : c \in \CC_q)$, $q \in \{L,R\}$, where $\p_c$ is a probability measure on continuous paths from $x$ to $y$ in $D$, satisfies the domain Markov property if for all $c \in \CC_q$ the following is true.  Suppose $\eta \sim \p_c$.  Then for every $\eta$ stopping time $\tau$, the law of $\eta|_{[\tau,\infty)}$ conditional on $\eta|_{[0,\tau]}$ is, up to reparameterization, given by $\p_{c_\tau}$ where $c_\tau = (D_\tau,\eta(\tau),y;z_\tau)$.  Here, $D_\tau$ is the connected component of $D \setminus \eta([0,\tau])$ which contains $y$ on its boundary.  If $q = L$ (resp.\ $q = R$), then $z_\tau$ is the leftmost (resp.\ rightmost) point on the clockwise (resp.\ counterclockwise) arc of $\partial D$ from $x$ to $y$ which lies to the right (resp.\ left) of $z$ and $\eta([0,\tau]) \cap \partial D$.
\end{definition}

Our conformal Markov characterization of $\SLE_\kappa(\rho)$ is the following:

\begin{theorem}
\label{thm::conformal_markov}
Suppose that $(\p_c : c \in \CC_q)$, $q \in \{L,R\}$, is a conformally invariant family which satisfies the domain Markov property in the sense of Definitions~\ref{def::conf_invariance} and~\ref{def::domain_markov}.  Assume further that when $c = (D,x,y;z) \in \CC_q$ and $D$ has smooth boundary and $\eta \sim \p_c$, the Lebesgue measure of $\eta \cap \partial D$ is zero almost surely.  Then there exists $\rho > -2$ such that for each $c = (D,x,y;z) \in \CC_q$, $\p_c$ is the law of an $\SLE_\kappa(\rho)$ process in $D$ from $x$ to $y$ with a single force point at $z$.
\end{theorem}

Theorem~\ref{thm::conformal_markov} is a generalization of the conformal Markov characterization of ordinary $\SLE_\kappa$ established by Schramm \cite{S0} (and used by Schramm \cite{S0} to characterize $\SLE_\kappa$ processes) but with the addition of one extra marked point.  It is implicit in the hypotheses that $\eta \sim \p_c$ cannot cross itself and also never enters the loops it creates with segments of the boundary or itself as it moves from $x$ to $y$.  This combined with the hypothesis that $\eta \cap \partial D$ has zero Lebesgue measure almost surely when $D$ is smooth implies that $\eta$ has a continuous Loewner driving function \cite[Section~6.2]{MS_IMAG} and that the evolution of the marked point under the uniformizing conformal maps is described by the Loewner flow (Lemma~\ref{lem::x_differential}).  The proof makes use of a characterization of continuous self-similar Markov processes due to Lamperti \cite[Theorem~5.1]{LAMP72}.  (In the setting that we consider, this is a rescaled version of the fact that ``the only continuous Markov process with stationary increments are the Brownian motions with drift.'')

The next step in the proof of the reversibility of $\SLE_\kappa(\rho)$ for $\rho > -2$ is to show that the time-reversal $\CR(\eta)$ of $\eta \sim \SLE_\kappa(\rho)$ for $\rho \in (-2,0]$ satisfies the criteria of Theorem~\ref{thm::conformal_markov}.  This is in some sense the heart of the argument.  We will first present a new proof in the case that $\rho = 0$, which is related to an argument sketch obtained (but never published) by the second author and Schramm several years ago.  The idea is to try to make sense of conditioning on a flow line of the Gaussian free field, whose law is an $\SLE_\kappa$, up to a {\em reverse} stopping time.  The conditional law of the initial part of the flow line is then in some sense a certain $\SLE_\kappa(\rho_1; \rho_2)$ process conditioned to merge into the tip of that flow line.  We make this idea (which involves conditioning on an event of probability zero) precise using certain bi-chordal $\SLE_\kappa$ constructions.  We will use similar tricks to establish a conformal Markov property for the time-reversal of $\SLE_\kappa$, which is then extended to give a similar property for $\SLE_\kappa(\rho)$.

The above arguments will imply that the time-reversal of an $\SLE_\kappa(\rho)$ is itself an $\SLE_\kappa(\wt\rho)$ for some $\wt{\rho} > -2$.  This will imply that there exists a function $R$ such that $\CR(\eta) \sim \SLE_\kappa(R(\rho))$.  One can then easily observe that the function $R$ is continuous and increasing and satisfies $R(R(\rho)) = \rho$ which implies $R(\rho) = \rho$.  Using another trick involving bi-chordal $\SLE$ configurations, we can extend the reversibility of $\SLE_\kappa(\rho)$ to all $\rho > -2$.

Using the interpretation of $\SLE_\kappa(\ul{\rho}^L;\ul{\rho}^R)$ processes as flow lines of the Gaussian free field with certain boundary data \cite{MS_IMAG}, we will also give a description of the time-reversal $\CR(\eta)$ of $\eta \sim \SLE_\kappa(\ul{\rho}^L;\ul{\rho}^R)$ processes with many force points, provided $\eta$ is almost surely non-boundary intersecting.

\begin{theorem}
\label{thm::multiple_force_points}
Suppose that $\eta$ is an $\SLE_\kappa(\ul{\rho}^L;\ul{\rho}^R)$, in a Jordan domain $D$ from $x$ to $y$, with $x,y \in \partial D$ distinct, that does not (or is conditioned not to) hit $\partial D$ except at $x$ and $y$.  Then the time-reversal $\CR(\eta)$ of $\eta$ is an $\SLE_\kappa(\ul{\rho}^L;\ul{\rho}^R)$ process (with appropriate force points) that does not (or is conditioned not to) hit $\partial D$ except at $x$ and $y$.
\end{theorem}

A more precise discussion and complete description of how to construct the law of the time-reversal (in particular how to set up the various $\rho$ values) appears in Section~\ref{sec::multiple_force_points}.  In order to make the theorem precise, we will in particular have to make sense of what it means to condition a path not to hit the boundary (which in some cases involves conditioning on a probability zero event).

\subsection{Relation to previous work}

As we mentioned earlier, the reversibility of $\SLE_\kappa$ for $\kappa \in (0,4]$ was first proved by Zhan \cite{Z_R_KAPPA} but also appears in the work of Dub\'edat \cite{DUB_DUAL}.  Both proofs are based on a beautiful technique that allows one to construct a \emph{coupling} of $\eta \sim \SLE_\kappa$ from $x$ to $y$ in $D$ with $\wt{\eta} \sim \SLE_\kappa$ in $D$ from $y$ to $x$ such that the two paths \emph{commute}.  In other words, one has a recipe for growing the paths one at a time, in either order, that produces the same overall joint law.  In the coupling of \cite{Z_R_KAPPA}, the joint law is shown to have the property that for every $\eta$ stopping time $\tau$, the law of $\wt{\eta}$ given $\eta|_{[0,\tau]}$ is an $\SLE_\kappa$ in the connected component of $D \setminus \eta([0,\tau])$ containing $y$ from $y$ to $\eta(\tau)$.  The same likewise holds when the roles of $\eta$ and $\wt{\eta}$ are reversed.  This implies that $\eta$ contains a dense subset of $\wt{\eta}$ and vice-versa.  Thus the continuity of $\eta$ and $\wt{\eta}$ implies that $\wt{\eta}$ is almost surely the time-reversal (up to reparameterization) of $\eta$.  In particular, the time-reversal $\CR(\eta)$ of $\eta$ is an $\SLE_\kappa$ in $D$ from $y$ to $x$.  The approaches of both Dub\'edat \cite{DUB_DUAL} and Zhan \cite{Z_R_KAPPA_RHO} to the reversibility of non-boundary intersecting $\SLE_\kappa(\rho)$ are also based on considering a commuting pair of $\SLE_\kappa(\rho)$ processes $\eta,\wt{\eta}$ growing at each other.\footnote{The results in \cite{DUB_DUAL, Z_R_KAPPA_RHO} only apply if $\rho$ is in the range for which the path avoids the boundary almost surely.  However, it is possible that the arguments could be extended to the boundary-hitting case of one-force-point $\SLE_\kappa(\rho)$. Zhan told us privately before we wrote \cite{MS_IMAG} that he believed the techniques in \cite{Z_R_KAPPA_RHO} could be extended to the single-force-point boundary-intersecting case if the continuity result of \cite{MS_IMAG} were known.}  The difference from the setup of ordinary $\SLE_\kappa$ is that in such a coupling, the conditional law of $\eta$ given $\wt{\eta}|_{[0,\wt{\tau}]}$, $\wt{\tau}$ an $\wt{\eta}$ stopping time, is not an $\SLE_\kappa(\rho)$ process in the connected component of $D \setminus \wt{\eta}([0,\wt{\tau}])$ containing $x$ from $x$ to $\wt{\eta}(\wt{\tau})$.  Rather, it is a more complicated variant of $\SLE_\kappa$, a so-called \emph{intermediate $\SLE$}.

Because our approach to $\SLE_\kappa(\rho)$ reversibility is somewhat different from the methods in \cite{DUB_DUAL, Z_R_KAPPA_RHO} (and in particular we use a domain Markov property to characterize the time-reversal), we will not actually need to define intermediate $\SLE$ explicitly (e.g., by giving an explicit formula for the Loewner drift term).  We will also avoid the analogous explicit calculations in the multiple force point cases with tricks involving bi-chordal resampling and the Gaussian free field.

We also remark that in \cite{WER_GIR} it is shown that an $\SLE_\kappa(\rho)$ process with $\rho > 0$ can be viewed as arising from an $\SLE_\kappa$ process conditioned not to hit a sample of a one-sided restriction measure.  This, combined with the reversibility of $\SLE_\kappa$ for $\kappa \in (0,4]$ \cite{DUB_DUAL,Z_R_KAPPA} yields another proof of the reversibility of $\SLE_\kappa(\rho)$ for $\kappa \in (0,4]$ and $\rho > 0$.  Also, it is shown in \cite{WER_WU_CLE_SLE_KR} that the boundary intersecting $\SLE_\kappa(\rho)$ processes when $\kappa \in (8/3,4]$ arise naturally in the context of the Brownian loop soups and this implies the reversibility of the $\SLE_\kappa(\rho)$ for $\kappa$ and $\rho$ in this regime.

Part of this paper addresses the question of making sense of $\SLE_\kappa(\rho)$ (and its variants) conditioned not to hit the boundary.  The simplest version of this is a consequence of the fact that a Bessel process of dimension $d \leq 2$ conditioned not to hit $0$ is a Bessel process of dimension $4-d \geq 2$ and various statements based on this fact have appeared in other places in the literature, for example \cite{DUB_MG_DUALITY}.  In the present work, we need to establish particular statements of this form that hold in the presence of many force points and also in the case of multiple paths.  For example, we need to know that the joint laws of certain families of non-intersecting paths are uniquely characterized by certain Gibbs properties.  Proving this requires some re-sampling tricks beyond the usual Girsanov martingale-weighting procedures, and we did not find a place in the literature where these kinds of characterizations were carefully stated or worked out.

\subsection{Outline}

The remainder of this article is structured as follows.  In the next section, we will review the basics of $\SLE_\kappa(\ul{\rho})$ processes.  We will also give a summary of how $\SLE_\kappa(\ul{\rho})$ processes can be viewed as flow lines of the Gaussian free field (GFF) ---  this is the so-called imaginary geometry of the GFF \cite{MS_IMAG}.  We will in particular emphasize how this interpretation can be used to construct couplings of systems of $\SLE_\kappa(\ul{\rho})$ processes and the calculus one uses in order to compute the conditional law of one such curve given the realizations of the others.  Next, in Section~\ref{sec::conformal_markov}, we will prove Theorem~\ref{thm::conformal_markov}, the conformal Markov characterization of $\SLE_\kappa(\rho)$ processes. In Section~\ref{sec::bi_chordal}, we will show in certain special cases that the joint law of a system of multiple $\SLE_\kappa(\ul{\rho})$ strands is characterized by the conditional laws of the individual strands.  This provides an alternative mechanism for constructing systems of $\SLE_\kappa$ type curves in which it is easy to compute the conditional law of one of the curves given the others.  It is the key tool for deducing the reversibility of $\SLE_\kappa(\rho_1;\rho_2)$ from the reversibility of $\SLE_\kappa(\rho)$.   In Section~\ref{sec::conditioned_not_to_hit} we discuss how to make sense of $\SLE_\kappa(\ul{\rho})$ processes conditioned not to intersect certain boundary segments.

We will then combine the above elements in Section~\ref{sec::time_reversal_conformal_markov} to show that the time-reversal of an $\SLE_\kappa(\rho)$ process satisfies the conformal Markov property.  In Section~\ref{sec::proof} we will complete the proof of Theorem~\ref{thm::all_reversible} by deducing the reversibility of $\SLE_\kappa(\rho_1;\rho_2)$ from the reversibility of $\SLE_\kappa(\rho)$.  We finish in Section~\ref{sec::multiple_force_points} by proving Theorem~\ref{thm::multiple_force_points}

\section{Preliminaries}
\label{sec::preliminaries}

The purpose of this section is to review the basic properties of $\SLE_\kappa(\ul{\rho}^L;\ul{\rho}^R)$ processes in addition to giving a non-technical overview of the so-called imaginary geometry of the Gaussian free field.  The latter provides a mechanism for constructing couplings of many $\SLE_\kappa(\ul{\rho}^L;\ul{\rho}^R)$ strands in such a way that it is easy to compute the conditional law of one of the curves given the realization of the others \cite{MS_IMAG}.

\subsection{$\SLE_\kappa(\rho)$ processes} \label{subsec::SLEoverview}

$\SLE_\kappa$ is a one-parameter family of conformally invariant random curves, introduced by Oded Schramm in \cite{S0} as a candidate for (and later proved to be) the scaling limit of loop erased random walk \cite{LSW04} and the interfaces in critical percolation \cite{S01, CN06}.  Schramm's curves have been shown so far also to arise as the scaling limit of the macroscopic interfaces in several other models from statistical physics: \cite{SS09,S07,CS10U,SS05,MillerSLE}.  More detailed introductions to $\SLE$ can be found in many excellent survey articles of the subject, e.g., \cite{W03, LAW05}.

An $\SLE_\kappa$ in $\h$ from $0$ to $\infty$ is defined by the random family of conformal maps $g_t$ obtained by solving the Loewner ODE
\begin{equation}
\label{eqn::loewner_ode}
\partial_t g_t(z) = \frac{2}{g_t(z) - W_t},\quad g_0(z) = z
\end{equation}
where $W = \sqrt{\kappa} B$ and $B$ is a standard Brownian motion.  Write $K_t := \{z \in \h: \tau(z) \leq t \}$.  Then $g_t$ is a conformal map from $\h_t := \h \setminus K_t$ to $\h$ satisfying $\lim_{|z| \to \infty} |g_t(z) - z| = 0$.

Rohde and Schramm showed that there almost surely exists a curve $\eta$ (the so-called $\SLE$ \emph{trace}) such that for each $t \geq 0$ the domain $\h_t$ is the unbounded connected component of $\h \setminus \eta([0,t])$, in which case the (necessarily simply connected and closed) set $K_t$ is called the ``filling'' of $\eta([0,t])$ \cite{RS05}.  An $\SLE_\kappa$ connecting boundary points $x$ and $y$ of an arbitrary simply connected Jordan domain can be constructed as the image of an $\SLE_\kappa$ on $\h$ under a conformal transformation $\psi \colon \h \to D$ sending $0$ to $x$ and $\infty$ to $y$.  (The choice of $\psi$ does not affect the law of this image path, since the law of $\SLE_\kappa$ on $\h$ is scale invariant.)  $\SLE_\kappa$ is characterized by the fact that it satisfies the domain Markov property and is invariant under conformal transformations.

$\SLE_\kappa(\ul{\rho}^L;\ul{\rho}^R)$ is the stochastic process one obtains by solving~\eqref{eqn::loewner_ode} where the driving function $W$ is taken to be the solution to the SDE
\begin{equation}
\label{eqn::sle_kappa_rho_eqn}
\begin{split}
dW_t &= \sqrt{\kappa} dB_t + \sum_{q \in \{L,R\}} \sum_{i} \frac{\rho^{i,q}}{W_t- V_t^{i,q}} dt \\
dV_t^{i,q} &= \frac{2}{V_t^{i,q} - W_t} dt,\quad V_0^{i,q} = x^{i,q}.
\end{split}
\end{equation}
The existence and uniqueness of solutions to~\eqref{eqn::sle_kappa_rho_eqn} is discussed in \cite[Section~2]{MS_IMAG}.  In particular, it is shown that there is a unique solution to~\eqref{eqn::sle_kappa_rho_eqn} until the first time $t$ that $W_t = V_t^{j,q}$ where $\sum_{i=1}^j \rho^{i,q} \leq -2$ for $q \in \{L,R\}$ (we call this time the {\bf continuation threshold}).  In particular, if $\sum_{i=1}^j \rho^{i,q} > -2$ for all $1 \leq j \leq |\ul{\rho}^q|$ for $q \in \{L,R\}$, then~\eqref{eqn::sle_kappa_rho_eqn} has a unique solution for all times $t$.  This even holds when one or both of the $x^{1,q}$ are zero.  The almost sure continuity of the $\SLE_\kappa(\ul{\rho}^L;\ul{\rho}^R)$ trace is also proved in \cite{MS_IMAG}.  In Section~\ref{sec::conformal_markov}, we will prove Theorem~\ref{thm::conformal_markov}, that the single-force-point $\SLE_\kappa(\rho)$ processes are characterized by conformal invariance and the domain Markov property with one extra marked point.  This extends Schramm's conformal Markov characterization of ordinary $\SLE_\kappa$ \cite{S0}.

\subsection{Imaginary geometry of the Gaussian free field}
\label{subsec::imaginary}

We will now give an overview of the so-called \emph{imaginary geometry} of the Gaussian free field (GFF).  In this article, this serves as a tool for constructing couplings of multiple $\SLE$ strands and provides a simple calculus for computing the conditional law of one of the strands given the realization of the others \cite{MS_IMAG}.  The purpose of this overview is to explain just enough of the theory so that this article may be read and understood independently of \cite{MS_IMAG}.  We refer the reader interested in proofs of the statements we make here to \cite{MS_IMAG}.  We begin by fixing a domain $D \subseteq \C$ with smooth boundary and let $C_0^\infty(D)$ denote the space of compactly supported $C^\infty$ functions on $D$.  For $f,g \in C_0^\infty(D)$, let
\[ (f,g)_\nabla := \frac{1}{2\pi} \int_D \nabla f (x) \cdot \nabla g(x)dx\]
denote the Dirichlet inner product of $f$ and $g$ where $dx$ is the Lebesgue measure on~$D$.  Let $H(D)$ be the Hilbert space closure of $C_0^\infty(D)$ under $(\cdot,\cdot)_\nabla$.  The continuum Gaussian free field $h$ (with zero boundary conditions) is the so-called standard Gaussian on $H(D)$.  It is given formally as a random linear combination
\begin{equation}
\label{eqn::gff_definition}
 h = \sum_n \alpha_n f_n,
\end{equation}
where $(\alpha_n)$ are i.i.d.\ $N(0,1)$ and $(f_n)$ is an orthonormal basis of $H(D)$.  The GFF with non-zero boundary data $\psi$ is given by adding the harmonic extension of $\psi$ to a zero-boundary GFF $h$ (see \cite{SHE06} for a more detailed introduction to the GFF).

\begin{figure}[h!]
\begin{center}
\includegraphics[width=140mm,height=120mm,clip=true, trim = 1mm 1mm 1mm 1mm]{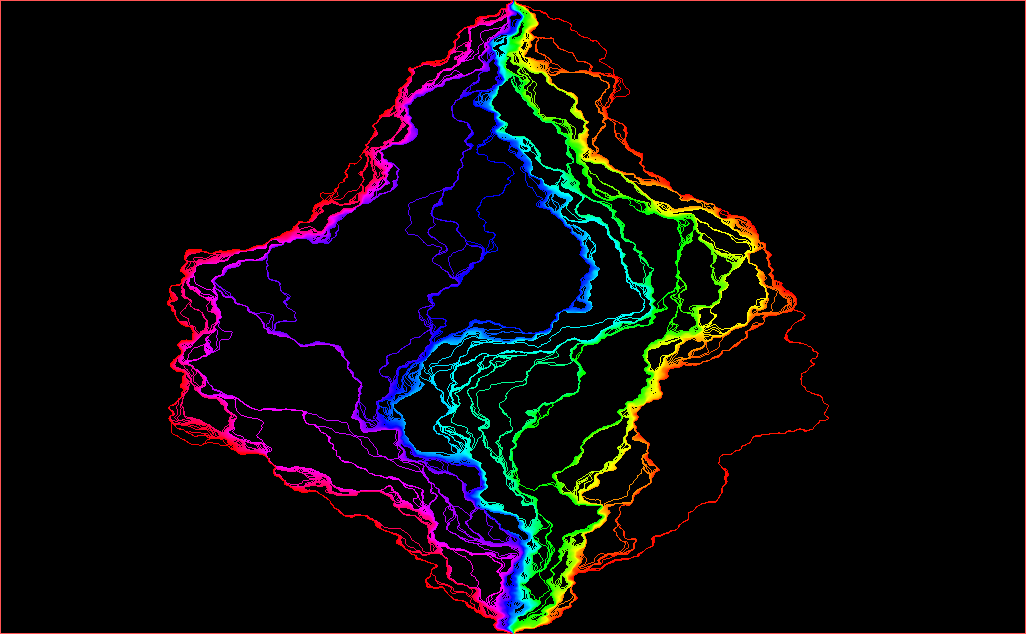}
\end{center}
\caption{ \label{fig::flowlines} Numerically generated flow lines, started at a common point, of $e^{i(h/\chi+\theta)}$ where $h$ is the projection of a GFF onto the space of functions piecewise linear on the triangles of a $300 \times 300$ grid and $\kappa=1/2$.  Different colors indicate different values of $\theta \in [-\pi/2,\pi/2]$.  We expect but do not prove that if one considers increasingly fine meshes (and the same instance of the GFF) the corresponding paths converge to limiting continuous paths.  One of the consequences of Theorem~\ref{thm::all_reversible} is that the law of the continuum analog of the random set depicted above is invariant under anti-conformal maps which swap the initial and terminal points of the paths.}
\end{figure}

The GFF is a two-dimensional-time analog of Brownian motion.  Just as Brownian motion can be realized as the scaling limit of many random lattice walks, the GFF arises as the scaling limit of many random (real or integer valued) functions on two dimensional lattices \cite{BAD96, KEN01, NS97, RV08, MillerGLCLT}.  The GFF can be used to generate various kinds of random geometric structures, in particular the imaginary geometry discussed here \cite{2010arXiv1012.4797S, MS_IMAG}.  This corresponds to considering the formal expression $e^{ih/\chi}$, for a fixed constant $\chi >0$.  Informally, the ``rays'' of the imaginary geometry are flow lines of the complex vector field $e^{i(h/\chi+\theta)}$, i.e., solutions to the ODE
\begin{equation}
\label{eqn::flowlineode} \eta'(t) = e^{i \left(h(\eta(t))/\chi+\theta\right)} \quad\text{for}\quad t > 0,
\end{equation}
for given values of $\eta(0)$ and $\theta$.

Although~\eqref{eqn::flowlineode} does not make sense as written (since $h$ is an instance of the GFF, not a continuous function), one can construct these rays precisely by solving~\eqref{eqn::flowlineode} in a rather indirect way: one begins by constructing explicit couplings of $h$ with variants of $\SLE$ and showing that these couplings have certain properties.  Namely, if one conditions on part of the curve, then the conditional law of $h$ is that of a GFF in the complement of the curve with certain boundary conditions (see Figure~\ref{fig::conditional_boundary_data}).  Examples of these couplings appear in \cite{She_SLE_lectures, SchrammShe10, DUB_PART, 2010arXiv1012.4797S} as well as variants in \cite{MakarovSmirnov09,HagendorfBauerBernard10,IzyurovKytola10}.  This step is carried out in some generality in \cite{DUB_PART, 2010arXiv1012.4797S, MS_IMAG}.  The next step is to show that in these couplings the path is almost surely {\em determined by the field} so that we can really interpret the ray as a path-valued function of the field.  This step is carried out for certain boundary conditions in \cite{DUB_PART} and in more generality in \cite{MS_IMAG}.

If $h$ is a smooth function, $\eta$ a flow line of $e^{ih/\chi}$, and $\psi \colon \wt D \to D$ a conformal transformation, then by the chain rule, $\psi^{-1}(\eta)$ is a flow line of $h \circ \psi - \chi \arg \psi'$, as in Figure~\ref{fig::coordinatechange}. With this in mind, we define an {\bf imaginary surface} to be an equivalence class of pairs $(D,h)$ under the equivalence relation
\begin{equation}
\label{eqn::ac_eq_rel}
 (D,h) \rightarrow (\psi^{-1}(D), h \circ \psi - \chi \arg \psi') = (\wt{D},\wt{h}).
\end{equation}
We interpret $\psi$ as a (conformal) {\em coordinate change} of the imaginary surface.  In what follows, we will generally take $D$ to be the upper half plane, but one can map the flow lines defined there to other domains using~\eqref{eqn::ac_eq_rel}.

\begin{figure}[h]
\begin{center}
\includegraphics[scale=0.85]{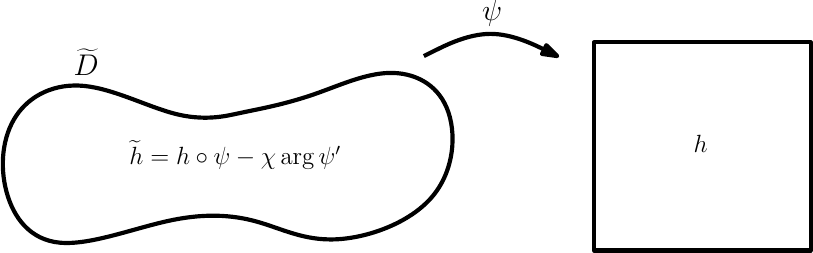}
\caption{\label{fig::coordinatechange} The set of flow lines in $\wt{D}$ will be the pullback via a conformal map $\psi$ of the set of flow lines in $D$ provided $h$ is transformed to a new function $\wt{h}$ in the manner shown.}
\end{center}
\end{figure}

\begin{figure}[h!]
\begin{center}
\subfigure{\includegraphics[scale=0.85]{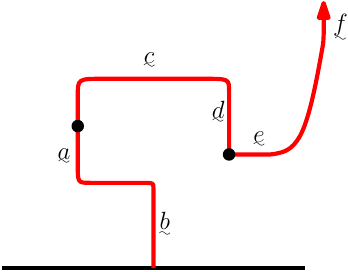}}
\hspace{0.05\textwidth}
\subfigure{\includegraphics[scale=0.85]{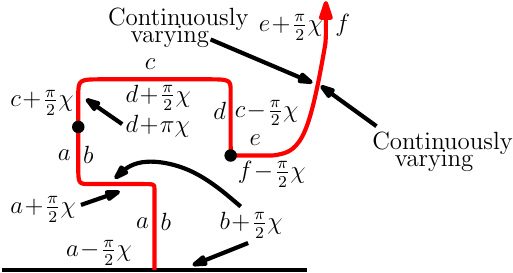}}
\caption{\label{fig::winding}  We will often make use of the notation depicted on the left hand side to indicate boundary values for Gaussian free fields.  Specifically, we will delineate the boundary $\partial D$ of a Jordan domain $D$ with black dots.  On each arc $L$ of $\partial D$ which lies between a pair of black dots, we will draw either a horizontal or vertical segment $L_0$ and label it with $\uwave{x}$.  This means that the boundary data on $L_0$ is given by $x$.  Whenever $L$ makes a quarter turn to the right, the height goes down by $\tfrac{\pi}{2} \chi$ and whenever $L$ makes a quarter turn to the left, the height goes up by $\tfrac{\pi}{2} \chi$.  More generally, if $L$ makes a turn which is not necessarily at a right angle, the boundary data is given by $\chi$ times the winding of $L$ relative to $L_0$.  If we just write $x$ next to a horizontal or vertical segment, we mean just to indicate the boundary data at that segment and nowhere else.  The right side above has exactly the same meaning as the left side, but the boundary data is spelled out explicitly everywhere.  Even when the curve has a fractal, non-smooth structure, the {\em harmonic extension} of the boundary values still makes sense, since one can transform the figure via the rule in Figure~\ref{fig::coordinatechange} to a half plane with piecewise constant boundary conditions. The notation above is simply a convenient way of describing what the constants are.  We will often include horizontal or vertical segments on curves in our figures (even if the whole curve is known to be fractal) so that we can label them this way.
}
\end{center}
\end{figure}

\begin{figure}[h!]
\begin{center}
\includegraphics[scale=0.85]{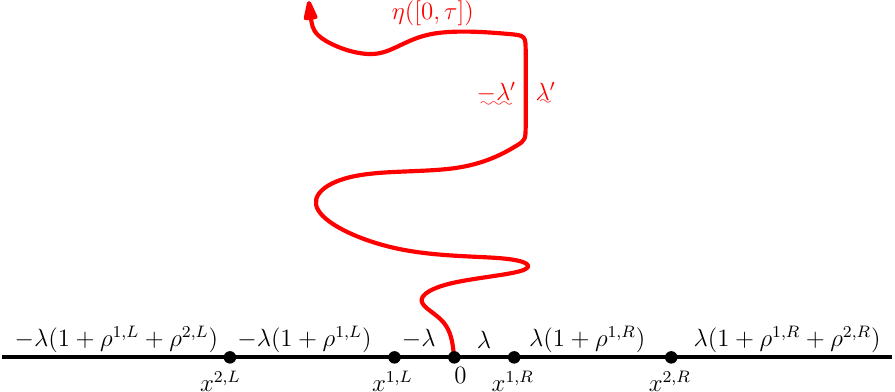}
\caption{\label{fig::conditional_boundary_data}  Suppose that $h$ is a GFF on $\h$ with the boundary data depicted above.  Then the flow line $\eta$ of $h$ starting from $0$ is an $\SLE_\kappa(\ul{\rho}^L;\ul{\rho}^R)$ curve in $\h$ where $|\ul{\rho}^L| = |\ul{\rho}^R| = 2$.  Conditional on $\eta|_{[0,\tau]}$ for any $\eta$ stopping time $\tau$, $h$ is equal in distribution to a GFF on $\h \setminus \eta([0,\tau])$ with the boundary data on $\eta([0,\tau])$ depicted above (the notation $\uwave{a}$ which appears adjacent to $\eta([0,\tau])$ is explained in some detail in Figure~\ref{fig::winding}).  It is also possible to couple $\eta' \sim\SLE_{\kappa'}(\ul{\rho}^L;\ul{\rho}^R)$ for $\kappa' > 4$ with $h$ and the boundary data takes on the same form (with $\lambda' := \tfrac{\pi}{\sqrt {\kappa'}}$ in place of $\lambda := \frac{\pi}{\sqrt \kappa}$).  The difference is in the interpretation.  The (almost surely self-intersecting) path $\eta'$ is not a flow line of $h$, but for each $\eta'$ stopping time $\tau'$ the left and right {\em boundaries} of $\eta'([0,\tau'])$ are $\SLE_{\kappa}$ flow lines, where $\kappa=16/\kappa'$, angled in opposite directions.  The union of the left boundaries --- over a collection of $\tau'$ values --- is a tree of merging flow lines, while the union of the right boundaries is a corresponding dual tree whose branches do not cross those of the tree.}
\end{center}
\end{figure}

We assume throughout the rest of this section that $\kappa \in (0,4)$ so that $\kappa' := 16/\kappa \in (4,\infty)$.  When following the illustrations, it will be useful to keep in mind a few definitions and identities:
\begin{equation} \label{eqn::deflist} \lambda := \frac{\pi}{\sqrt \kappa}, \,\,\,\,\,\,\,\,\lambda' := \frac{\pi}{\sqrt{16/\kappa}} = \frac{\pi \sqrt{\kappa}}{4} = \frac{\kappa}{4} \lambda < \lambda, \,\,\,\,\,\,\,\, \chi := \frac{2}{\sqrt \kappa} - \frac{\sqrt \kappa}{2} \end{equation}
\begin{equation} \label{eqn::fullrevolution} 2 \pi \chi = 4(\lambda-\lambda'), \,\,\,\,\,\,\,\,\,\,\,\lambda' = \lambda - \frac{\pi}{2} \chi \end{equation}
\begin{equation} \label{eqn::fullrevolutionrho}
2 \pi \chi = (4-\kappa)\lambda = (\kappa'-4)\lambda'.
\end{equation}

The boundary data one associates with the GFF on $\h$ so that its flow line from $0$ to $\infty$ is an $\SLE_\kappa(\ul{\rho}^L;\ul{\rho}^R)$ process with force points located at $\ul{x} = (\ul{x}^L,\ul{x}^R)$ is
\begin{align}
 -&\lambda\left( 1 + \sum_{i=1}^j \rho^{i,L}\right) \quad\text{for}\quad x \in [x^{j+1,L},x^{j,L}) \quad\text{and}\\
 &\lambda\left( 1 + \sum_{i=1}^j \rho^{i,R}\right) \quad\text{for}\quad x \in [x^{j,R},x^{j+1,R})
\end{align}
This is depicted in Figure~\ref{fig::conditional_boundary_data} in the special case that $|\ul{\rho}^L| = |\ul{\rho}^R| = 2$.  As we explained earlier, for any $\eta$ stopping time $\tau$, the law of $h$ conditional on $\eta([0,\tau])$ is a GFF in $\h \setminus \eta([0,\tau])$.  The boundary data of the conditional field agrees with that of $h$ on $\partial \h$.  On the right side of $\eta([0,\tau])$, it is $\lambda' + \chi \cdot {\rm winding}$, where the terminology ``winding'' is explained in Figure~\ref{fig::winding}, and to the left it is $-\lambda' + \chi \cdot {\rm winding}$.  This is also depicted in Figure~\ref{fig::conditional_boundary_data}.

By considering several flow lines of the same field (starting either at the same point or at different points), we can construct couplings of multiple $\SLE_\kappa(\ul{\rho})$ processes.  For example, suppose that $\theta \in \R$.  The flow line $\eta_\theta$ of $h+\theta\chi$ should be interpreted as the flow line of the vector field $e^{ih/\chi + \theta}$.  We say that $\eta_\theta$ is the flow line of $h$ with angle $\theta$.  If $h$ were a smooth function and $\theta_1 < \theta_2$, then it would be obvious that $\eta_{\theta_1}$ lies to the right of $\eta_{\theta_2}$.  Although non-trivial to prove, this is also true in the setting of the GFF \cite[Theorem~1.5]{MS_IMAG} and is depicted in Figure~\ref{fig::monotonicity}.  The case in which the flow lines start at different points is depicted in Figure~\ref{fig::different_starting_point}.

\begin{figure}[h!]
\begin{center}
\includegraphics[scale=0.85]{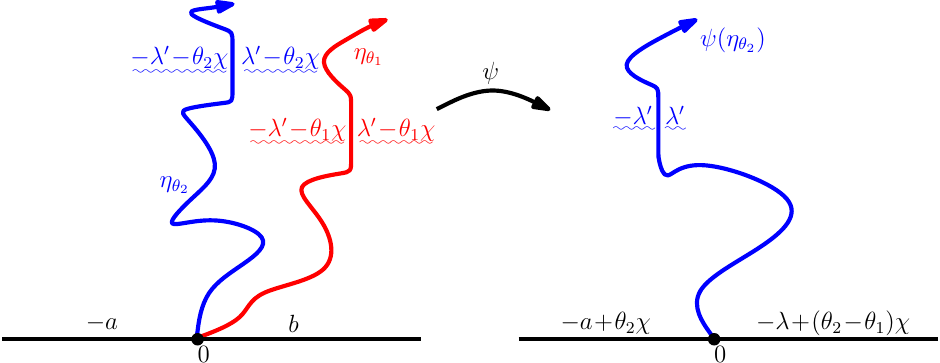}
\caption{\label{fig::monotonicity}  Suppose that $h$ is a GFF on $\h$ with the boundary data depicted above.  For each $\theta \in \R$, let $\eta_\theta$ be the flow line of the GFF $h+\theta \chi$.  This corresponds to setting the angle of $\eta_\theta$ to be $\theta$.  Just as if $h$ were a smooth function, if $\theta_1 < \theta_2$ then $\eta_{\theta_1}$ lies to the right of $\eta_{\theta_2}$.  The conditional law of $h$ given $\eta_{\theta_1}$ and $\eta_{\theta_2}$ is a GFF on $\h \setminus (\eta_{\theta_1} \cup \eta_{\theta_2})$ shown above.  By applying a conformal mapping and using the transformation rule, we can compute the conditional law of~$\eta_{\theta_1}$ given the realization of~$\eta_{\theta_2}$ and vice-versa.  That is, $\eta_{\theta_2}$ given $\eta_{\theta_1}$ is an $\SLE_\kappa((a-\theta_2 \chi)/\lambda -1; (\theta_2-\theta_1)\chi/\lambda-2)$ process independently in each of the connected components of $\h \setminus \eta_{\theta_1}$ which lie to the left of $\eta_{\theta_1}$, and moreover, $\eta_{\theta_1}$ given $\eta_{\theta_2}$ is independently an $\SLE_\kappa((\theta_2-\theta_1) \chi/\lambda -2;(b+\theta_1\chi)/\lambda-1)$ in each of the connected components of $\h \setminus \eta_{\theta_2}$ which lie to the right of $\eta_{\theta_2}$.}
\end{center}
\end{figure}

For $\theta_1 < \theta_2$, we can compute the conditional law of $\eta_{\theta_2}$ given $\eta_{\theta_1}$.  It is an $\SLE_\kappa((a-\theta_2 \chi)/\lambda-1; (\theta_2-\theta_1) \chi/\lambda -2)$ process independently in each connected component of $\h \setminus \eta_{\theta_1}$ which lies to the left of $\eta_{\theta_1}$.  Moreover, $\eta_{\theta_1}$ given $\eta_{\theta_2}$ is an $\SLE_\kappa((\theta_2-\theta_1) \chi/\lambda -2;(b+\theta_1\chi)/\lambda-1)$ in each of the connected components of $\h \setminus \eta_{\theta_2}$ which lie to the right of $\eta_{\theta_2}$.  This is depicted in Figure~\ref{fig::monotonicity}.  We can also couple together $\SLE_\kappa(\ul{\rho})$ processes starting at different points by considering the flow lines of $h$ initialized at different points.  This is depicted in Figure~\ref{fig::different_starting_point}.

\begin{figure}[h!]
\begin{center}
\includegraphics[scale=0.85]{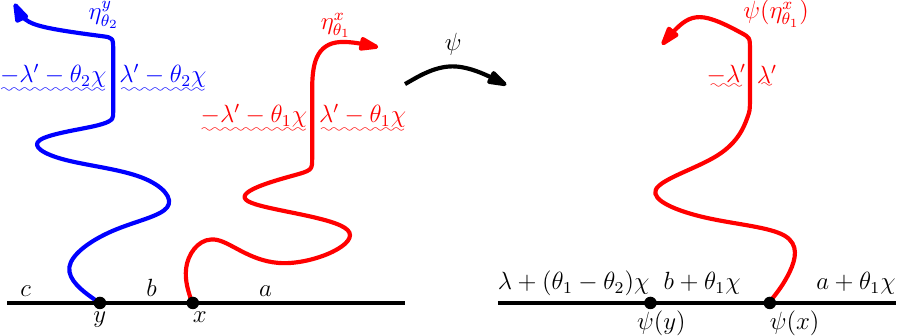}
\caption{\label{fig::different_starting_point}  Suppose that $h$ is a GFF on $\h$ with the boundary data depicted above.  For each $x,\theta \in \R$, let $\eta_\theta^x$ be the flow line of the GFF $h+ \theta \chi$ starting at $x$.  Assume $y < x$.  Just as in Figure~\ref{fig::monotonicity}, if $\theta_2 > \theta_1$, then $\eta_{\theta_2}^y$ lies to the left of $\eta_{\theta_1}^x$.  The conditional law of $h$ given $\eta_{\theta_1}^x, \eta_{\theta_2}^y$ is a GFF on $\h \setminus (\eta_{\theta_1}^x \cup \eta_{\theta_2}^y)$ shown above.  By applying a conformal mapping and using the transformation rule, we can compute the conditional law of either curve given the realization of the other.  For example, as depicted above, the conditional law of $\eta_{\theta_1}^x$ given $\eta_{\theta_2}^y$ is an $\SLE_\kappa( -(b+\theta_1 \chi)/\lambda-1,(b+\theta_2 \chi)/\lambda - 1 ; (a+ \theta_1 \chi)/\lambda-1)$ process.}
\end{center}
\end{figure}

Recall that $\kappa' = 16/\kappa \in (4,\infty)$.  We refer to $\SLE_{\kappa'}$ processes as \emph{counterflow} lines because of the particular way that they are related to the flow lines of the Gaussian free field.  As explained in \cite{MS_IMAG}, the counterflow line runs in the opposite direction of (or {\em counter to}) the so-called ``light cone'' of angle-varying flow lines whose angles stay within some range.  We can construct couplings of $\SLE_{\kappa}$ and $\SLE_{\kappa'}$ processes (flow lines and counterflow lines) within the same imaginary geometry.  This is depicted in Figure~\ref{fig::counterflowline}.  Just as in the setting of multiple flow lines, we can compute the conditional law of a flow line given the realization of a counterflow line within the same geometry.  This will be rather useful for us in Section~\ref{sec::bi_chordal} and is explained in Figure~\ref{fig::counterflowline}.

\begin{figure}[h!]
\begin{center}
\includegraphics[scale=0.85]{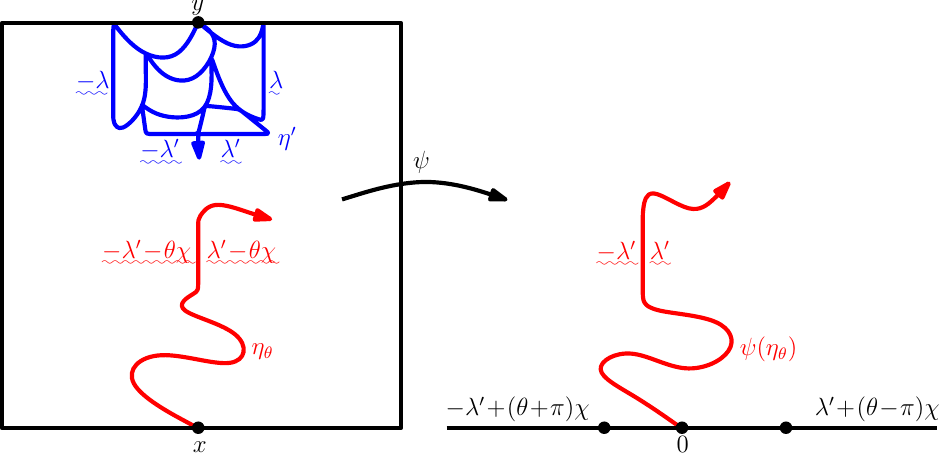}
\end{center}
\caption{\label{fig::counterflowline} We can construct $\SLE_\kappa$ flow lines and $\SLE_{\kappa'}$ counterflow lines within the same geometry when $\kappa' = 16/\kappa$.  This is depicted above for a single counterflow line $\eta'$ emanating from $y$ and a flow line $\eta_\theta$ with angle $\theta$ starting from $0$.  Also shown is the boundary data for $h$ in $D \setminus (\eta'([0,\tau']) \cup \eta_\theta([0,\tau]))$ where $\tau'$ and $\tau$ are stopping times for $\eta'$ and $\eta_\theta$ respectively (we intentionally did not specify the boundary data of $h$ on $\partial D$).  If $\theta = \tfrac{1}{\chi}(\lambda'-\lambda) = -\tfrac{\pi}{2}$ so that the boundary data on the right side of $\eta_\theta$ matches that on the right side of $\eta'$, then $\eta_\theta$ will almost surely hit and then ``merge'' into the right boundary of $\eta'$ --- this fact is known as $\SLE$ duality.  We can compute the conditional law of $\eta_\theta$ given $\eta'([0,\tau'])$ by conformally mapping the connected component of $D \setminus \eta'([0,\tau'])$ which contains $x$ to $\h$ and then applying Figure~\ref{fig::conditional_boundary_data}.  If $\eta_\theta \sim \SLE_\kappa(\ul{\rho}^L;\ul{\rho}^R)$, then $\eta_\theta$ conditional on $\eta'([0,\tau'])$ is an $\SLE_\kappa(\ul{\rho}^L,\tfrac{3}{2}(\tfrac{\kappa}{2}-2)-\tfrac{\theta \chi}{\lambda} - \ol{\rho}^L;\ul{\rho}^R,\tfrac{3}{2}(\tfrac{\kappa}{2}-2)+\tfrac{\theta \chi}{\lambda} - \ol{\rho}^R)$ process where $\ol{\rho}^q = \sum_i \rho^{i,q}$ with the extra force points located at the left and right sides of $\eta'([0,\tau']) \cap \partial D$.  In the special case that $\theta = -\tfrac{\pi}{2}$, then the law of $\eta_\theta$ given $\eta'([0,\tau'])$ is an $\SLE_\kappa(\ul{\rho}^L,\tfrac{\kappa}{2}-2 - \ol{\rho}^L;\ul{\rho}^R,\kappa-4 - \ol{\rho}^R)$ process.
}
\end{figure}

\begin{figure}[ht!]
\begin{center}
\includegraphics[scale=0.85]{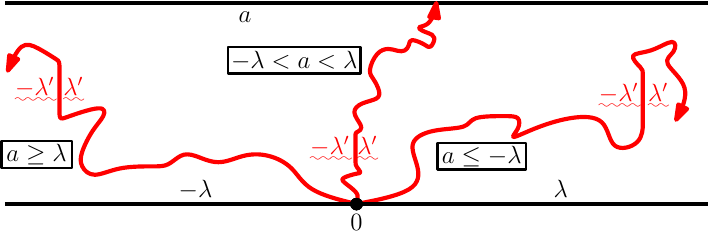}
\caption{\label{fig::hittingrange}  Suppose that $h$ is a GFF on the strip $\strip$ with the boundary data depicted above and let $\eta$ be the flow line of $h$ starting at $0$ (we do not specify the boundary data on the bottom of the strip $\stripbot$, but only assume that the flow line never hits the continuation threshold upon hitting $\stripbot$).  The interaction of $\eta$ with the upper boundary $\striptop$ of $\partial \strip$ depends on $a$, the boundary data of $h$ on $\striptop$.  Curves shown represent almost sure behaviors corresponding to the three different regimes of $a$ (indicated by the closed boxes).  The path hits $\striptop$ almost surely if and only if $a \in (-\lambda, \lambda)$.  When $a \geq \lambda$, it tends to $-\infty$ (left end of the strip) and when $a \leq - \lambda$ it tends to $+\infty$ (right end of the strip) without hitting $\striptop$.  If $\eta$ can hit the continuation threshold upon hitting some point on $\stripbot$, then $\eta$ only has the possibility of hitting $\striptop$ if $a \in (-\lambda,\lambda)$ (but does not necessarily do so); if $a \notin (-\lambda,\lambda)$ then $\eta$ almost surely does not hit $\striptop$.  Using the transformation rule~\eqref{eqn::ac_eq_rel}, we can extract from this the values of the boundary data for the boundary segments that $\eta$ can hit with other orientations.  We can also rephrase this in terms of the weights $\ul{\rho}$: an $\SLE_{\kappa}(\ul{\rho})$ process almost surely does not hit a boundary interval $(x^{i,R},x^{i+1,R})$ (resp.\ $(x^{i+1,L},x^i)$) if $\sum_{s=1}^i \rho^{s,R} \geq \tfrac{\kappa}{2}-2$ (resp.\ $\sum_{s=1}^i \rho^{s,L} \geq \tfrac{\kappa}{2}-2$).  See \cite[Lemma~5.2 and Remark~5.3]{MS_IMAG}.  These facts hold for all $\kappa > 0$.}
\end{center}
\end{figure}

\begin{figure}[h!]
\begin{center}
\includegraphics[scale=0.85]{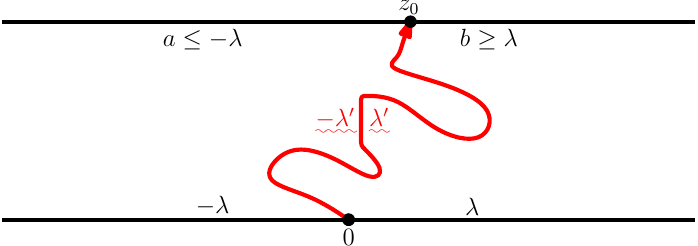}
\caption{\label{fig::hittingsinglepoint} (Continuation of Figure~\ref{fig::hittingrange})  Assume the boundary data on $\stripbot$ are such that the flow line $\eta$ of $h$ almost surely does not hit the continuation threshold on $\stripbot$.  Suppose the boundary data for $h$ on $\striptop$ is $a \leq -\lambda$ to the left of $z_0$ and $b \geq \lambda$ to the right of $z_0$.  Then $\eta$ almost surely terminates at $z_0$  without first hitting any other point on $\striptop$ \cite[Theorem~1.3]{MS_IMAG}.  If $a \leq -\lambda$ and $b \in (-\lambda+ \pi \chi, \lambda)$, then $\eta$ still almost surely reaches $z_0$ but it hits $\striptop$ to the right (and only to the right) of $z_0$ before it hits $z_0$.  Similarly, if we have $b \geq \lambda$ and $a \in (-\lambda, \lambda- \pi \chi)$, then $\eta$ hits $\striptop$ to the left (and only to the left) of $z_0$ before it hits $z_0$.  If $b \in (-\lambda+ \pi \chi, \lambda)$ and $a \in (-\lambda, \lambda- \pi \chi)$ then $\eta$ almost surely hits both sides of $z_0$ before reaching $z_0$.  If either $b \leq -\lambda+\pi\chi$ or $a \geq \lambda - \pi \chi$ then $\eta$ almost surely hits the continuation threshold before reaching $z_0$.
}
\end{center}
\end{figure}

It is also possible to determine which segments of the boundary a flow or counterflow line cannot hit.  This is described in terms of the boundary data of the field in Figure~\ref{fig::hittingrange} and Figure~\ref{fig::hittingsinglepoint} and proved in \cite[Lemma~5.2]{MS_IMAG} (this result gives the range of boundary data that $\eta$ cannot hit, contingent on the almost sure continuity of $\eta$; this, in turn, is given in \cite[Theorem~1.3]{MS_IMAG}).  This can be rephrased in terms of the weights $\ul{\rho}$: an $\SLE_{\kappa}(\ul{\rho})$ process almost surely does not hit a boundary interval $(x^{i,R},x^{i+1,R})$ (resp.\ $(x^{i+1,L},x^i)$) if $\sum_{s=1}^i \rho^{s,R} \geq \tfrac{\kappa}{2}-2$ (resp.\ $\sum_{s=1}^i \rho^{s,L} \geq \tfrac{\kappa}{2}-2$).  See \cite[Remark~5.3]{MS_IMAG}.

\subsection{Naive time-reversal}
\label{subsec::naive_time_reversal}
If $h$ were a smooth function and $\eta$ were a flow line of $e^{ih/\chi}$, then the time-reversal of $\eta$ would be a flow line of $e^{i(h/\chi + \pi)}$.  This turns out not to be the case for the flow lines of the Gaussian free field (which is not a smooth function).  To give an instructive explanation of what actually happens if we try to reverse direction naively in this way, consider the Gaussian free field on an infinite strip with boundary conditions $-a$ and $b$ as in Figure~\ref{fig::chicken}.  If $a,b \in (-\lambda', \lambda)$ then the results of \cite{MS_IMAG} (see Figure~\ref{fig::hittingrange}) imply that both a forward path from the bottom to the top of strip and an ``opposite direction'' path (sometimes called a {\em dual flow line}) from the top to the bottom are defined, and that both paths almost surely hit both sides of the strip infinitely often, as Figure~\ref{fig::chicken} illustrates.

However, the extent to which these two paths fail to coincide with one another is somewhat amusing.  By way of remark, we mention a few facts about the joint law of this pair of paths that follow from certain results of \cite{MS_IMAG}, which are explained in Figure~\ref{fig::hittingrange} and Figure~\ref{fig::hittingsinglepoint}).  Combined with Figure~\ref{fig::different_starting_point} and Figure~\ref{fig::monotonicity}, these describe when flow lines can intersect each other and at which segments and points they can exit given domains with positive probability.  If we condition on the dual flow line $\wt \eta$ up to any rational time $\wt \tau$ and on the flow line $\eta$ up until any rational time $\tau$ and the two paths run to these times do not intersect each other, then it almost certainly the case that, after time $\tau$, $\eta$ first intersects $\wt \eta([0,\wt\tau])$ at either the last place $\wt \eta$ hit the left boundary before time $\wt \tau$ or the last place $\wt \eta$ hit the right boundary before time $\wt \tau$.  An analogous result holds with the roles of $\eta$ and $\wt \eta$ reversed.

This implies in particular that the two paths almost surely do not intersect each other at {\em any} point in the interior of the strip.  It also implies that if they do intersect at a point $x = \eta(\tau) = \wt \eta(\wt \tau)$ on the boundary of the strip, then the next place $\eta$ hits $\wt \eta([0,\wt \tau])$ {\em after} time $\tau$ must be on the opposite side of the strip.  (If it collided with a point $y$ on same side before crossing the strip, and we restricted $\eta$ and $\wt \eta$ to the portions in between their hitting of $x$ and $y$, then the boundary intersection rules mentioned above would imply that each of these portions must be to the right of the other, a contradiction.)  This implies that the paths almost surely behave in the manner illustrated in Figure~\ref{fig::chicken}.  In particular, whenever the two paths intersect at a boundary point, they immediately cross each other and do not intersect anywhere else in a neighborhood of that point.

The fact that flow lines and dual flow lines almost surely do not intersect each other (except at isolated boundary points) will actually be useful to us.  Figure~\ref{fig::chicken} also illustrates another principle that will be useful in this paper, namely that in some circumstances a flow line avoids a boundary interval if and only if a dual flow line avoids another boundary interval.

\begin{figure}[h!]
\begin{center}
\includegraphics[scale=0.85]{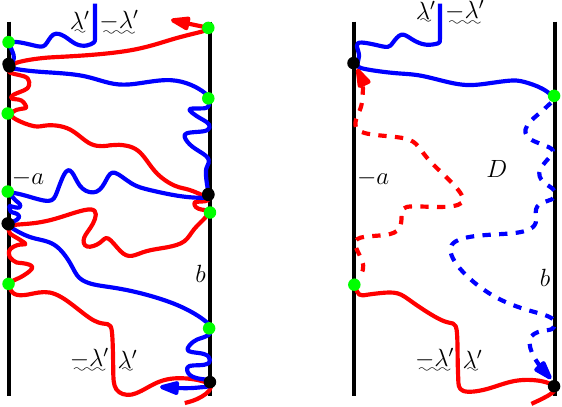}
\end{center}
\caption{\label{fig::chicken}
Suppose $a,b \in (-\lambda',\lambda')$ so that one can draw a red flow line and blue dual flow line directed at each other, starting from opposite ends of the strip.  Each path almost surely hits both sides of the strip infinitely often (the range of boundary values for which the path can hit both sides comes from combining Figure~\ref{fig::hittingrange} with~\eqref{eqn::ac_eq_rel}).  A green dot indicates the first time a path hits a given side of the strip after hitting the other side; a black dot indicates the last time.  The results of \cite{MS_IMAG} imply that the black dots for the blue and red paths coincide (see also Figures~\ref{fig::hittingrange}--\ref{fig::hittingsinglepoint}). In the right figure, we consider the red path stopped at a stopping time when it has just completed a crossing (a green dot) and the blue path stopped at an analogous forward stopping time (a green dot on the opposite side).  On the event that the red path and blue path up to these stopping times are disjoint, we denote by $D$ the component of the complement of this pair of paths that is incident to both paths.  Conditioned on $D$, there is a positive probability that the following three {\em equivalent} events will occur: the dotted lines (which are continuations of the red and blue flow lines) avoid each other, the red dotted line avoids the right boundary, and the blue dotted line avoids the left boundary. }
\end{figure}

\subsection{Idea for a Gaussian free field reversibility proof}

We are now going to provide some intuition behind our proof of the reversibility of $\SLE_\kappa(\rho_1;\rho_2)$ which comes from the imaginary geometry perspective of $\SLE$.  First we present a heuristic argument for the reversibility of ordinary $\SLE_\kappa$ that uses Gaussian free field machinery.  (This was discovered by Schramm and second author some years ago but never published.)

\begin{figure}[ht!]
\begin{center}
\subfigure[\label{fig::reversibility_intuition_interior}]{
\includegraphics[scale=0.85,page=1]{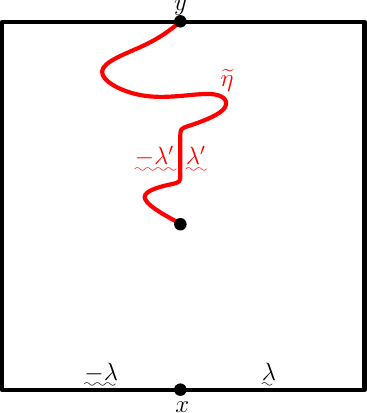}}
\hspace{0.05\textwidth}
\subfigure[\label{fig::reversibility_intuition_interior2}]{
\includegraphics[scale=0.85,page=2]{figures/imaginary_reversibility2.pdf}}
\end{center}
\caption{\label{fig::reversibility_intuition} Although the coupling of $\SLE_\kappa$ processes with the GFF is non-reversible, the theory nevertheless provides an intuitive explanation of the reversibility.  Suppose that $h$ is a GFF on $D$ whose boundary data is compatible with an ordinary $\SLE_\kappa$ process from $x$ to $y$ as above.  On the left side, we have drawn an interior flow line $\eta'$ of $h$ starting from $z$.  Conditionally on $\wt{\eta}$, $\eta$ is an $\SLE_\kappa(\kappa-4;\kappa-4)$ process in $D \setminus \wt{\eta}$ with force points at $y^-,y^+$, respectively.  If $\wt{\eta}$ is the terminal segment of $\eta$, then $\eta$ must first hit $\wt{\eta}$ at $z$ (see the right side).  This gives us the conditional law of $\eta$ given the realization of its time-reversal up to a reverse stopping time.  Once one shows that $\SLE_\kappa(\kappa-4;\kappa-4)$ processes conditioned not to hit the boundary (which corresponds to $\eta$ first exiting $D \setminus \eta'$ at $z$) are (appropriately interpreted) $\SLE_\kappa$ processes, the reversibility follows.
}
\end{figure}

Suppose that $D$ is a Jordan domain with $x,y \in \partial D$ distinct.  Let $h$ be a GFF on $D$ with boundary data as described in Figure~\ref{fig::reversibility_intuition_interior}, so that the flow line $\eta$ of $h$ starting from~$x$ is an $\SLE_\kappa$ process from $x$ to $y$.  Let $\wt{\eta}$ be any simple curve ending at~$y$ which does not hit $\partial D \setminus \{y\}$ and which starts at a point $z = \wt{\eta}(0) \in D$.  Formally, we can write the event $E$ that $\wt{\eta}$ is the terminal segment of $\eta$ as the intersection of the event $E_1$ that $\wt{\eta}$ is a flow line of $h$ starting at $z$ and the event $E_2$ that $\eta$ intersects and then merges with $\wt{\eta}$ precisely at $z$.  Although $E_1$ is a probability zero event, we expect that the conditional law of $h$ given $E_1$ is that of a GFF in $D \setminus \wt{\eta}$ with the boundary data depicted in Figure~\ref{fig::reversibility_intuition_interior}.  This would imply that the conditional law of $\eta$ given $E_1$ is that of an $\SLE_\kappa(\kappa-4;\kappa-4)$ in $D \setminus \wt{\eta}$ with force points on the left and right sides of $y$.  Therefore $E_2$ given $E_1$ is the event that an $\SLE_\kappa(\kappa-4;\kappa-4)$ process first exits $D \setminus \wt{\eta}$ at $z$.

The event $E_2$ has probability zero; however, we will see in Section~\ref{sec::conditioned_not_to_hit} that, with an appropriate interpretation, one can still ``condition'' on this event so that the conditional law of the process is that of an $\SLE_\kappa$ in $D \setminus \wt{\eta}$ from $x$ to $z$.  That is, the law of $\eta$ conditional on $\wt{\eta}$ being its terminal segment is an $\SLE_\kappa$ from $x$ to $z$ in $D \setminus \wt{\eta}$.  This implies the domain Markov property for the time-reversal of $\SLE_\kappa$, which implies that the time-reversal is itself an $\SLE_\kappa$.

This argument also motivates our approach to the time-reversal symmetry of $\SLE_\kappa(\rho_1;\rho_2)$.
At least heuristically, we can also consider the same setup with more general boundary data to obtain that for $\eta \sim \SLE_\kappa(\ul{\rho}^L;\ul{\rho}^R)$, the law of $\eta$, given the realization of its time-reversal $\wt{\eta}$ up to a reverse stopping time $\tau_R$, is that of an $\SLE_\kappa(\ul{\rho}^L,\kappa-4-\ol{\rho}^L;\ul{\rho}^R,\kappa-4-\ol{\rho}^R)$ process in the connected component of $D \setminus \wt{\eta}([0,\tau_R])$ containing $x$ from $x$ to $\wt{\eta}(\tau_R)$ where $\ol{\rho}^q = \sum_i \rho^{i,q}$ for $q \in \{L,R\}$, conditioned to hit $\wt{\eta}([0,\tau_R])$ precisely at its tip.

\section{$\SLE_\kappa(\rho)$: conformal Markov characterization}
\label{sec::conformal_markov}

The purpose of this section is to prove Theorem~\ref{thm::conformal_markov}, that conformal invariance (Definition~\ref{def::conf_invariance}) and the domain Markov property (Definition~\ref{def::domain_markov}) with one extra marked point single out the single-force-point $\SLE_\kappa(\rho)$ processes.  This is a generalization of Schramm's conformal Markov characterization of ordinary $\SLE$ \cite{S0}.

Before we proceed to the proof, we pause momentarily to remind the reader of a few basic facts regarding Bessel processes \cite{RY04}.  We recall that a Bessel process ${\rm BES}^\delta$ of dimension $\delta$ is a solution to the SDE
\begin{equation}
\label{eqn::bessel}
 dX(t) = \frac{\delta-1}{2 X(t)} dt + dB(t),\quad X(0) = x_0 \geq 0
\end{equation}
where $B$ is a standard Brownian motion.  Bessel processes are typically defined by first starting with a solution to the square Bessel equation ${\rm BESQ}^\delta$
\begin{equation}
\label{eqn::square_bessel}
 dZ(t) = 2 \sqrt{Z(t)} dB(t) + \delta t,\quad Z(0) = z_0 \geq 0,
\end{equation}
which exists for all $t \geq 0$ provided $\delta > 0$ \cite{RY04}, and then taking the square root.   In this case, $\sqrt{Z}$ solves~\eqref{eqn::bessel} at times when $Z > 0$ and is instantaneously reflecting at $0$ (almost surely spends zero Lebesgue measure time at $0$).  When $\delta > 1$, $\sqrt{Z}$ solves~\eqref{eqn::bessel} for all times $t$.  We make sense of this when $Z = 0$ by taking the integral version of~\eqref{eqn::bessel}.  The infinitesimal generator $\CA$ of $Z$ is given by
\begin{equation}
\label{eqn::bessel_generator}
 \CA f(x) = \delta f'(x) + x f''(x).
\end{equation}
Therefore any Markov process whose infinitesimal generator $\wt{\CA}$ takes the form
\begin{equation}
\label{eqn::bessel_generator_general}
 \wt{\CA} f(x) = a_0 f'(x) + a_1 x f''(x)
\end{equation}
for $a_0, a_1 > 0$ is a ${\rm BESQ}^\delta$ process for $\delta = a_0/\sqrt{a_1} > 0$ multiplied by $\sqrt{a_1}$.

The driving function $W$ for an $\SLE_\kappa(\rho)$ process with one force point of weight $\rho > -2$ can be constructed by first taking a solution $X$ of~\eqref{eqn::bessel} with
\begin{equation}
\label{eqn::dimension_kappa_rho}
\delta = 1+ \frac{2(\rho+2)}{\kappa} > 1,
\end{equation}
then letting $O(t) = -2 \kappa^{-1/2} \int_0^t \tfrac{2}{X(s)} ds,$
 and finally setting $W = O +\sqrt{\kappa} X$ (see, e.g.\ \cite[Section~3.3]{SHE_CLE}).  The main ingredient to the proof of Theorem~\ref{thm::conformal_markov} is the following characterization of Bessel processes, which is essentially a consequence of Lamperti's characterization \cite{LAMP72} of time homogeneous self-similar Markov processes taking values in $[0,\infty)$.

\begin{lemma}
\label{lem::bessel_characterization}
Suppose that $V$ is a time homogeneous Markov process with continuous sample paths taking values in $[0,\infty)$ which is instantaneously reflecting at $0$ and satisfies Brownian scaling.  Then $V$ is a positive multiple of a ${\rm BES}^\delta$ process for some $\delta > 0$.
\end{lemma}
\begin{proof}
Note that $V^2$ is also a time homogeneous Markov process with continuous sample paths taking values in $[0,\infty)$ which is instantaneously reflecting at $0$.  Moreover, $V^2$ satisfies the scaling relation $(t \mapsto \alpha^{-1} V^2(\alpha t))\stackrel{d}{=} (t \mapsto V^2(t))$.  Consequently, it follows from \cite[Theorem~5.1]{LAMP72} that the generator $\CA$ of $V^2$ takes the form
\[ \CA f(x) = a_0 f'(x) + a_1 x f''(x),\quad x > 0\]
for constants $a_0,a_1 > 0$ (Lamperti's characterization allows for $a_0 \leq 0$; this, however, requires that $0$ is an absorbing state, which is not the case in our setting since our process is defined for all times and is instantaneously reflecting at $0$).  This is the generator for a positive multiple $a$ of a ${\rm BESQ}^\delta$ process, some $\delta > 0$ (recall~\eqref{eqn::bessel_generator_general} and the surrounding text).  That is, $a V^2$ evolves as a ${\rm BESQ}^\delta$ process when it is away from~$0$.  Since $V^2$ is instantaneously reflecting, it follows that $a V^2(t)$ solves the ${\rm BESQ}^\delta$ equation for all time.  Consequently, $V$ is a positive multiple of a ${\rm BES}^\delta$ process.
\end{proof}

Before we proceed to the proof of Theorem~\ref{thm::conformal_markov}, we need to collect two technical lemmas.

\begin{lemma}
\label{lem::leb_zero}
Suppose that $\eta$ is a continuous path in $\ol{\h}$ from $0$ to $\infty$ that admits a continuous Loewner driving function.  Assume that $\eta$ is parameterized by capacity and let $J = \{ t \geq 0 : \eta(t) \in \R\}$.  Then $J$ has Lebesgue measure zero.
\end{lemma}
\begin{proof}
This is \cite[Lemma~2.5]{MS_IMAG}.
\end{proof}

\begin{lemma}
\label{lem::x_differential}
Suppose that $\eta$ is a continuous path in $\ol{\h}$ from $0$ to $\infty$ that admits a continuous Loewner driving function $W$.  Let $(g_t)$ be the corresponding family of conformal maps, parameterized by capacity.  For each $t$, let $X(t)$ be the right most point of $g_t(\eta([0,t]))$ in $\R$.  If the Lebesgue measure of $\eta([0,\infty]) \cap \R$ is zero, then $X$ solves the integral equation
\[ X(t) = \int_0^t \frac{2}{X(s) - W(s)} ds,\quad X(0) = 0^+.\]
\end{lemma}
\begin{remark}
\label{rem::x_differential}
This always holds for $\SLE_\kappa(\ul{\rho})$ paths, even when $\kappa > 4$ and $\ul{\rho}$ is such that the path is boundary filling.  However, there are examples of continuous curves which intersect the boundary with positive Lebesgue measure so that the above does not hold.
\end{remark}
\begin{proof}[Proof of Lemma~\ref{lem::x_differential}]
\footnote{We thank Wendelin Werner for suggesting the following simple proof.} Fix $\epsilon > 0$.  We define iteratively the process $X^\epsilon$ as follows.  We take $X^\epsilon(0) = \epsilon$ and let $X^\epsilon(t)$ evolve according to the Loewner flow which is driven by $W$ up until time $\tau_1 = \inf\{t \geq 0 : X^\epsilon(t) = W(t)\}$.  We then take $X^\epsilon(\tau_1+) = X^\epsilon(\tau_1) + \epsilon$.  Assuming that $X^\epsilon$ has been defined up until time $\tau_k$, we take $X^\epsilon(\tau_k+) = X^\epsilon(\tau_k)+\epsilon$ and then take $\tau_{k+1} = \inf \{ t > \tau_k : X^\epsilon(t) = W(t)\}$.  Then we clearly have that
\[ X(t) \leq X^\epsilon(t) \leq X(t)+\epsilon \quad\text{for all}\quad t.\]
We then let $N^\epsilon(t)$ be the number of $\epsilon$-jumps made by $X^\epsilon$ before time $t$.  Then we have that the Lebesgue measure of the $\epsilon$-neighborhood of $\eta([0,t]) \cap \R$ is at least $\epsilon N^\epsilon(t)$ as each jump of $X^\epsilon$ corresponds to an interval of length larger than $\epsilon$ on the real line that contains a point of $\eta \cap \R$ and, for a given $\epsilon$, these intervals are disjoint.  Hence, since $\eta([0,t]) \cap \R$ is both closed and has zero Lebesgue measure, it follows that $\epsilon N^\epsilon(t) \to 0$ as $\epsilon \to 0$.  On the other hand, we have that
\[ X^\epsilon(t) = \epsilon N^\epsilon(t) + \int_0^t \frac{2}{X^\epsilon(s) - W(s)} ds.\]
Thus sending $\epsilon \to 0$, we obtain
\[ X(t) = \int_0^t \frac{2}{X(s) - W(s)} ds,\]
as desired. (The convergence of the integral follows from the monotone convergence theorem as $X(s) - W(s) \leq X^\epsilon(s) - W(s)$ for all $s$.)
\end{proof}

Lemma~\ref{lem::x_differential} allows us to make sense of $g_t(x)$ for any $x \in \R$ for all $t \geq 0$, even after being swallowed by $\eta$, via the Loewner evolution.  We can now complete the proof of Theorem~\ref{thm::conformal_markov}.

\begin{proof}[Proof of Theorem~\ref{thm::conformal_markov}]
For $x \geq 0^+$, let $c(x) = (\h,0,\infty;x)$.  Let $\eta \sim \p_{c(x)}$ and, for each $t > 0$, let $c_t(x) = (\h_t,\eta(t),\infty;x_t)$ where we recall the definitions of $\h_t$ and $x_t$ from Definition~\ref{def::domain_markov}.  We are now going to argue that $\eta$ almost surely has a continuous Loewner driving function viewed as a path in $\ol{\h}$ from $0$ to $\infty$.  To see this, we will check the criteria of \cite[Proposition~6.12]{MS_IMAG}.  It is implicit in the domain Markov property that, almost surely, for every $t > 0$ we have that $\eta((t,\infty))$ is contained in the closure of the unbounded connected component of $\h \setminus \eta((0,t))$.  The hypothesis that $\eta \cap \R$ has zero Lebesgue measure combined with the domain Markov property and conformal invariance implies that, almost surely, for every $t > 0$ we have $\eta|_{[t,\infty)}^{-1}(\R \cup \eta([0,t]))$ has empty interior.  Therefore $\eta$ almost surely has a continuous Loewner driving function.  Let $(g_t)$ be the family of Loewner maps associated with $\eta$, parameterized by capacity.

Let $X(t) = g_t(x)$ (which makes sense for all $t \geq 0$ by Lemma~\ref{lem::x_differential}) and $W$ be the Loewner driving function for $\eta$.  Finally, let $V(t) = X(t) - W(t)$.  Then $V(t)$ takes values in $[0,\infty)$.  Since $W$ and $X$ are continuous, so is $V$.  We note that $\eta(t) \in \R$ precisely when $V(t) = 0$.  Therefore Lemma~\ref{lem::leb_zero} implies that $V$ is instantaneously reflecting at $0$.  We are going to complete the proof of the theorem by showing that $V$ is a Bessel process by checking the criteria of Lemma~\ref{lem::bessel_characterization}.

We are first going to show that $V$ is a time homogeneous Markov process.  Let $\CF_t = \sigma(\eta(s) : s \leq t)$ be the filtration generated by $\eta$.  For $t \in \R$, let $\theta_t u(s) = u(s+t)$ be the usual shift operator.  For any $t \geq 0$, the domain Markov property implies
\begin{align}
\label{eqn::domain_markov_app}
  \p_{c(V(0))}[ \theta_t V \in A \giv \CF_t]
= \p_{c_t(V(0))}[V \in A].
\end{align}
Note that we used $V(0) = X(0)$.  By applying the conformal map $g_t - W(t)$, conformal invariance implies that the expression in~\eqref{eqn::domain_markov_app} is equal to $\p_{c(V(t))}[V \in A]$.  Therefore $V$ is a time homogeneous Markov process.

We are now going to show in the special case $x = 0^+$ that $V$ satisfies Brownian scaling.  The reason we choose to work with this particular case is that $\p_{c(0^+)}$ is scale invariant by conformal invariance.  We will first show that $W$ satisfies Brownian scaling.  For $\alpha > 0$, we note that
\begin{align*}
   \partial_t g_{\alpha t}(z)
&= \frac{2 \alpha}{g_t(\alpha t) - W(\alpha t)}.
\end{align*}
Dividing both sides by $\alpha^{1/2}$ and rearranging, we see that
\begin{align*}
   \partial_t \big( \alpha^{-1/2} g_{\alpha t}(z) \big)
&= \frac{2}{ \alpha^{-1/2} g_{\alpha t}(z) - \alpha^{-1/2} W(\alpha t)}.
\end{align*}
The scale invariance of $\p_{c(0^+)}$ and the half-plane capacity scaling relation $\hcap(\alpha A) = \alpha^2 \hcap(A)$ (this appears, for example, just after \cite[Definition~3.35]{LAW05}) imply that $\big(t \mapsto g_{\alpha t}(z)\big) \stackrel{d}{=} \big(t \mapsto \alpha^{1/2} g_t(\alpha^{-1/2}z)\big)$.  Consequently, we have that $\big(t \mapsto \alpha^{-1/2} W(\alpha t) \big) \stackrel{d}{=} \big( t \mapsto W(t) \big)$.  That is, $W$ satisfies Brownian scaling.  Note that
\[ X(\alpha t) = \int_0^{\alpha t} \frac{2}{X(s) - W(s)} ds,\quad X(0) = 0^+\]
(the integral representation of $X$ is justified by Lemma~\ref{lem::x_differential}).  Making the change of variables $\alpha u = s$, we see that
\[ X(\alpha t) = \int_0^t \frac{2 \alpha}{X(\alpha u) - W(\alpha u)} du,\quad X(0) = 0^+.\]
Thus multiplying both sides by $\alpha^{-1/2}$, we see that $X$ satisfies Brownian scaling jointly with $W$.  Therefore $V$ satisfies Brownian scaling.

Lemma~\ref{lem::bessel_characterization} therefore implies that $V$ is a positive multiple of a ${\rm BES}^\delta$ process.  Moreover, we must have that $\delta > 1$ since $\delta \in (0,1]$ would imply that $X(t) = \int_0^t \tfrac{2}{V(s)} ds = \infty$ almost surely \cite{RY04}.  Therefore there exists $\kappa > 0$ such that
\[ dV(t) = \frac{\sqrt{\kappa}(\delta-1)}{2 V(t)} dt + \sqrt{\kappa} dB(t)\]
where $B$ is a standard Brownian motion.  Combining everything, we see that $(X,W)$ solves the SDE
\begin{align*}
 dW(t) &= d(X(t) - V(t)) = \frac{\sqrt{\kappa} (\delta-1)/2 - 2}{W(t) - X(t)} dt - \sqrt{\kappa} dB(t)\\
 dX(t) &= \frac{2}{X(t) - W(t)} dt,\quad X(0) = 0.
\end{align*}
This exactly says that $\p_{c(0^+)}$ is an $\SLE_\kappa(\rho)$ process with a single force point at $0^+$.

We will now explain how the case $x > 0$ follows from the case $x=0^+$.  Let $\tau = \inf\{t > 0 : V(t) = x\}$.  By the domain Markov property, the law of $\p_{c_\tau(0^+)}$ is the conformal image of $\p_{c(x)}$ under the conformal map $\h \to \h_\tau$ which fixes $\infty$, sends $0$ to $\eta(\tau)$, and sends $x$ to $(0^+)_\tau$ (recall Definition~\ref{def::domain_markov}).  On the other hand, we also know that $\p_{c_\tau(0^+)}$ is the law of an $\SLE_\kappa(\rho)$ process in $\h_\tau$ from $\eta(\tau)$ to $\infty$ with a single force point of weight $\rho$ at $(0^+)_\tau$.  Therefore $\p_{c(x)}$ is an $\SLE_\kappa(\rho)$ process in $\h$ from $0$ to $\infty$ with a single force point at $x$.
\end{proof}

\section{Bi-chordal $\SLE_\kappa$ processes}
\label{sec::bi_chordal}

\begin{figure}[ht!]
\begin{center}
\includegraphics[scale=0.85]{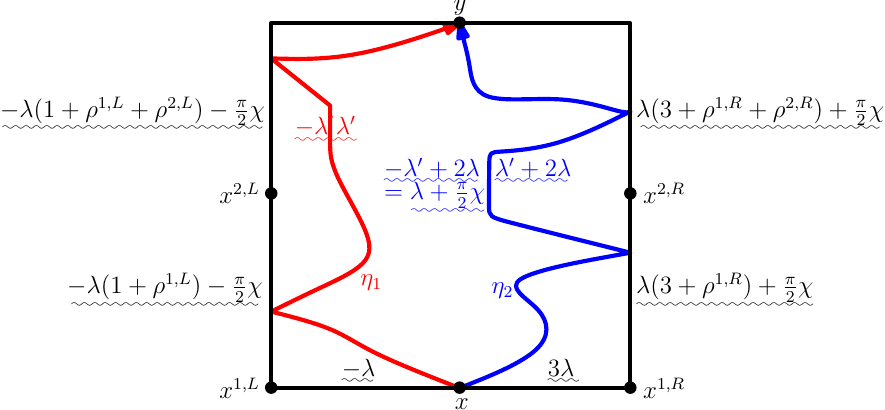}
\end{center}
\caption{\label{fig::bi_chordal}  Suppose that $D$ is a Jordan domain in $\C$.  Fix weights $\ul{\rho}^L,\ul{\rho}^R$ with $\sum_{i=1}^j \rho^{i,q} > -2$ for all $1 \leq j \leq |\ul{\rho}^q|$ and $q \in \{L,R\}$.  Let $h$ be a GFF on $D$ whose boundary data is depicted above (with the obvious generalization from $|\ul{\rho}^q| = 2$).  Let $\eta_1$ be the flow line of $h$ starting at $x$ and $\eta_2$ the flow line of $h-2\lambda$ starting at $x$.  Then $\eta_1 \sim \SLE_\kappa(\ul{\rho}^L;2+\ul{\rho}^R)$ (where $2+\ul{\rho} = (2+\rho_1,\rho_2,\ldots,\rho_n)$) and $\eta_2 \sim \SLE_\kappa(2+\ul{\rho}^L;\ul{\rho}^R)$.  The conditional law of $\eta_1$ given $\eta_2$ is an $\SLE_\kappa(\ul{\rho}^L)$ in the left connected component of $D \setminus \eta_2$ and the conditional law of $\eta_2$ given $\eta_1$ is an $\SLE_\kappa(\ul{\rho}^R)$ in the right connected component of $D \setminus \eta_1$.  We prove in Theorem~\ref{thm::bi_chordal} that the joint law of $(\eta_1,\eta_2)$ is characterized by these conditional laws.}
\end{figure}

Suppose that $D$ is a Jordan domain in $\C$.  Fix $x,y \in \partial D$ distinct.  Let $h$ be an instance of the GFF on $D$ with some boundary conditions.  We would like to consider two flow lines $\eta_1$ and $\eta_2$ from $x$ to $y$ with respective angles $\theta_1 > \theta_2$. We can change coordinates (using the coordinate change of Figure~\ref{fig::coordinatechange}) to the setting where $D$ is the upper half plane, with $x$ mapping to $0$ and $y$ to $\infty$.  Let us assume that after such a mapping, the boundary conditions of the GFF become piecewise constant, and that these values are less than $\lambda - \theta_1 \chi$ on $(-\infty, 0)$ and more than $-\lambda - \theta_2 \chi$ on $(0, \infty)$.  This ensures that each $\eta_i$ is some kind of $\SLE_\kappa(\ul\rho)$ process that can almost surely be continued all the way from $x$ to $y$.

In this setting, we also know (from the theory developed in \cite{MS_IMAG}) that if we condition on one of the $\eta_i$, then the conditional law of the other is given by a particular $\SLE_\kappa(\ul\rho)$ process (recall Figure~\ref{fig::monotonicity} and Figure~\ref{fig::different_starting_point}).  Thus, we have a recipe that tells us how to resample either one of the paths (leaving the other fixed) in a way that preserves the joint law of the pair of paths.  A natural question is whether the invariance of the joint law under this one-path-at-a-time resampling completely characterizes the joint law.  The purpose of this section is to show that indeed it does.

The most straightforward way to approach this would be to argue that if we start with any pair of paths $(\eta_1,\eta_2)$, and we repeatedly alternate between resampling $\eta_1$ and resampling $\eta_2$, then the law of the pair of paths ``mixes'', i.e., converges to some unique invariant measure.  The problem is that we do not actually expect a statement this strong to be true, mainly because we do not expect that the behavior in a small neighborhood of $x$ (or in a small neighborhood of $y$) can be completely mixed in finitely many steps.  We will get around this by first using a Gaussian free field trick that allows us to put a little bit of distance between the endpoints of the two paths, and second arguing that this little distance does not affect the macroscopic behavior of the pair of paths.

Before we proceed to prove this, let us restate the theorem formally in terms of $\SLE_\kappa(\underline{\rho})$ processes.  Let $\ul{\rho}^L,\ul{\rho}^R$ be vectors of weights with $\sum_{i=1}^j \rho^{i,q} > -2$ for all $1 \leq j \leq |\ul{\rho}^q|$ and $q \in \{L,R\}$.  Let $\eta_1$ be an $\SLE_\kappa(\ul{\rho}^L;2+\ul{\rho}^R)$ in $D$ from $x$ to $y$ (here and hereafter, we will write $2+\ul{\rho}$ for $(2+\rho_1,\rho_2,\ldots,\rho_n)$).  Conditional on $\eta_1$, let $\eta_2$ be an $\SLE_\kappa(\ul{\rho}^R)$ in the right connected component of $D \setminus \eta_1$ from $x$ to $y$.  Then $\eta_2$ is an $\SLE_\kappa(2+\ul{\rho}^L;\ul{\rho}^R)$ from $x$ to $y$ in $D$ and the conditional law of $\eta_1$ given $\eta_2$ is an $\SLE_\kappa(\ul{\rho}^L)$ in the left connected component of $D \setminus \eta_2$ from $x$ to $y$ (see Section~\ref{subsec::imaginary} and Figure~\ref{fig::bi_chordal}).  We will show that if $(\wt{\eta}_1,\wt{\eta}_2)$ is any other pair of non-intersecting simple random curves in $D$ in which $\wt{\eta}_1$ lies to the left of $\wt{\eta}_2$ and such that $\wt{\eta}_i$ connects $x$ to $y$ with the conditional law of $\wt{\eta}_i$ given $\wt{\eta}_j$ equal to the conditional law of $\eta_i$ given $\eta_j$ for $i,j \in \{1,2\}$ distinct, then $(\wt{\eta}_1,\wt{\eta}_2) \stackrel{d}{=} (\eta_1,\eta_2)$.  We call this configuration a ``bi-chordal $\SLE_\kappa(\ul{\rho})$'' process.  This characterization is the main ingredient in our derivation of the reversibility of $\SLE_\kappa(\rho_1; \rho_2)$ from the reversibility of single-force-point $\SLE_\kappa(\rho)$ processes.  The following theorem states this result.

\begin{theorem}
\label{thm::bi_chordal}
Suppose that $D \subseteq \C$ is a Jordan domain.  Fix weights $\ul{\rho}^L,\ul{\rho}^R$ with $\sum_{i=1}^j \rho^{i,q} > -2$ for $1 \leq j \leq |\ul{\rho}^q|$ and $q \in \{L,R\}$ and fix $x,y \in \partial D$ distinct.  There exists a unique measure on pairs of simple curves $(\eta_1,\eta_2)$ --- each connecting $x$ to $y$ in $D$ and intersecting each other only at these endpoints --- such that the following is true:  The law of $\eta_1$ given $\eta_2$ is an $\SLE_\kappa(\ul{\rho}^L)$ in the left connected component of $D \setminus \eta_2$, with the force points corresponding to $\ul{\rho}^L$ located on the clockwise arc from $x$ to $y$; and the law of $\eta_2$ given $\eta_1$ is an $\SLE_\kappa(\ul{\rho}^R)$ in the right connected component of $D \setminus \eta_1$, with force points corresponding to $\ul{\rho}^R$ located on the counterclockwise arc from $x$ to $y$.
\end{theorem}

The proof of Theorem~\ref{thm::bi_chordal} consists of two steps.  First, we will reduce to the case that the initial and terminal points of $\eta_1$ and $\eta_2$ are almost surely distinct as follows.  By $\SLE$ duality (see Figure~\ref{fig::counterflowline}), conditionally on $\eta_2$, there exists an $\SLE_{16/\kappa}(\ul{\rho}')$ process $\eta_1'$ running from $y$ to $x$ whose right boundary is almost surely equal to $\eta_1$.  We condition on the realization of an initial segment of $\eta_1$ as well as on an initial segment of $\eta_1'$.  Then $\eta_1$ will merge into the latter before hitting $y$, so that this conditioning has an effect similar to conditioning on the terminal segment of $\eta_1$.  Upon applying a conformal mapping, this in turn allows us to assume that $\eta_1$ connects $x_1$ to $(y_1,y_2)$ and that $\eta_2$ connects $x_2$ to $y_2$ where $x_1,x_2,y_2,y_1 \in \partial D$ are distinct and given in counterclockwise order.  Next, we will show that the chain which transitions from a configuration $(\gamma_1,\gamma_2)$ of non-intersecting simple paths where $\gamma_1$ connects $x_1$ to $(y_1,y_2)$ and $\gamma_2$ connects $x_2$ to $y_2$ by picking $i \in \{1,2\}$ uniformly and then resamples $\gamma_i$ according to the conditional law of $\eta_i$ given $\eta_j$ for $j=3-i$ has a unique stationary distribution.  The result then follows by noting that the law of $\eta_1$ is continuous in the realization of its initial segment and terminal dual segment \cite[Section~2]{MS_IMAG}.  Before we proceed to the proof, we establish the following technical lemma:

\begin{lemma}
\label{lem::coupling}
Suppose that $D$ is a Jordan domain and fix $x,y \in \partial D$ distinct.  Fix a vector of weights $\ul{\rho}^L$ and let $\eta$ be an $\SLE_\kappa(\ul{\rho}^L)$ process in $D$ from $x$ to $y$ where the force points $\ul{x}^L$ associated with the weights $\ul{\rho}^L$ are located on the clockwise arc of $\partial D$ from $x$ to $y$.  Let $\wt{D}$ be another Jordan domain with $x,y \in \partial \wt{D}$.  Assume that $\partial \wt{D}$ agrees with $\partial D$ in neighborhoods of both $x$ and $y$, as well as along the clockwise arc from $x$ to $y$.  If $\wt{\eta}$ is an $\SLE_\kappa(\ul{\rho}^L)$ process in $\wt{D}$ from $x$ to $y$ with the same force points $\ul{x}^L$ as $\eta$, then there exists a coupling $(\eta,\wt{\eta})$ such that $\p[\eta = \wt{\eta}] > 0$.
\end{lemma}
\begin{proof}
Let $h$ (resp.\ $\wt{h}$) be a GFF on $D$ (resp.\ $\wt{D}$) whose boundary data is such that we can couple $h$ (resp.\ $\wt{h}$) with $\eta$ (resp.\ $\wt{\eta}$) so that $\eta$ (resp.\ $\wt{\eta}$) is the flow line of $h$ (resp.\ $\wt{h}$) from $x$ to $y$ as in \cite[Theorem~1.1]{MS_IMAG}.  Let $x''$ (resp.\ $y''$) be in the counterclockwise arc of both $\partial D$ and $\partial \wt{D}$ from $x$ to $y$ such that the counterclockwise arcs of $\partial D$ and $\partial \wt{D}$ from $x$ to $x''$ and from $y''$ to $y$ agree.  Fix $x'$ (resp.\ $y'$) in the counterclockwise arc of $\partial D$ from $x$ to $x''$ (resp.\ $y$ to $y''$) distinct from $x$ and $x''$ (resp.\ $y$ and $y''$).  Assume further that $x',x'',y',y''$ are at a positive distance from where $\partial D$ and $\partial \wt{D}$ disagree.  Let $U \subseteq D \cap \wt{D}$ be a Jordan domain such that the clockwise arc of $\partial D$ from $x'$ to $y'$ is contained in $\partial U$ and such that $U$ has a positive distance from where $\partial D$ and $\partial \wt{D}$ differ.  Part~2 of \cite[Proposition~3.4]{MS_IMAG} implies that the laws of $h$ and $\wt{h}$, both restricted to $U$, are mutually absolutely continuous.  By \cite[Theorem~1.2]{MS_IMAG}, $\eta$ (resp.\ $\wt{\eta}$) up until first exiting $U$ is almost surely determined by $h$ (resp.\ $\wt{h}$) restricted to $U$.  Therefore the laws of $\eta$ and $\wt{\eta}$ stopped upon first exiting $U$ are mutually absolutely continuous.

To finish proving the lemma, it suffices to show that $\eta$ has a positive chance of staying in $U$.  To see this, suppose that $\wh{D} \subseteq D \cap \wt{D}$ is a Jordan domain with $x'',y'' \in \partial \wh{D}$ whose boundary agrees with that of $D$ and $\wt{D}$ along the clockwise arc from $x''$ to $y''$.  Assume also that $U \subseteq \wh{D}$ and that the clockwise arc of $\partial U$ from $x'$ to $y'$ has a positive distance from the clockwise arc of $\partial \wh{D}$ from $x''$ to $y''$.  Suppose that $\wh{h}$ is a GFF on $\wh{D}$ with the same boundary data as $h$ and $\wt{h}$ on the clockwise arc from $x'$ to $y'$ and so that its flow line $\wh{\eta}$ from $x$ to $y$ has the law of an $\SLE_\kappa(\ul{\rho}^L)$ process with force points at $\ul{x}^L$.  Then $\wh{\eta}$ almost surely does not hit the counteclockwise segment of $\partial \wh{D}$ from $x$ to $y$, hence almost surely stays in $D \cap \wt{D}$.  Applying the same reasoning as above, we have that the law of $\wh{\eta}$ stopped upon exiting $U$ is mutually absolutely continuous with respect to the laws of $\eta$ and $\wt{\eta}$ stopped upon exiting $U$.  Since $\wh{\eta}$ is almost surely continuous and almost surely does not hit the counterclockwise segment $L$ of $\partial \wh{D}$ from $x'$ to $y'$, it has a positive chance of not getting within distance $\delta > 0$ of $L$ provided we make $\delta > 0$ small enough.  Thus (by possibly expanding $U$) it follows that the same is true for $\eta$ and $\wt{\eta}$, which implies that they both have a positive chance of staying in $U$.
\end{proof}

We remark that we have not made any hypotheses on the weights $\ul{\rho}^L$ in the statement of Lemma~\ref{lem::coupling}.  The important assumption is that $\eta$ (resp.\ $\wt{\eta}$) does not have force points on the segment of $\partial D$ (resp.\ $\partial \wt{D}$) which connects $x$ to $y$ in the counterclockwise direction because this ensures that $\eta$ (resp.\ $\wt{\eta}$) almost surely does not intersect this part of $\partial D$ (resp.\ $\partial \wt{D}$).  The same result is true if we add force points on the arc of $\partial D$ (resp.\ $\partial \wt{D}$) which connects $x$ to $y$ in the counterclockwise direction as long the particular choice does not imply that $\eta$ (resp.\ $\wt{\eta}$) almost surely intersects this part of the boundary.  We are now ready to prove Theorem~\ref{thm::bi_chordal}.

\begin{figure}[h!]
\begin{center}
\includegraphics[scale=0.85]{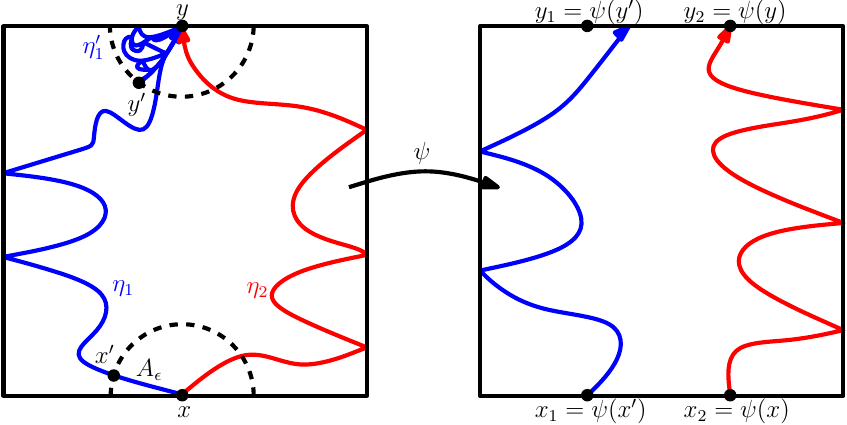}
\end{center}
\caption{\label{fig::coupling_reduction}The first step in the proof of Theorem~\ref{thm::bi_chordal} is to reduce to the case that the initial and terminal points of $\eta_1$ and $\eta_2$ are distinct.  This is accomplished as follows.  Conditional on $\eta_2$, we know that $\eta_1$ is an $\SLE_\kappa(\ul{\rho}^L)$ process in the left connected component $D_2$ of $D \setminus \eta_2$.  This implies that there exists an $\SLE_{16/\kappa}(\ul{\rho}')$ process $\eta_1'$ in $D_2$ such that $\eta_1$ is almost surely the right boundary of $\eta_1'$ (recall Figure~\ref{fig::counterflowline}).  Fix $\epsilon > 0$, let $\tau_\epsilon = \inf\{t \geq 0 : \eta_1(t) \notin B(x,\epsilon)\}$, $\tau_\epsilon' = \inf\{t \geq 0 : \eta_1'(t) \notin B(y,\epsilon)\}$,  $A_\epsilon = \eta_1([0,\tau_\epsilon])$, and $A_\epsilon' = \eta_1'([0,\tau_\epsilon'])$.
Then $\eta_1$ almost surely merges into (and never separates from) the right boundary of $A_\epsilon'$ before hitting $y$.  Let $\psi$ be a conformal map which takes the connected component of $D \setminus (A_\epsilon \cup A_\epsilon')$ which contains $x^+$ and $y^-$ on its boundary to $D$.  Let $\ol{\rho}^L = \sum_{i} \rho^{i,L}$.  Then the law of $\psi(\eta_1)$ conditional on $\psi(\eta_2)$ is an $\SLE_\kappa(\ul{\rho}^L,\tfrac{\kappa}{2}-2-\ol{\rho}^L;\kappa-4)$ process with the extra force points located at the image under $\psi$ of the left and right most points of $A_\epsilon' \cap \partial D$ (see Figure~\ref{fig::counterflowline}).  Let $y_1$ be the image of the tip $y'$ of $A_\epsilon'$ under $\psi$ and $y_2 = \psi(y^-)$.  Then $\psi(\eta_1)$ almost surely first exits $D$ in $(y_1,y_2)$ (since $\eta_1$ almost surely merges into the right boundary of $A_\epsilon'$ before hitting $y$) and $\psi(\eta_2)$ exits at $y_2$.  Conditional on $\psi(\eta_1)$, $\psi(\eta_2)$ is an $\SLE_\kappa(\ul{\rho}^R)$ process.
}
\end{figure}

\begin{proof}[Proof of Theorem~\ref{thm::bi_chordal}]
We are first going to explain how, using $\SLE$ duality, we can reduce to the setting that the initial and terminal points of $\eta_1$ and $\eta_2$ are distinct.  Let $D_1$ be the right connected component of $D \setminus \eta_1$ and let $D_2$ be the left connected component of $D \setminus \eta_2$.  Fix $\epsilon > 0$ and let $\tau_\epsilon$ be the first time $t$ that $|\eta_1(t)-x| = \epsilon$.  Conditional on $\eta_2$, we can couple $\eta_1$ with an $\SLE_{16/\kappa}(\ul{\rho}')$ process $\eta_1'$ from $y$ to $x$ in $D_2$ such that $\eta_1$ is almost surely the right boundary of $\eta_1'$ (recall Figure~\ref{fig::counterflowline}).  Let $\tau_\epsilon'$ be the first time $t$ that $|\eta_1'(t) - y| = \epsilon$.  The conditional law of $\eta_1$ given $\eta_2$, $A_\epsilon := \eta_1([0,\tau_\epsilon])$, and $A_\epsilon' := \eta_1'([0,\tau_\epsilon'])$ is an $\SLE_\kappa(\ul{\rho}^L,\tfrac{\kappa}{2}-2-\ol{\rho}^L;\kappa-4)$ process from $\eta_1(\tau_\epsilon)$ to $\eta_1'(\tau_\epsilon')$ where $\ol{\rho}^L = \sum_i \rho^{i,L}$.  The extra force points are located at the left and right most points of $A_\epsilon' \cap \partial D$ (the right most point is actually $y^-$; see Figure~\ref{fig::coupling_reduction}).  Now, conditional on $\eta_1, A_\epsilon$, and $A_\epsilon'$, the law of $\eta_2$ is still an $\SLE_\kappa(\ul{\rho}^R)$ from $x$ to $y$ in $D_1$.  Let $\psi \colon D_{\epsilon} \to D$, $D_{\epsilon}$ the connected component of $D \setminus (A_\epsilon \cup A_\epsilon')$ whose boundary contains both $x^+$ and $y^-$, be a conformal transformation and set $x_1 = \psi(\eta_1(\tau_\epsilon))$, $x_2 = \psi(x^+)$, $y_1 = \psi(\eta_1'(\tau_\epsilon'))$, and $y_2 = \psi(y^-)$.  Given $\psi(\eta_2)$, we have that $\psi(\eta_1)$ is an $\SLE_\kappa(\ul{\rho}^L,\tfrac{\kappa}{2} - 2-\ol{\rho}^L;\kappa-4)$ process from $x_1$ to $y_1$ where the extra force points are located at the image of the left most point of $A_\epsilon' \cap \partial D$ under $\psi$ and at $y_2$.  This implies that $\psi(\eta_1)$ almost surely exits $D$ in the open interval $(y_1,y_2)$ (since $\eta_1$ almost surely merges with the right boundary of $A_\epsilon'$ before hitting $y$).  Conditional on $\psi(\eta_1)$, $\psi(\eta_2)$ is an $\SLE_\kappa(\ul{\rho}^R)$ from $x_2$ to $y_2$ in the right connected component of $D \setminus \psi(\eta_1)$.  Therefore the initial and terminal points of $\psi(\eta_1)$ and $\psi(\eta_2)$ are almost surely distinct.

Suppose that $\mu$ is any stationary distribution of the chain described just after the statement of Theorem~\ref{thm::bi_chordal}.  Fix $\epsilon > 0$ small so that the counterclockwise segment $S_1$ of $\partial D$ from $y_1$ to $x_1$ has distance at least $4 \epsilon$ from the counterclockwise segment $S_2$ of $\partial D$ from $x_2$ to $y_2$ and let $A_\epsilon$ be the event that $\eta_1$ (resp.\ $\eta_2$) has distance at least $\epsilon$ from $S_2$ (resp.\ $S_1$).  Let $\mu_\epsilon$ be given by $\mu$ conditioned on $A_\epsilon$.  We note that from the form of the Markov chain described just above and the stationarity of~$\mu$ it is easy to see that $\mu(A_\epsilon) > 0$.  We further note that~$\mu_\epsilon$ is stationary under the Markov chain as defined just above except in each step we resample each of the paths conditioned on $A_\epsilon$.  We will call this the ``$\epsilon$-Markov chain.''

To be concrete, we view $\mu_\epsilon$ as a Borel probability measure on the space of pairs of closed, connected sets $(K_1,K_2)$ in $\ol{D}$ equipped with the Hausdorff distance where the distance between $K_1$ (resp.\ $K_2$) and $S_2$ (resp.\ $S_1$) is at least $\epsilon$ and $x_j,y_j \in K_j$ for $j \in \{1,2\}$.  We equip the space of such measures with the weak topology.  Let $\CS_\epsilon$ be the set of all such stationary probability measures.  Then $\CS_\epsilon$ is clearly convex.  We claim that~$\CS_\epsilon$ is in fact compact.  To see this, we suppose that $(\nu_n)$ is a sequence in~$\CS_\epsilon$ which converges weakly to $\nu$.  We will argue that $\nu \in \CS_\epsilon$.  Suppose that $(K_1^n,K_2^n)$ has law $\nu_n$ and $(K_1,K_2)$ has law $\nu$.  By the Skorohod representation theorem, we can put $(K_1^n,K_2^n)$ and $(K_1,K_2)$ onto a common probability space so that $K_j^n \to K_j$ almost surely as $n \to \infty$ for $j \in \{1,2\}$ with respect to the Hausdorff distance.  Let~$V_j^n$ (resp.~$V_j$) be the component of $D \setminus K_j^n$ (resp.\ $D \setminus K_j$) which contains $x_{3-j},y_{3-j}$ on its boundary.

Assume that $w_j$ has been fixed in $S_j$ and is distinct from $x_j,y_j$.  Let $\varphi_1 \colon V_1 \to \D$ be the unique conformal map which takes $x_2$ to $-i$, $w_2$ to $1$, and $y_2$ to $i$.  Since the distance between $K_1$ and $S_2$ is at least $\epsilon$, we note that there exists $\zeta > 0$ depending only on $\epsilon$, $D$, $x_2$, $w_2$, and $y_2$ such that $\varphi_1(K_1)$ is contained in the counterclockwise segment $T_1$ of $\partial \D$ from $e^{i(\pi/2+\zeta)}$ to $e^{i(3\pi/2-\zeta)}$.  For each $\delta > 0$ we let $\wt{U}_1^\delta$ be the set of points $z \in \D$ whose distance from $T_1$ is at least $\delta$.  Finally, we let $U_1^\delta = \varphi_1^{-1}(\wt{U}_j^\delta)$.  We define $U_2^\delta$ analogously.

Let $\CF_j = \sigma(K_j, K_j^n : n \in \N)$ be the $\sigma$-algebra generated by $K_j$ and \emph{all of the} $K_j^n$ for $n \in \N$.  In what follows, we will always be conditioning on $\CF_j$ (so that $K_j$ and \emph{all of the} $K_j^n$ for $n \in \N$ are fixed).  By the convergence of $K_j^n \to K_j$ with respect to the Hausdorff distance, we have that there exists $n_0 \in \N$ such that $n \geq n_0$ implies that $U_1^\delta \subseteq V_1^n$.  Note that both $n_0$ and $U_j^\delta$ are $\CF_j$-measurable.

For each $n$, we let $h_j^n$ be an instance of the GFF on $V_j^n$ with the boundary conditions so that its flow line from $x_{3-j}$ to $y_{3-j}$ has the correct law to perform the resampling operation.  We define~$h_j$ analogously for~$V_j$.  We assume that the $h_j^n$ (for all $n \in \N$) and $h_j$ are coupled together on a common probability space so that, given $\CF_j$, we have that the~$h_j^n$ (for all $n \in \N$) and~$h_j$ are conditionally independent.  Fix $\delta > 0$.  We will now argue that, given $\CF_j$, the total variation distance between the law of $h_j^n|_{U_j^\delta}$ and the law of $h_j|_{U_j^\delta}$ almost surely tends to $0$ as $n \to \infty$.  That is, if we let $\CL(X \giv \CF_j)$ denote the conditional law of a random variable $X$ given $\CF_j$ and $\| \cdot \|_{\mathrm{TV}}$ denote the total variation distance,  we have almost surely that
\begin{equation}
\label{eqn::tv_convergence}
\lim_{n \to \infty} \| \CL(h_j^n|_{U_j^\delta} \giv \CF_j) - \CL(h_j|_{U_j^\delta} \giv \CF_j) \|_{\mathrm{TV}} = 0.
\end{equation}
To see this, we fix $\wh{\delta} \in (0,\delta/4)$.  Assume that $n$ is such that $U_j^{\wh{\delta}} \subseteq V_j^n$ (note that we can choose $n$ in an $\CF_j$-measurable way).  By the Markov property for $h_j^n$ (resp.\ $h_j$) we can write $h_j^n$ (resp.~$h_j$) as the sum of a zero boundary GFF $\wh{h}_j^n$ (resp.\ $\wh{h}_j$) on $U_j^{\wh{\delta}}$ and the harmonic extension $\Fh_j^n$ (resp.\ $\Fh_j$) of the boundary values of $h_j^n$ (resp.~$h_j$) from $\partial U_j^{\wh{\delta}}$ to $U_j^{\wh{\delta}}$.  We let $\wh{\phi}_j^\delta$ be a $C_0^\infty$ function which is supported in $U_j^{\delta/2}$ with $\wh{\phi}_j^\delta|_{U_j^\delta} \equiv 1$ and then take $\phi_j^{n,\delta,\wh{\delta}} = (\Fh_j^n-\Fh_j) \wh{\phi}_j^\delta$.  (We note that we can choose $\wh{\phi}_j^\delta$ in an $\CF_j$-measurable way.)  We can transform from $\CL(h_j|_{U_j^\delta} \giv \CF_j)$ to $\CL(h_j^n|_{U_j^\delta} \giv \CF_j)$ by weighting the former by $\exp( (\wh{h}_j,\phi_j^{n,\delta,\wh{\delta}})_\nabla)$ (this is the infinite dimensional analog of the fact that weighting the law of $Z \sim N(0,1)$ by $e^{\mu x}$ yields a $N(\mu,1)$).  We emphasize that $\exp((\wh{h}_j,\phi_j^{n,\delta,\wh{\delta}})_\nabla)$ is a measurable function of the pair $h_j,h_j^n$ because it is determined by $\wh{h}_j$, $\Fh_j^n$, and $\Fh_j$.  In particular, $\Fh_j^n$ and $\Fh_j$ can be constructed from $h_j^n$ and $h_j$ via a series expansion where the coefficients are computed by taking $(\cdot,\cdot)_\nabla$-inner products of $h_j$ and $h_j^n$ with an orthonormal basis of those functions in $H(V_j)$ and $H(V_j^n)$ which are harmonic in $U_j^{\wh{\delta}}$.  Given $\Fh_j^n$ and $\Fh_j$, hence $\phi_j^{n,\delta,\wh{\delta}}$, we can determine $(\wh{h}_j,\phi_j^{n,\delta,\wh{\delta}})_\nabla$ by representing both $\wh{h}_j$ and $\phi_j^{n,\delta,\wh{\delta}}$ in terms of an orthonormal basis of $H(V_j)$ and then taking the sum of the products of the coefficients.  It is not difficult to see that for each $\delta,\epsilon' > 0$ there exists $\wh{\delta}_0 > 0$ and $n_0 \in \N$ (both $\CF_j$-measurable) such that $\wh{\delta} \in (0,\wh{\delta}_0)$ and $n \geq n_0$ implies that
\[ \p\left[ \sup_{z \in U_j^{\delta/4}}|\Fh_j^n(z) - \Fh_j(z)| \leq \epsilon' \giv \CF_j \right] \geq 1-\epsilon'.\]

By applying a (non-random) conformal transformation to all of $D$, we may assume without loss of generality that $D \subseteq \h$ and that $\partial V_j \setminus K_j = \partial V_j^n \setminus K_j^n$ is given by an interval of $\R$.  Since $\Fh_j^n - \Fh_j$ is harmonic on $U_j^{\delta/4}$ and vanishes on $\partial V_j \setminus K_j = \partial V_j^n \setminus K_j^n$, we can extend it to a function $\Fg_j^n$ which is harmonic on $U_j^{\delta/4} \cup (U_j^{\delta/4})^*$ where $(U_j^{\delta/4})^* = \{ \ol{z} : z \in U_j^{\delta/4}\}$.  Since $\Fg_j^n$ is harmonic on $U_j^{\delta/4} \cup (U_j^{\delta/4})^*$, its derivatives can be bounded in terms of its supremum on compact sets.  Since the closure of $U_j^{\delta/2} \cup (U_j^{\delta/2})^*$, $(U_j^{\delta/2})^* = \{ \ol{z} : z \in U_j^{\delta/2}\}$, is a compact subset of $U_j^{\delta/4} \cup (U_j^{\delta/4})^*$, it follows that for each $\delta,\epsilon' > 0$ there exists $\wh{\delta}_0 > 0$ and $n_0 \in \N$ (both $\CF_j$-measurable) such that $\wh{\delta} \in (0,\wh{\delta}_0)$ and $n \geq n_0$ implies that
\[ \p\left[ \| \Fg_j^n |_{U_j^{\delta/2} \cup (U_j^{\delta/2})^*} \|_\nabla \leq \epsilon' \giv \CF_j \right] \geq 1-\epsilon'.\]
It therefore follows the same is true for $\| (\Fh_j^n - \Fh_j)|_{U_j^{\delta/2}}\|_\nabla$.   By the Cauchy-Schwarz inequality, we thus have that for each $\delta,\epsilon' > 0$ there exists $\wh{\delta}_0 > 0$ ($\CF_j$-measurable) such that $\delta \in (0,\wh{\delta}_0)$ implies that
\begin{equation}
\label{eqn::phi_norm_bound}
\p\left[ \| \phi_j^{n,\delta,\wh{\delta}}\|_\nabla \leq \epsilon' \giv \CF_j \right] \geq 1-\epsilon'.
\end{equation}
Combining~\eqref{eqn::phi_norm_bound} with the form of the Radon-Nikodym derivative implies the claimed convergence in total variation~\eqref{eqn::tv_convergence}.

We will now deduce from the claim established in the previous paragraph that the conditional law given $\CF_j$ of the flow line from $x_{3-j}$ to $y_{3-j}$ generated using $h_j^n$ almost surely converges in total variation as $n \to \infty$ to the conditional law given $\CF_j$ of the flow line from $x_{3-j}$ to $y_{3-j}$ generated using $h_j$.  Fix~$\epsilon' > 0$.  Then there exists $\delta > 0$ ($\CF_j$-measurable) such that the conditional probability given $\CF_j$ that the flow line of the latter stays in $U_j^\delta$ is at least $1-\epsilon'/2$.  By the total variation convergence established in the previous paragraph, we can find $n_0 \in \N$ ($\CF_j$-measurable) such that $n \geq n_0$ implies that
\[ \| \CL(h_j^n|_{U_j^\delta} \giv \CF_j) - \CL(h_j|_{U_j^\delta} \giv \CF_j) \|_{\mathrm{TV}} \leq \frac{\epsilon'}{2}.\]
Combining this with \cite[Theorem~1.2]{MS_IMAG} implies that for $n \geq n_0$ we have that the total variation distance between the conditional laws of the two aforementioned flow lines given $\CF_j$ is at most~$\epsilon'$ (as we can couple the paths onto a common probability space, given $\CF_j$, to be equal with probability at least $1-\epsilon'$).  The claimed total variation convergence given $\CF_j$ follows since~$\epsilon' > 0$ was arbitrary.  In particular, the conditional probability given $\CF_j$ of the event that the former has distance at least $\epsilon$ from $S_j$ converges to the corresponding conditional probability given $\CF_j$ for the latter.  This implies that the conditional law of the former given $\CF_j$ and \emph{conditioned} on having distance at least $\epsilon$ from $S_j$ converges in total variation to the conditional law given $\CF_j$ of the latter \emph{conditioned} on the same event.  Since total variation convergence implies weak convergence, we therefore have that the one step transition kernel for the $\epsilon$-Markov chain is continuous, which implies that $\nu$ is stationary.  This proves that $\CS_\epsilon$ is compact.

As $\CS_\epsilon$ is a compact and convex subset of a locally convex space, Choquet's theorem (see, e.g., \cite[Section~3]{choquet_lectures}) thus implies that $\mu_\epsilon$ can be uniquely expressed as a superposition of extremal (or ergodic measures) in $\CS_\epsilon$.  That is, there exists a probability measure $m_\epsilon$ on $\CS_\epsilon$ supported on the extremal elements $\nu$ of $\CS_\epsilon$ such that
\[ \mu_\epsilon= \int \nu dm_\epsilon(\nu).\]
It is clear from the definition of the $\epsilon$-Markov chain that $\mu_\epsilon$ is supported on pairs of non-crossing paths in $D$ connecting $x_1,y_1$ and $x_2,y_2$.  Therefore $m_\epsilon$ is necessarily supported on such measures.  To show that $\CS_\epsilon$ consists of a single such element, it suffices to show that there is only one extremal such element of $\CS_\epsilon$.  Suppose that $\nu,\wt{\nu}$ are extremal elements of $\CS_\epsilon$.

We will now show that if $\nu,\wt{\nu}$ are distinct then $\nu,\wt{\nu}$ are necessarily mutually singular.  Lebesgue's decomposition theorem implies that we can uniquely write $\nu = \nu_0 + \nu_1$ where $\nu_0$ (resp.\ $\nu_1$) is absolutely continuous (resp.\ singular) with respect to $\wt{\nu}$.  If both $\nu_0$ and $\nu_1$ are non-zero, then they can be normalized to be probability measures which are both stationary for the $\epsilon$-Markov chain (by the uniqueness of the Lebesgue decomposition).  This contradicts that $\nu$ is an extremal measure.  We also need to rule out the possibility that $\nu$ is absolutely continuous with respect to $\wt{\nu}$.  We suppose for contradiction that $\nu$ is absolutely continuous with respect to $\wt{\nu}$ and let $f = d\nu/d\wt{\nu}$ be the Radon-Nikodym derivative of $\nu$ with respect to $\wt{\nu}$.

Suppose that, given $X_0$, we sample $(X_k)$ using the $\epsilon$-Markov chain defined above.  We will write $\E_\nu$ (resp.\ $\E_{\wt{\nu}}$) for the expectation under the law where the distribution of $X_0$ is given by $\nu$ (resp.\ $\wt{\nu}$).  Fix a bounded function $g$.  Then we have that
\[ \E_\nu[ g(X_1) f(X_0)] = \E_{\wt{\nu}}[ g(X_1) ] = \E_{\wt{\nu}}[ g(X_0)] = \E_\nu[g(X_0) f(X_0)] = \E_\nu[ g(X_1) f(X_1)].\]
It therefore follows that
\[ \E_\nu[ g(X_1) \E_\nu[ f(X_0) \giv X_1]] = \E_\nu[ g(X_1) f(X_1)]\]
for all bounded functions $g$.  This implies that $\E_\nu[ f(X_0) \giv X_1] = f(X_1)$ under $\nu$ almost surely.  More generally, the same argument implies that $\E_\nu[ f(X_n) \giv X_{n+1}] = f(X_{n+1})$ for all $n$ under $\nu$ almost surely.  As $(X_k)$ is a Markov chain, we have for each $n$ that $X_n$ is conditionally independent of $X_m$ for $m \geq n+2$ given $X_{n+1}$.  This implies that
\[ \E_\nu[ f(X_n) \giv X_m,\ m \geq n+1] = \E_\nu[ f(X_n) \giv X_{n+1}] = f(X_{n+1})\]
which, in turn, implies that $M_k^n = f(X_{n-k})$ for $0 \leq k \leq n$ is a martingale under~$\E_\nu$.  By stationarity, we have for all $n_1,n_2 \in \N$ that $M^{n_1}|_{[0,n]} \stackrel{d}{=} M^{n_2}|_{[0,n]}$ where $n = n_1 \wedge n_2$.  Therefore by the Kolmogorov extension theorem there exists a process $\wt{M}_k$ which is defined for all $k$ such that $\wt{M}|_{[0,n]} \stackrel{d}{=} M^n$ for all $n$.  As $\wt{M}$ is a non-negative martingale, the martingale convergence theorem implies that $\lim_n \wt{M}_n$ almost surely exists.  In particular, we almost surely have that $\wt{M}_n - \wt{M}_{n-1} \to 0$ as $n \to \infty$.  Fix $\delta > 0$.  Then we have that
\begin{align*}
     \p_\nu[ |f(X_0) - f(X_1)| \geq \delta ] &= \lim_{n \to \infty} \p_\nu[ | M_n^n - M_{n-1}^n| \geq \delta]\\
     &= \lim_{n \to \infty} \p[ |\wt{M}_n - \wt{M}_{n-1}| \geq \delta]
     = 0.
\end{align*}
That is, $f(X_0) = f(X_1)$ almost surely under $\nu$.  If $\nu,\wt{\nu}$ are distinct then there exists $a > 0$ such that both $\{ f > a\}$ and $\{ f \leq a\}$ have positive $\wt{\nu}$-measure.  We can write $\nu = \mu_0 + \mu_1$ where $\mu_0$ (resp.\ $\mu_1$) is the measure with Radon-Nikodym derivative with respect to $\wt{\nu}$ given by $f \one_{\{ f > a\}}$ (resp.\ $f \one_{\{f \leq a\}}$).  The above argument implies that $\mu_0,\mu_1$ can be normalized to be probability measures in $\CS_\epsilon$, which contradicts that $\nu$ is extremal.

We will now argue that $\nu = \wt{\nu}$ by showing that they cannot be mutually singular.  Suppose that we have two initial configurations $(\gamma_1,\gamma_2) \sim \nu$ and $(\wt{\gamma}_1,\wt{\gamma}_2) \sim \wt{\nu}$, sampled independently.  Let $(\gamma_1^n,\gamma_2^n)$ and $(\wt{\gamma}_1^n,\wt{\gamma}_2^n)$ be the $n$th step of the $\epsilon$-Markov chain described above with the initial data $(\gamma_1,\gamma_2)$ and $(\wt{\gamma}_1,\wt{\gamma}_2)$ (corresponding to $n=0$), respectively, coupled as follows.  Fix $n \geq 0$.  We sample the same value of $i_{n+1}$ for both pairs $(\gamma_1^n,\gamma_2^n)$ and $(\wt{\gamma}_1^n,\wt{\gamma}_2^n)$ and then take $(\gamma_{i_{n+1}}^{n+1},\wt{\gamma}_{i_{n+1}}^{n+1})$ to maximize the probability that they are equal given $(\gamma_{j_{n+1}}^{n+1},\wt{\gamma}_{j_{n+1}}^{n+1})$, where $j_{n+1} = 3 - i_{n+1}$.  Two applications of Lemma~\ref{lem::coupling} imply that $\p[(\gamma_1^2,\gamma_2^2) = (\wt{\gamma}_1^2,\wt{\gamma}_2^2)] > 0$.  Since $(\gamma_1^2,\gamma_2^2) \sim \nu$ and $(\wt{\gamma}_1^2,\wt{\gamma}_2^2) \sim \wt{\nu}$, we thus have that $\nu$ is not mutually singular with respect to $\wt{\nu}$, which in turn implies $\nu = \wt{\nu}$.  We conclude that $\CS_\epsilon$ consists of a single element and therefore the $\epsilon$-Markov chain has a unique stationary distribution.  Sending $\epsilon \to 0$ implies that the original chain has a unique stationary distribution.

We can explicitly construct a pair of random paths $(\wt{\eta}_1,\wt{\eta}_2)$ whose law is a stationary distribution to this chain as follows.  We let $\wt{\eta}_1$ be an $\SLE_\kappa(\ul{\rho}^L,\tfrac{\kappa}{2}-2-\ol{\rho}^L;2+\ul{\rho}^R,\kappa-4-\ol{\rho}^R)$ process in $D$ from $x_1$ to $y_1$ where the force points on the left side are at the same location as those of $\psi(\eta_1)$, the first $|\ul{\rho}^R|$ force points on the right side are at the same location as those of $\psi(\eta_2)$, and the force point with weight $\kappa-4-\ol{\rho}^R$ is located at $y_2$.  We then take $\wt{\eta}_2$ to be an $\SLE_\kappa(\ul{\rho}^R)$ in the right connected component of $D \setminus \wt{\eta}_1$, where the force points are at the same location as those of $\psi(\eta_2)$.  The justification that $(\wt{\eta}_1,\wt{\eta}_2)$ is stationary follows from a GFF construction analogous to that considered in Figure~\ref{fig::bi_chordal} (see also Figure~\ref{fig::monotonicity} and the nearby discussion in Section~\ref{subsec::imaginary}).
This allows us to write down the law of $\eta_1$ given $A_\epsilon, A_\epsilon'$.  It is an $\SLE_\kappa(\ul{\rho}^L,\tfrac{\kappa}{2} - 2-\ol{\rho}^L;2+\ul{\rho}^R,\kappa-4-\ol{\rho}^R)$ process from $\eta(\tau_\epsilon)$ to $\eta'(\tau_\epsilon')$ where $\ol{\rho}^R = \sum_i \rho^{i,R}$.  The extra force points are at the left and right most points of $A_\epsilon' \cap \partial D$ and $\eta_1'(\tau_\epsilon)$.  Since the law of such processes are continuous in the locations of their force points \cite[Section~2]{MS_IMAG}, sending $\epsilon \to 0$, we see that $\eta_1$ has to be an $\SLE_\kappa(\ul{\rho}^L; 2 + \ul{\rho}^R)$.  This completes the proof since, of course, we know the conditional law of~$\eta_2$ given~$\eta_1$.
\end{proof}

\begin{remark}
Theorem~\ref{thm::bi_chordal} can be extended to $n$ paths for any $n>2$.  That is, suppose $h$ is an instance of the GFF on $\h$ with piecewise constant boundary conditions, and we have a collection of flow lines $\eta_1, \eta_2, \ldots \eta_n$ of $\h$ starting at $0$ but with different angles $\theta_1 > \theta_2 > \ldots >  \theta_n$.  Then the joint law is characterized by the conditional law of each path given the others.  In other words, if $\wt{\eta}_1, \wt{\eta}_2, \ldots, \wt{\eta}_n$ is any collection of paths such that the conditional law of each $\wt{\eta}_j$ given the others is the same as the law of a flow line of the appropriate component of $\h \setminus (\cup_{k \neq j} \wt{\eta}_k)$ (with the corresponding boundary conditions) then the $\wt{\eta}_j$ have the same joint law as the $\eta_j$.  This can be shown by induction.  Assume the statement is true for $n-1$ paths.  Then given $\eta_1$, the inductive hypothesis implies that the conditional law of each of the other paths is determined.  Moreover, given $\eta_2$, the conditional law of $\eta_1$ is determined.  This in particular determines the marginal law of $\eta_1$.  The result then follows since, as already mentioned, the conditional of the other paths given $\eta_1$ is determined.
\end{remark}

\begin{remark}
\label{rem::rate_of_convergence}  In the context of the proof of Theorem~\ref{thm::bi_chordal}, a natural question to ask is how long it takes the Markov chain in question to converge to stationarity (after we have separated the initial and terminal points of the paths).  It turns out that it is possible to construct a coupling between two different realizations of the chain which requires a finite number of steps to couple with positive probability \emph{uniformly} in the initial configuration.  Although we will not need this statement for our main results, we sketch a proof in Figure~\ref{fig::division}.
\end{remark}

\begin{figure}[ht!]
\begin{center}
\subfigure{
\includegraphics[scale=0.85,page=1]{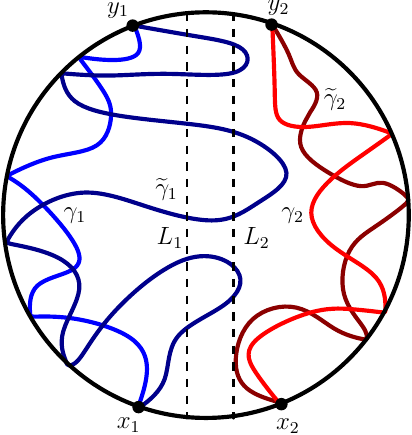}}
\hspace{0.025\textwidth}
\subfigure{
\includegraphics[scale=0.85,page=2]{figures/mixing_proof.pdf}}
\end{center}
\caption{\label{fig::division} It is in fact possible to use another coupling in the proof of Theorem~\ref{thm::bi_chordal} which requires only a finite number of steps to couple with positive probability which is \emph{uniform} in the initial configuration (the coupling described there couples with positive probability, but thiss probability \emph{a priori} depends on the initial configuration).  The idea to prove this is the following.  Let $(\gamma_1^n,\gamma_2^n)$ and $(\wt{\gamma}_1^n,\wt{\gamma}_2^n)$ denote the $n$th step of this Markov chain with initial data $(\gamma_1,\gamma_2)$ and $(\wt{\gamma}_1,\wt{\gamma}_2)$, respectively.  We couple together the two chains by taking $i_n$ to be the same for both.  When $i_{n+1}=1$, we pick $\gamma_1^{n+1},\wt{\gamma}_1^{n+1}$ to maximize the probability that they are equal and, when $i_{n+1}=2$, we pick $\gamma_2^{n+1},\wt{\gamma}_2^{n+1}$ independently unless $\gamma_1^n = \wt{\gamma}_1^n$, in which case we take them to be the same.  It is possible to argue that when $i_{n+1}=2$, with positive probability uniform in the realization of $\gamma_1^{n+1}$ and $\wt{\gamma}_1^{n+1}$, we have that both $\gamma_2^{n+1}$ and $\wt{\gamma}_2^{n+1}$ lie in the right side of $D$ (to the right of $L_2$ depicted in the left panel).  Conditional on this event happening, $\gamma_1^{n+2},\wt{\gamma}_1^{n+2}$ can be coupled to be equal with uniformly positive probability.  The technical estimate that this proof requires is a quantitative estimate on the Radon-Nikodym derivative which appears implicitly in the proof of Lemma~\ref{lem::coupling}.}
\end{figure}

\section{Paths conditioned to avoid the boundary}
\label{sec::conditioned_not_to_hit}

\subsection{Weighting by martingales}
\label{subsec::martingales}

As illustrated later in Figure~\ref{fig::abswap} and Figure~\ref{fig::abrhocalculation}, an $\SLE_\kappa(\rho^L;\rho^R)$ process from the bottom to top of the vertical strip $\vstrip$ can also be understood as a flow line of a Gaussian free field on $\vstrip$ with boundary conditions of $-a$ on the left side $\vstripleft$ of $\vstrip$ and $b$ on the right side $\vstripright$ of $\vstrip$, where $\rho^L = \tfrac{\kappa}{4}-2 + \tfrac{a}{\lambda}$ and $\rho^R = \tfrac{\kappa}{4}-2 + \tfrac{b}{\lambda}$.  Recalling~\eqref{eqn::fullrevolution}, the constraint that $\rho^q > -2$ for $q \in \{L,R \}$ corresponds to $a > - \lambda'$ and $b > - \lambda'$.  If this constraint holds, then the path hits both $\vstripleft$ and $\vstripright$ almost surely provided that $ \rho^q < \tfrac{\kappa}{2}-2$ for $q \in \{L,R \}$, which corresponds to $a < \lambda'$ and $b < \lambda'$.

In this section, we explore the following question: what does it mean to consider such a path {\em conditioned} not to hit $\vstripleft$ and $\vstripright$?  One has to be careful with this question because it involves conditioning on an event of probability zero.  We would like to argue that after this conditioning, appropriately interpreted, the $\SLE_\kappa(\rho^L;\rho^R)$ process becomes an $\SLE_\kappa(\refrho{\rho}^L;\refrho{\rho}^R)$ process where $\refrho{\rho}^q = (\kappa -4) -\rho^q$ for $q \in \{L,R\}$.

Observe that $\rho^q$ and $\refrho{\rho}^q$ are equidistant from the critical value $\tfrac{\kappa}{2}-2$ for boundary intersection.  We note that the effect that the conditioning has on the weights of the force points is natural in view of the some standard facts about Bessel processes.  The driving function for a boundary-intersecting one-sided $\SLE_\kappa(\rho)$ from $0$ to $\infty$ in $\h$ can be described completely in terms of a Bessel process $X$ of dimension $\delta < 2$ as in~\eqref{eqn::deflist}, where $X_t$ reaches zero precisely at the times when the curve hits the boundary.  The reader may recall that $|X_t|^\alpha$ is a local martingale when $\alpha= 2 - \delta > 0$ \cite{RY04}.  Girsanov's theorem then implies that if $|X_0| > 0$ then {\em conditioning} $|X_t|$ to reach $N$ before $0$ amounts to adding a drift term to $X_t$ given by the spatial derivative of $\log |X_t|^\alpha$, i.e., adding $(\alpha/X_t)dt$ to the RHS of~\eqref{eqn::bessel}, which changes $X_t$ from a Bessel process of dimension $\delta$ to a Bessel process of dimension $\delta + 2 \alpha = 4 - \delta$.  This corresponds to changing the weight of the force point to $(\kappa-4)-\rho$.

Alternatively, one can note that conditioning on the curve avoiding the boundary until a given {\em capacity} time $t$ is equivalent to conditioning $X$ on the event $\Omega_t^+ = \{X|_{[0,t]} > 0\}$, an event that has positive probability as long as the initial force point is {\em not} the SLE seed.  A similar argument shows that the conditional law of $X$ given $\Omega_t^+$ converges as $t \to \infty$ to a Bessel process of dimension $4-\delta$.

The following proposition (expressed using $\h$ instead of $\vstrip$) will give us one way to make sense of this conditioning in the two-sided case.

\begin{proposition}
\label{prop::mtweight}
Consider an $\SLE_\kappa(\rho^L;\rho^R)$ process $\eta$ in $\h$ from $0$ to $\infty$ with the initial force points $x^L < 0$ and $x^R > 0$, where $\rho^q < \tfrac{\kappa}{2}-2$ for $q \in \{L,R\}$, so that at least one of the force points is almost surely absorbed in finite time.  Let $V_t^L,V_t^R$ denote the evolution of the force points corresponding to $x^L,x^R$, respectively, under the Loewner flow.  If we weight the law of $\eta$ by the local martingale
\begin{equation}
\label{eq::Mtdef}
M_t:= \frac{1}{\CZ} |V^L_t - W_t|^{(\kappa-4-2 \rho^L)/\kappa} |V^R_t - W_t|^{(\kappa-4-2 \rho^R)/\kappa} |V_t^R - V^L_t|^{(\kappa-4 - \rho^L - \rho^R)(\kappa-4)/2\kappa},
\end{equation}
where $\CZ$ is chosen so that $M_0 = 1$, then we obtain an $\SLE_\kappa(\refrho{\rho}^L;~\refrho{\rho}^R)$ process with $\refrho{\rho}^q = (\kappa -4) -\rho^q$ for $q \in \{L,R\}$.  In particular, we obtain a process in which the force points are almost surely not absorbed in finite time.
\end{proposition}
\begin{proof}
This is essentially shown in \cite[Theorem~6]{SW05} and \cite[Remark~7]{SW05}.  In that paper the Radon-Nikodym derivative of $\SLE_\kappa(\rho^L;\rho^R)$ with respect to ordinary $\SLE_\kappa$ (stopped at an appropriate stopping time) was computed for any $\rho^L$ and $\rho^R$.  Taking the ratio of this quantity for $\SLE_\kappa(\rho^L;\rho^R)$ and for $\SLE_\kappa(\refrho{\rho}^L;\refrho{\rho}^R)$ yields the Radon-Nikodym derivative of one of these processes with respect to the other provided the curve has not hit the boundary ---and this has the form given in~\eqref{eq::Mtdef}.  Alternatively, the fact that $M_t$ is a local martingale --- and that weighting the law of an $\SLE_\kappa(\rho^L;\rho^R)$ by $M_t$ produces an $\SLE_\kappa(\refrho{\rho}^L;\refrho{\rho}^R)$ --- can be verified directly with \Ito/ calculus.
\end{proof}

A similar argument yields the following:

\begin{proposition} \label{prop::mtweight2}
Consider an $\SLE_\kappa(\rho^L;\rho^R)$ process $\eta$ in $\h$ from $0$ to $\infty$ with the initial force points $x^L < 0$ and $x^R > 0$, where $\rho^q < \tfrac{\kappa}{2}-2$ for $q \in \{L,R\}$ so that at least one of the force points, say $x^L$, is almost surely absorbed in finite time.  Let $V_t^L ,V_t^R$ denote the evolution of the force points corresponding to $x^L,x^R$, respectively, under the Loewner flow.  If we weight the law of $\eta$ by the local martingale
\begin{equation}
\label{eq::Mtdef2}
M^1_t:= \frac{1}{\CZ} |V^L_t - W_t|^{(\kappa-4-2 \rho^L)/\kappa}  |V_t^R - V^L_t|^{(\kappa-4 - 2\rho^L )\rho^R/2\kappa},
\end{equation}
where $\CZ$ is chosen so that $M^1_0 = 1$, then we obtain an $\SLE_\kappa(\refrho{\rho}^L; \rho^R)$ process with $\refrho{\rho}^L = (\kappa -4) -\rho^L$.  In particular, we obtain a process in which the left force point is almost surely not absorbed in finite time.
\end{proposition}

Note that in the setting of Proposition~\ref{prop::mtweight}, $M_t$ vanishes precisely when $\eta$ touches either $(-\infty,x^L]$ or $[x^R,\infty)$.  (Similarly, in the setting of Proposition~\ref{prop::mtweight2} $M^1_t$ vanishes precisely when $\eta$ touches $(-\infty,x^L]$.)  For each $N \geq 0$, we let $\tau_N = \inf\{t \geq 0 : M_t = N\}$.  Since $M_{t \wedge \tau_N \wedge \tau_0}$ is a bounded and continuous martingale (provided $N \geq 1$), the optional stopping theorem implies that $\E[M_{\tau_N \wedge \tau_0}] = 1$ and hence the event $E_N := \{\tau_N < \tau_0\}$ occurs with probability $1/N$.  Reweighting the law of $\eta_{t \wedge \tau_N \wedge \tau_0}$ by $M_{\tau_N \wedge \tau_0}$ corresponds to conditioning on $E_N$.  Thus Proposition~\ref{prop::mtweight} implies that $\eta_{t \wedge \tau_N \wedge \tau_0}$ conditioned on $E_N$ evolves as an $\SLE_\kappa(\refrho
{\rho}^L;\refrho{\rho}^R)$ process from $0$ to $\infty$ until time $\tau_N$. Since $M_t$ is a continuous local martingale, it evolves as a Brownian motion when parameterized by its quadratic variation.  Conditioning $M_t$ to stay positive for all time (by conditioning it to stay positive until $\tau_N$ and taking the $N \to \infty$ limit) corresponds to replacing this Brownian motion by a Bessel process of dimension $3$.  (Recall that a Bessel process of dimension $1$ evolves a Brownian motion when it is away from $0$, and that conditioning it to stay non-negative produces a Bessel process of dimension $4-1=3$.)  As explained above, this amounts to replacing the $\SLE_\kappa(\rho^L;\rho^R)$ process by an $\SLE_\kappa(\refrho{\rho}^L;\refrho{\rho}^R)$ process.  In this sense, at least, Proposition~~\ref{prop::mtweight} implies that an $\SLE_\kappa(\rho^L; \rho^R)$ process conditioned not to hit the boundary is an $\SLE_\kappa(\refrho{\rho}^L;\refrho{\rho}^R)$ process.  (And Proposition
~\ref{prop::mtweight2} yields an analogous statement for $\SLE_\kappa(\rho^L; \rho^R)$ conditioned not to hit one side of the boundary.)  Note also that Propositions~\ref{prop::mtweight} and~\ref{prop::mtweight2} make sense even for $\rho^q \leq -2$, which is the threshold for which an $\SLE_\kappa(\rho^L;\rho^R)$ process cannot be continued after hitting the boundary \cite[Section~2]{MS_IMAG}.

Finally, we point out that in fact both of the above results are special cases of the following more general proposition, which also follows from \cite[Theorem~6]{SW05} and \cite[Remark~7]{SW05} (and exactly the same argument given in the proof of Proposition~\ref{prop::mtweight}).

\begin{proposition}
\label{prop::mtweight3}
Consider an $\SLE_\kappa(\ul{\rho})$ process $\eta$ in $\h$ from $0$ to $\infty$ with finitely many initial force points $x^{j,L} < 0$ and finitely many force points $x^{i,R} > 0$.  Let $V_t^{j,q}$ denote the evolution of the force point corresponding to $x^{j,q}$ under the Loewner flow.  For notational purposes, we also write $V^{0,0}_t := W_t$ and $\rho^{0,0} := 2$ and treat this as an ``extra force point''.  Let $\wt{\ul{\rho}}$ be obtained from $\ul{\rho}$ by replacing some subset of the weights $\rho^{j,q}$ (other than $\rho^{0,0}$) with the ``dual weights'' $\refrho{\rho}^{j,q} := \kappa-4 - \rho^{j,q}$.
Define the local martingale
\begin{equation}
\label{eq::Mtdef3}
M_t:= \prod |V_t^{j,q} - V_t^{j',q'}|^{\bigl(\wt{\rho}^{j,q}~\wt{\rho}^{j',q'} - \rho^{j,q}\rho^{j',q'}\bigr)/2\kappa },
\end{equation}
where the product is taken over all distinct pairs $(j,q)$ and $(j',q')$ corresponding to force points (including the ``extra'' one $(0,0)$) and where $\CZ$ is chosen so that $M_0 = 1$.  Then if we weight the law of $\eta$ by $M_t$ we obtain an $\SLE_\kappa(\wt{\ul \rho})$ process.
\end{proposition}

\subsection{A single path avoiding the boundary}

Let $\eta$ be a flow line started at $0$ and targeted at $\infty$ of a GFF $h$ on $\h$ with piecewise constant boundary conditions.  Then $\eta$ is an $\SLE_\kappa(\ul \rho)$ process.  Let $\ul{x} = (\ul{x}^L;\ul{x}^R)$ denote the locations of the force points of $\eta$.  By adding zero-weight force points if necessary, we may assume without loss of generality that $x^{1,L} = 0^-$ and $x^{1,R} = 0^+$.  Let $I = [x^{2,L},x^{2,R}]$.  Then the boundary conditions for $h$ are constant on the left and right sides of $I  \setminus \{0\}$, which we denote by $I_L = [x^{2,L},0)$ and $I_R = (0,x^{2,R}]$.  Assume that either $h|_{I_L} > -\lambda+\pi\chi$ or $h|_{I_R} < \lambda - \pi \chi$ so that $\eta$ can hit at least one of $I_L$ or $I_R$ (or immediately hits the continuation threshold), see Figure~\ref{fig::hittingrange}.  We would like to define the law of $\eta$ conditioned {\em not} to hit $I$ before reaching $\infty$.   Since this amounts to conditioning on an event of probability zero, we have to state carefully what we mean.

The main result of this subsection is the following proposition which states that one can characterize the law by a ``Gibbs property'' that says if we are given any initial segment of the path, then the conditional law of the remainder of the path is what we expect it to be.  We will assume that the boundary conditions for $h$ are such that given a finite (boundary avoiding) segment of $\eta$, the continuation of $\eta$ as a flow line would have a positive probability of reaching $\infty$ before hitting the continuation threshold.  Concretely, this means that the boundary data for $h$ is constant on $(-\infty,x^{k,L})$ for $k = |\ul{\rho}^L|$ and at most $-\lambda+\pi \chi$ and constant in $(x^{\ell,R},\infty)$ for $\ell = |\ul{\rho}^R|$ and at least $\lambda-\pi \chi$.  Equivalently, we have both $\sum_{i=1}^k \rho^{i,L} \geq \tfrac{\kappa}{2}-2$ and $\sum_{i=1}^\ell \rho^{i,R} \geq \tfrac{\kappa}{2}-2$ (see~\eqref{eqn::ac_eq_rel} and Figure~\ref{fig::hittingrange}).

\begin{figure}[h!]
\begin{center}
\includegraphics[scale=0.85]{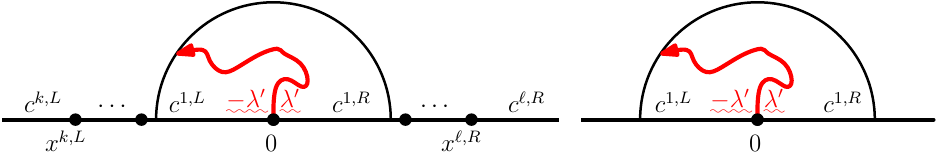}
\end{center}
\caption{\label{fig::two_half_discs}
In both panels, $\eta$ is ``conditioned'' in some sense to avoid hitting the intervals with values $c^{1,L}$ and $c^{1,R}$ before reaching $\infty$.  In the left panel, the boundary conditions are such that given the portion of the path shown, there is a positive probability that its continuation after hitting $\partial B(0,1)$ will reach $\infty$ before hitting these intervals.  Concretely, this means that $c^{k,L} \leq -\lambda+\pi \chi$ and $c^{\ell,R} \geq \lambda - \pi \chi$ (see~\eqref{eqn::ac_eq_rel} and Figure~\ref{fig::hittingrange}).  The law of a path segment that ends on $\partial B(0,1)$ and avoids the boundary is characterized by a certain path-resampling invariance (Proposition~\ref{prop::conditionalunique}).  In the right panel, the intervals with values $c^{1,L}$ and $c^{1,R}$ are infinite, and the ``conditioning not to hit'' these intervals is achieved by weighting by a martingale (recall Proposition~\ref{prop::mtweight}) that amounts to replacing the corresponding $\rho^{1,L}$ and $\rho^{1,R}$ with $\refrho{\rho}^{1,L}$ and $\refrho{\rho}^{1,R}$, respectively.  The law of the boundary-avoiding path on the right is simply that of an $\SLE_\kappa(\refrho{\rho}^{1,L};\refrho{\rho}^{1,R})$ process up until the first time that it exits the unit disk.  The idea of the proof of Proposition~\ref{prop::conditionalunique} is to show that the law described on the left side (assuming it exits) is absolutely continuous with respect to the law on the right side and we will explicitly identify the Radon-Nikodym derivative.}
\end{figure}

\begin{proposition}
\label{prop::conditionalunique}
Suppose that $\ul{\rho} = (\ul{\rho}^L;\ul{\rho}^R)$ is a collection of weights corresponding to force point locations $\ul{x} = (\ul{x}^L;\ul{x}^R)$.  Assume that $x^{1,L}=0^-$, $x^{1,R} = 0^+$, that both $\rho^{1,L} < \tfrac{\kappa}{2}-2$ and $\rho^{1,R} < \tfrac{\kappa}{2}-2$, and that both $\sum_{i=1}^k \rho^{i,L} \geq \tfrac{\kappa}{2}-2$ and $\sum_{i=1}^\ell \rho^{i,R} \geq \tfrac{\kappa}{2}-2$ where $k=|\ul{\rho}^L|$ and $\ell = |\ul{\rho}^R|$.  There is at most one law for a random continuous path $\eta$ in $\ol{\h}$ from $0$ to $\infty$ with the following Gibbs property.  For every almost surely positive stopping time $\tau$, the conditional law of $\eta$ given $\eta([0,\tau])$ is that of an ordinary $\SLE_\kappa(\ul{\rho})$ process in $\h \setminus \eta([0,\tau])$ from $\eta(\tau)$ to $\infty$ conditioned on the positive probability event that it reaches $\infty$ before hitting $I = [x^{2,L},x^{2,R}]$.
\end{proposition}

\begin{figure}[h!]
\begin{center}
\includegraphics[scale=0.85]{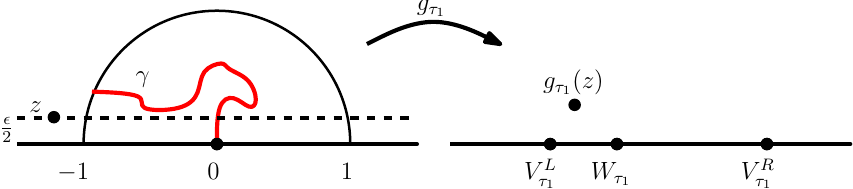}
\end{center}
\caption{\label{fig::radon_bound} Let $\Omega_\epsilon$ be the set of boundary avoiding, continuous, simple paths in $\h$ which connect $0$ to $\partial B(0,1)$, modulo reparameterization whose endpoint on $\partial B(0,1)$ has distance at least $\epsilon$ to $\partial \h$.  Fix $\gamma \in \Omega_\epsilon$.  Since $\gamma$ is simple, it admits a continuous Loewner driving function $W$ \cite[Proposition~6.12]{MS_IMAG}.  Let $(g_t)$ be the corresponding Loewner chain and let $V_t^L$ (resp.\ $V_t^R$) be the image of $0^-$ (resp.\ $0^+$) under $g_t$.  Then there exists positive and finite constants $C_\epsilon^1 \leq C_\epsilon^2$ which depend only on $\epsilon > 0$ such that $C_\epsilon^1 \leq |V_{\tau_1}^q - W_{\tau_1}| \leq C_\epsilon^2$ for $q \in \{L,R\}$ where $\tau_1$ is the first time that $\gamma$ hits $\partial B(0,1)$.  To see this, we let $z$ be the point on $\partial B(0,1+\epsilon)$ closest to $-1$ with $\im(z) = \epsilon/2$.  Then it is easy to see that there exists $p=p(\epsilon)> 0$ such that a Brownian motion starting at $z$ has probability at least $p(\epsilon)$ of exiting $\h \setminus \gamma$ in $(-\infty,0)$ (resp.\ the left side of $\gamma$, the right side of $\gamma$, or $(0,\infty)$).  Thus the conformal invariance of Brownian motion implies that the probability that a Brownian motion starting at $g_{\tau_1}(z)$ first exits $\h$ in $(-\infty,V_{\tau_1}^L)$ (resp.\ $(V_{\tau_1}^L,W_{\tau_1})$, $(W_{\tau_1},V_{\tau_1}^R)$, or $(V_{\tau_1}^R,\infty)$) is also at least $p$.  Since $g_{\tau_1}(z) = z + o(1)$ as $z \to \infty$, it follows from \cite[Theorem~3.20]{LAW05} that there exists finite and positive constants $D_\epsilon^1 \leq D_\epsilon^2$ depending only on $\epsilon > 0$ such that $D_\epsilon^1 \leq |g_{\tau_1}'(z)| \leq D_\epsilon^2$.  Thus \cite[Corollary 3.23]{LAW05} then implies that there exists a constant $E_\epsilon > 0$ depending only on $\epsilon > 0$ such that $\im(g_{\tau_1}(z)) \geq E_\epsilon$.  Combining everything implies the claim.
}
\end{figure}

We note that Proposition~\ref{prop::conditionalunique} is stated for $\SLE_\kappa(\ul{\rho})$ processes which can hit both immediately to the left and right of $0$ (or immediately hit the continuation threshold).  The same proof, however, works in the case that the curve can only hit on one side or is conditioned not to hit only one side.  Before we give its proof, we will need to prove several auxiliary lemmas; we recommend that the reader skips to the main body of the proof of Proposition~\ref{prop::conditionalunique} on a first reading to see how everything fits together.  Throughout, we let $\Omega$ be the set of simple, continuous, boundary avoiding  path segments in $\ol{\h}$ which connect $0$ to $\partial B(0,1)$, modulo reparameterization.  By \cite[Proposition~6.12]{MS_IMAG}, we know that if $\gamma \in \Omega$ then $\gamma$ (parameterized by capacity) admits a continuous Loewner driving function $W \colon [0,T] \to \R$.  We equip $\Omega$ with the following metric.  For $\gamma_i \in \Omega$ for $i=1,2$ with continuous Loewner driving functions $W_i \colon [0,T_i] \to \R$ we write
\[ d(\gamma_1,\gamma_2) = \| W_1(\cdot \wedge T_1) - W_2(\cdot \wedge T_2)\|_\infty + |T_1 - T_2|.\]
Then $(\Omega,d)$ is a metric space.

\begin{lemma}
\label{lem::distance_continuous}
Let $D \colon \Omega \to \R$ be the function which associates $\gamma \colon [0,T] \to \ol{\h}$ in $\Omega$ with the distance of its tip $\gamma(T)$ in $\partial B(0,1)$ to $\partial \h$, i.e.\ $\im(\gamma(T))$.  Then $D$ is a continuous function on $(\Omega,d)$.
\end{lemma}
\begin{proof}
Suppose that $(\gamma_n)$ is a sequence in $\Omega$ converging to $\gamma \in \Omega$ with respect to the topology induced by $d$.  Let $u_n$ be the tip of $\gamma_n$ in $\partial B(0,1)$ and $u$ the tip of $\gamma$.  We are actually going to show that $u_n$ converges to $u$.  By passing to a subsequence, we may assume without loss of generality that $u_n$ converges to $u_0 \in \partial B(0,1)$.  It suffices to show that $u_0 = u$.  For each $n$, let $g_n$ be the conformal map which takes $\h \setminus \gamma_n$ to $\h$ with $g_n(z) = z + o(1)$ as $z \to \infty$ and let $g$ be the corresponding conformal map for $\h \setminus \gamma$.  Then \cite[Proposition~4.43]{LAW05} implies that $g_n$ converges to $g$ in the Caratheodory sense; see \cite[Section~4.7]{LAW05}.  This means that for each compact set $K \subseteq \h \setminus \gamma$, there exists $n = n(K)$ such that $n \geq n(K)$ implies that $K \subseteq \h \setminus \gamma_n$ and $g_n \to g$ uniformly on $K$.  In particular, if $z \in \partial B(0,1)$ is not equal to $u$, there exists a compact neighborhood $K_z$ of $z$ in $\h \setminus \gamma$ which $\gamma_n$ does not intersect for all sufficiently large $n$.  Hence $u_0 \neq z$ and therefore $u_0 = u$, which completes the proof.
\end{proof}

For each $\epsilon > 0$, we let $\Omega_\epsilon = \{ \gamma \in \Omega : D(\gamma) \geq \epsilon\}$ be the set of continuous, boundary avoiding, simple paths in $\ol{\h}$ connecting $0$ to $\partial B(0,1)$ whose tip on $\partial B(0,1)$ has distance at least $\epsilon$ from $\partial \h$, modulo reparameterization.  Note that Lemma~\ref{lem::distance_continuous} implies that $\Omega_\epsilon$ is a closed subset of $(\Omega,d)$.

\begin{lemma}
\label{lem::images_of_zero}
Suppose that $\gamma \in \Omega$ and let $(g_t)$ be the corresponding Loewner chain.  Let $W$ be the Loewner driving function of $\gamma$ and $V_t^L,V_t^R$ be the images of $0^-,0^+$, respectively, under $g_t$.  There exists $0 < C_\epsilon^1 \leq C_\epsilon^2 < \infty$ depending only $\epsilon > 0$ such that
\[ C_\epsilon^1 \leq |V_{\tau_1}^q - W_{\tau_1}| \leq C_\epsilon^2 \quad\text{for}\quad q \in \{L,R\}\]
where $\tau_1$ is the first time $t$ that $\gamma$ exits $B(0,1)$.
\end{lemma}
\begin{proof}
Let $B$ be a Brownian motion starting at $z \in \partial B(0,1)$ with $\im(z) = \epsilon/2$.  It is easy to see that the probability that $B$ first exits $\h \setminus \gamma$ in $(-\infty,0)$ (resp.\ in $(0,\infty)$, the left side of $\gamma$, or the right side of $\gamma$) is at least a positive constant $p = p(\epsilon)$ which depends only on $\epsilon$.  Consequently, the desired result follows by the conformal invariance of Brownian motion as well as elementary derivative estimates \cite[Theorem~3.20 and Corollary 3.23]{LAW05} for conformal maps which imply that $g_{\tau_1}(z)$ is not too close to $\partial \h$ (see Figure~\ref{fig::radon_bound} for additional explanation).
\end{proof}

\begin{lemma}
\label{lem::prob_escape_infinity}
For each $\epsilon,M > 0$, there exists $p=p(\epsilon,M) > 0$ such that the following is true.  Let $\ul{\rho} = (\ul{\rho}^L;\ul{\rho}^R)$ be weights such that $|\sum_{i=1}^j \rho^{i,q}| \leq M$ for all $1 \leq j \leq |\ul{\rho}^q|$ and $\sum_i \rho^{i,q} \geq \tfrac{\kappa}{2}-2$ for $q \in \{L,R\}$.  Let $\ul{x} = (\ul{x}^L;\ul{x}^R)$ be a corresponding vector of force point locations in $\partial \h$ with $|x^{i,q}| \leq M$ for all $1 \leq i \leq |\ul{\rho}^q|$ and $q \in \{L,R\}$.  Let $P \colon \Omega \to [0,1]$ be the probability that an $\SLE_\kappa(\ul{\rho})$ process in $\h \setminus \gamma$ starting at the tip of $\gamma$ in $\partial B(0,1)$ with force points of weight $\ul{\rho}$ located at $\ul{x}$ makes it to $\infty$ without otherwise hitting $\partial \h$.  Then $P|_{\Omega_\epsilon} \geq p > 0$.
\end{lemma}
\begin{proof}
This can be seen by viewing $\eta$ as a flow line of a GFF on $\h \setminus \gamma$ and then invoking the absolute continuity properties of the GFF \cite[Proposition~3.4 and Remark~3.5]{MS_IMAG}.  In particular, if $\sigma_\delta$ is the first time $t$ that $\im(\eta(t)) = \delta$ for $\delta > 0$, then the law of $\eta|_{[0,\sigma_\delta]}$ is mutually absolutely continuous with respect to that of an $\SLE_\kappa(\ol{\rho}^L;\ol{\rho}^R)$ process with force points $\ol{\rho}^L,\ol{\rho}^R$ located at $0^-,0^+$, respectively, where $\ol{\rho}^q = \sum_i \rho^{i,q}$ for $q \in \{L,R\}$ stopped at the corresponding time.  Under the latter law, the probability of the event $\{\sigma_\delta = \infty\}$ converges to $1$ as $\delta \downarrow 0$ because such processes do not hit the boundary and are transient \cite[Theorem~1.3]{MS_IMAG}.
\end{proof}

Let $(\ol{\Omega},\ol{d})$ be the completion of the metric space $(\Omega,d)$.  Then $(\ol{\Omega},\ol{d})$ consists of all growth processes in $\ol{\h}$ targeted at $\infty$ which admit a continuous Loewner driving function \cite[Proposition~4.44]{LAW05} stopped upon first hitting $\partial B(0,1)$.

\begin{lemma}
\label{lem::weak_convergence}
Suppose that $\eta \in \Omega$.  For each $r \in (0,1)$, let $\tau_r$ be the first time that $\eta$ hits $\partial B(0,r)$ and let $\eta_r = \eta([0,\tau_r])$.  Let $\nu_r$ be the law of an $\SLE_\kappa(\rho^L;\rho^R)$ process in $\h \setminus \eta_r$ with $\rho^L,\rho^R > -2$ starting at the tip of $\eta_r$ and with force points located at $0^-$ and $0^+$ stopped upon exiting $B(0,1)$.  Let $\nu$ be the law of an $\SLE_\kappa(\rho^L;\rho^R)$ process in $\h$ from $0$ to $\infty$ with force points located at $0^-,0^+$, also stopped upon exiting $B(0,1)$.  Then $\nu_r$, viewed as a measure on $\ol{\Omega}$ equipped with the topology induced by $\ol{d}$, converges to $\nu$ weakly as $r \to 0$.
\end{lemma}
\begin{proof}
This is proved in \cite[Section~2]{MS_IMAG}.
\end{proof}

Lemma~\ref{lem::weak_convergence} holds more generally for $\SLE_\kappa(\ul{\rho})$ processes.  The reason that we introduced the space $\ol{\Omega}$ just before its statement was to allow for the possibility of boundary hitting curves, though in its application below we will only need to refer to the space $\Omega$.

\begin{proof}[Proof of Proposition~\ref{prop::conditionalunique}]
It suffices to show that the Gibbs property determines the law $\mu$ of $\eta([0,\tau])$ where $\tau$ is some almost surely positive stopping time (since the Gibbs property determines the law of $\eta|_{[\tau,\infty)}$ given $\eta([0,\tau])$).  By rescaling if necessary, we may assume without loss of generality that $[-1,1]$ is contained in the interior of $I$.  For concreteness, we take $\tau$ to be the first time that $\eta$ reaches $\partial B(0,1)$.  Fix $r \in (0,1)$, let $\tau_r$ be the first time that $\eta$ exits $B(0,r)$, and let $\eta_r = \eta([0,\tau_r])$.  Conditional on $\eta_r$, we let $\mu(\cdot;\eta_r)$ be the law of $\eta|_{[\tau_r,\tau]}$ given $\eta_r$.  Let $\refrho{\rho}^{1,L} = \kappa - 4 - \rho^{1,L} \geq \tfrac{\kappa}{2}-2$ and $\refrho{\rho}^{1,R} = \kappa-4-\rho^{1,R} \geq \tfrac{\kappa}{2}-2$ be the reflected values of $\rho^{1,L}$ and $\rho^{1,R}$, respectively.  We let $\nu(\cdot;\eta_r)$ be the law of an $\SLE_\kappa(\refrho{\rho}^{1,L};\refrho{\rho}^{1,R})$ process in $\h \setminus \eta_r$ starting at $\eta(\tau_r)$ with force points located at $0^-$ and $0^+$ stopped upon exiting $B(0,1)$.  We view $\mu(\cdot;\eta_r)$ and $\nu(\cdot;\eta_r)$ as probability measures on $\Omega$.

We are now going to compute $Z_r(\cdot) = d\mu(\cdot;\eta_r) / d\nu(\cdot;\eta_r)$, the Radon-Nikodym derivative of $\mu(\cdot;\eta_r)$ with respect to $\nu(\cdot;\eta_r)$.  We will do so by computing both $Z_r^\mu(\cdot) = d\mu(\cdot;\eta_r) / d\wt{\nu}(\cdot;\eta_r)$ and $Z_r^\nu(\cdot) = d\nu(\cdot;\eta_r) / d \wt{\nu}(\cdot;\eta_r)$ where $\wt{\nu}(\cdot;\eta_r)$ is the law of an $\SLE_\kappa(\rho^{1,L};\rho^{1,R})$ process in $\h \setminus \eta_r$ starting at $\eta(\tau_r)$ and stopped upon exiting $B(0,1)$, viewed as a probability on $\Omega$ by including the initial segment $\eta_r$.  Recall from Proposition~\ref{prop::mtweight} that we know how to weight an $\SLE_\kappa(\rho^{1,L}; \rho^{1,R})$ process by a martingale in a way that makes the path non-boundary intersecting and amounts to changing (one or both of the) $\rho^{1,q}$ to $\refrho{\rho}^{1,q}$.  That is, $Z_r^\nu = \tfrac{M_{\tau_1}}{M_{\tau_r}}$ where $M_t = M_t(\gamma)$ is the functional associated with the martingale described in~\eqref{eq::Mtdef} (since $\gamma \in \Omega$, it follows that the quantities $W_t,V_t^L,V_t^R$ in the definition of $M_t(\gamma)$ make sense for $\gamma$).  Note that $M_{\tau_r}(\gamma)$ depends only on $\gamma$ until it first hits $\partial B(0,r)$.  We can write $Z_r^\mu = c_r^\mu P \tfrac{F_{\tau_1}}{F_{\tau_r}}$ where
\begin{enumerate}
 \item $c_r^\mu$ is a normalizing constant,
 \item $F_t = F_t(\gamma)$ is the martingale which weights an $\SLE_\kappa(\rho^L;\rho^R)$ process in $\h$ starting at $0$ and targeted at $\infty$ with force points located at $0^-,0^+$ to yield the law of an $\SLE_\kappa(\ul{\rho}^{L};\ul{\rho}^{R})$ process in $\h$ with force points located at $\ul{x}$ \cite{SW05}, and
 \item $P = P(\gamma)$ is the probability that an $\SLE_\kappa(\ul{\rho}^L;\ul{\rho}^R)$ process in $\h \setminus \gamma$ starting at $\gamma(\tau_1)$ and with force points located at $\ul{x}$ makes it to $\infty$ without otherwise hitting $\partial \h$ (in particular, $P(\gamma) = 0$ if $\gamma$ hits $\partial \h$).
 \end{enumerate}
Therefore
\[ Z_r = c_r \wt{Z} \quad\text{where}\quad \wt{Z} = \frac{F_{\tau_1} P}{ M_{\tau_1}}\]
where $c_r$ is a normalizing constant which depends only on $\eta_r$.  We emphasize that $F_{\tau_1}$, $M_{\tau_1}$, and $P$ are functions $\Omega \to (0,\infty)$ which do not depend on $r$, so that $\wt{Z}$ does not depend on $r$.  Our goal now is to take a limit as $r \to 0$ to argue that $d\mu/d\nu = Z$ where $Z = c \wt{Z}$ for some constant $c > 0$.  In order to do so, we are going to argue that $\wt{Z}$ is continuous on $(\Omega,d)$, that $\wt{Z}$ is positive and bounded on $\Omega_\epsilon$, and that $c_r$ is bounded as $r \to 0$ (explained in the next three paragraphs).

We begin by showing that $\wt{Z}$ is a continuous function on $(\Omega,d)$ by observing that $M_{\tau_1}$, $F_{\tau_1}$, and $P$ are each individually continuous.  Indeed, it is clear that $\gamma \mapsto M_{\tau_1}(\gamma)$ is a continuous because Lemma~\ref{lem::x_differential} implies that $V_t^q$ for $q \in \{L,R\}$ is given by the Loewner flow associated with $\gamma$, which is a continuous function of the Loewner driving function $W$ of $\gamma$.  That $F_{\tau_1}$ is continuous follows from the same argument: all the terms in the formula for $F_{\tau_1}$ \cite{SW05} can all be expressed in terms of the Loewner flow, so continuously depend on $W$.  We now turn to argue that $P$ is continuous on $(\Omega,d)$.  Fix $\gamma \in \Omega$ and let $g$ be the conformal map which takes $\h \setminus \gamma$ to $\h$ with $g(z) = z + o(1)$ as $z \to \infty$.  Then $P(\gamma)$ is equal to the probability that an $\SLE_\kappa(\ul{\rho})$ process in $\h$ with force points located at $\ul{y} = g(\ul{x})$ makes it to $\infty$ without otherwise hitting $\partial \h$.  Let $W$ be the Loewner driving function of $\gamma$.  Then we know from \cite[Proposition~4.43]{MS_IMAG} that $g$ depends cotinuously on $W$ hence on $\gamma$, so we just need to argue that this probability depends continuously on $\ul{y}$.  This in turn can be seen through the coupling of $\SLE_\kappa(\ul{\rho})$ with the GFF: jiggling the force points $\ul{y}$ corresponds to jiggling the boundary data of the field, which affects its law continuously away from $\ul{y}$ (see \cite[Proposition~3.4 and Remark~3.5]{MS_IMAG} as well as Section~\ref{sec::multiple_force_points}).

Next, we are going to argue that there exists positive and finite constants $C_\epsilon^1 \leq C_\epsilon^2$ which only depend on $\epsilon > 0$ such that
\[ C_\epsilon^1 \leq \wt{Z}(\gamma) \leq C_\epsilon^2 \quad\text{for all}\quad \gamma \in \Omega_\epsilon.\]
Observe that on $\Omega_\epsilon$, $F_{\tau_1}$ is bounded between two positive and finite constants $C_\epsilon^1 \leq C_\epsilon^2$ (this is clear from the explicit formula for $F_{\tau_1}$ given in \cite{SW05}; it is also possible to see this using the GFF discussion in Section~\ref{sec::multiple_force_points}).  By possibly decreasing $C_\epsilon^1$ and increasing $C_\epsilon^2$, Lemma~\ref{lem::images_of_zero} implies that $M_{\tau_1}$ is bounded from below and above by $C_\epsilon^1$ and $C_\epsilon^2$, respectively, on $\Omega_\epsilon$.  Lemma~\ref{lem::prob_escape_infinity} gives the corresponding bound for $P$.  Hence, by possibly again decreasing $C_\epsilon^1$ and increasing $C_\epsilon^2$, the same holds for $\wt{Z}$ (in particular, the bounds do not depend on either $\eta_r$ or $r$).

Finally, we will show that $c_r$ has a limit $c$ as $r \to 0$ and simultaneously complete the proof.  Note that, for each $\epsilon > 0$ and $r \geq 0$, $\Omega_\epsilon$ is a continuity set for $\nu(\cdot;\eta_r)$.  This means that $\nu(\partial \Omega_\epsilon;\eta_r) = 0$; the reason for this is that $\partial \Omega_\epsilon$ consists of those paths in $\Omega$ which terminate at either of the two points in $\h \cap \partial B(0,1)$ with distance $\epsilon$ to $\partial \h$ and the probability that an $\SLE_\kappa(\ul{\rho})$ process exits $\h \cap \partial B(0,1)$ at a particular (interior) point is zero.  Thus since $\nu(\cdot;\eta_r) \to \nu(\cdot)$ weakly (Lemma~\ref{lem::weak_convergence}), the Portmanteau theorem implies that we have that $\nu(\Omega_\epsilon;\eta_r) \to \nu(\Omega_\epsilon)$ as $r \downarrow 0$.  Since $\nu(\Omega) = 1$ and $\Omega = \cup_{\epsilon > 0} \Omega_\epsilon$, it thus follows that for each $\delta > 0$ there exists $\epsilon_0 = \epsilon_0(\delta) > 0$ such that $\nu(\Omega_\epsilon;\eta_r) \geq 1-\delta$ for every $r \in (0,\tfrac{1}{2})$ and $\epsilon \in (0,\epsilon_0)$.  Let $D$ be the function as in Lemma~\ref{lem::distance_continuous} and let $G_\epsilon \colon \Omega \to [0,\infty)$ be the function which is equal to $0$ if $D(\gamma) \leq \epsilon$, equal to $1$ if $D(\gamma) \geq 2\epsilon$, and given by linearly interpolating between $0$ and $1$ for $D(\gamma) \in [\epsilon,2\epsilon]$.  Since $G_\epsilon$ can be expressed as a composition of a continuous function $[0,\infty) \to [0,1]$ with $D$, Lemma~\ref{lem::distance_continuous} implies that it is a bounded, continuous function on $(\Omega,d)$. Thus since
\begin{equation}
\label{eqn::rn_deriv}
 G_\epsilon d\mu(\cdot;\eta_r) = c_r \wt{Z} G_\epsilon d\nu(\cdot;\eta_r)
\end{equation}
it follows that
\[ c_r C_\epsilon^1(1-\delta) \leq \mu(\Omega_\epsilon;\eta_r) \leq 1\]
for each $\delta > 0$, all $r \in (0,\tfrac{1}{2})$, and all $\epsilon \in (0,\epsilon_0(\delta))$.  Therefore there exists $Y < \infty$ non-random which does not depend on $r$ such that $c_r \leq Y$ for all $r \in (0,\tfrac{1}{2})$.  Consequently, there exists a sequence $(r_k)$ in $(0,\tfrac{1}{2})$ so that $\lim_{k \to \infty} r_k = 0$ and $\lim_{k \to \infty} c_{r_k} = c \leq Y$, almost surely for $\mu$.  Using~\eqref{eqn::rn_deriv} and that $\nu(\cdot;\eta_r) \to \nu(\cdot)$ as $r \to 0$ (Lemma~\ref{lem::weak_convergence}), we have that $G_\epsilon d\mu = Z G_\epsilon d\nu$ where $Z = c \wt{Z}$.  Since this holds for any $\epsilon > 0$, we see that $d\mu/d\nu = Z$.  Therefore any measure satisfying the resampling property must be an $\SLE_\kappa(\refrho \rho^{1,L};\refrho \rho^{1,R})$ process weighted by $Z$, which completes the proof.
\end{proof}

\begin{remark}
\label{rem::cond_not_hit_existence}  We remark that the proof of Proposition~\ref{prop::conditionalunique} does not imply the existence of $\SLE_\kappa(\ul{\rho})$ conditioned not to hit the boundary.  Rather, the proof shows that \emph{if this process exists}, then the Radon-Nikodym derivative of its initial stub with respect to the law of the corresponding stub of an $\SLE_\kappa(\refrho{\rho}^{1,L};\refrho{\rho}^{1,R})$ process has to take a specific form.  Namely, it must be by $\tilde Z$, times a normalizing constant.  In order to establish existence, one must show that the $\nu$ expectation of $\tilde Z$ is finite, so that an appropriate choice of normalizing constant indeed produces a probability measure.
One way to establish this would be to show that the $c_r$ of~\eqref{eqn::rn_deriv} converges to a positive constant as $r \to 0$.
\end{remark}

\subsection{A pair of paths avoiding each other}

\begin{figure}[ht!]
\begin{center}
\includegraphics[scale=0.85]{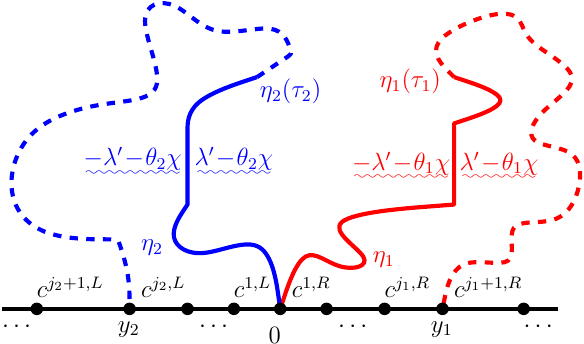}
\end{center}
\caption{\label{fig::two_paths_gibbs} Suppose that $y_2 < 0 < y_1$.  In Proposition~\ref{prop::conditionalunique2}, we prove that there exists at most one measure on pairs of non-intersecting paths $(\eta_1,\eta_2)$ in $\h$ where $\eta_i$ connects $0$ to $y_i$ for $i=1,2$ which satisfies the following Gibbs property.  Suppose that $\tau_i$ for $i=1,2$ is a positive and finite stopping time for $\eta_i$.  Then the conditional law of $(\eta_1,\eta_2)$ given $(\eta_1([0,\tau_1]), \eta_2([0,\tau_2]))$ is given by the flow lines of a GFF $h$ on $\h \setminus (\eta_1([0,\tau_1]) \cup \eta_2([0,\tau_2]))$ whose boundary data is as depicted above starting at $\eta_i(\tau_i)$ and with angle $\theta_i$, $i=1,2$, conditioned on the positive probability event that they do not hit each other and terminate at $y_1,y_2$, respectively.  We assume that $\theta_2-\theta_1 < 2\lambda'/\chi$ so that $\eta_1$ and $\eta_2$ can hit each other with positive probability (recall Figure~\ref{fig::hittingrange}).}
\end{figure}

\begin{figure}[ht!]
\begin{center}
\includegraphics[scale=0.85]{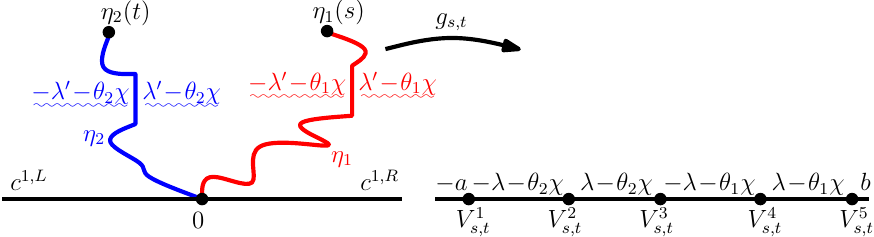}
\end{center}
\caption{\label{fig::two_path_martingale}
Suppose that $h$ is a GFF on $\h$ whose boundary data is as depicted in the left side (this can be viewed as the change of coordinates of a simplified version of the setup described in Figure~\ref{fig::two_paths_gibbs}).  For angles $\theta_1,\theta_2$, we let $\eta_i$ be the flow line of $h$ starting at $0$ with angle $\theta_i$, $i=1,2$.  Assume that $\theta_1,\theta_2$ are chosen so that $\eta_1$ can hit $\eta_2$, i.e.\ $\theta_2-\theta_1 < 2\lambda'/\chi$ (recall Figure~\ref{fig::hittingrange}).  Let $g_{s,t}$ be the conformal map which takes the unbounded connected component of $\h \setminus (\eta_1([0,s]) \cup \eta_2([0,t]))$ back to $\h$ satisfying $g_{s,t}(z) = z+o(1)$ as $z \to \infty$.  Let $V_{s,t}^i$ for $i \in \{1,\ldots,5\}$ be the image points shown.  Let $\rho_i$ be such that as one traces $\R$ from left to right, the heights in the right figure jump by $\rho_i \lambda$ at the points $V_{s,t}^i$.  Proposition~\ref{prop::mtweight3} implies that if we are given the paths in the figure up to some positive stopping times and then we generate their continuations as flow lines, then the product $M_{s,t} := \prod_{j \not = 3} |V_{s,t}^j - V_{s,t}^3|^{(\refrho{\rho}_3 - \rho_3)\rho_j }$ evolves as a martingale in each of $s$ and $t$ separately.  Moreover, Proposition~\ref{prop::mtweight3} implies that reweighting the law of the pair of paths $(\eta_1,\eta_2)$ by $M_{s,t}$ yields the law of a new pair of paths which do not intersect each other.  When one path is fixed the evolution of the other in the weighted law is the same as in the unweighted law except with $\rho_3$ replaced by $\refrho{\rho}_3=\kappa-4-\rho_3$.  This new pair of paths can be constructed as flow lines of a GFF with modified boundary data (see Figure~\ref{fig::two_path_martingale}).
}
\end{figure}

\begin{figure}[ht!]
\begin{center}
\includegraphics[scale=0.85]{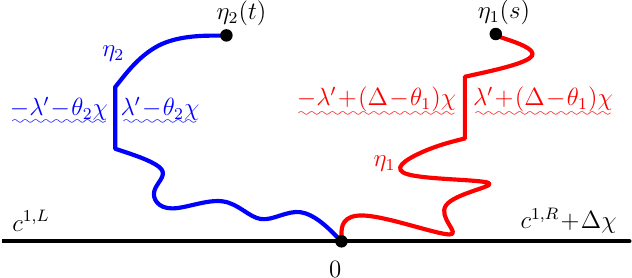}
\end{center}
\caption{\label{fig::two_path_gff_rep} (Continuation of Figure~\ref{fig::two_path_martingale}.)  Here, we will describe the law of the pair of paths from Figure~\ref{fig::two_path_martingale} after reweighting by $M_{s,t}$ as GFF flow lines.  Suppose that $h$ is a GFF on $\h$ whose boundary data is as depicted above.  Let $\eta_1$ be the flow line of $h$ with angle $\theta_1-\Delta$ and let $\eta_2$ be the flow line of $h$ with angle $\theta_2$, both starting from $0$.  Assume, as in Figure~\ref{fig::two_path_martingale}, that $\theta_2-\theta_1 < 2\lambda'/\chi$.  We choose $\Delta$ as follows.  First, we take $\rho_3 \lambda = (\theta_2-\theta_1)\chi - 2\lambda$ (recall the right side of Figure~\ref{fig::two_path_martingale}), then take $\refrho{\rho}_3 = \kappa-4-\rho_3$ so that $\refrho{\rho}_3\lambda = -2\pi\chi - \rho_3\lambda$, and then set $\Delta = (\refrho{\rho}_3-\rho_3)\lambda/\chi$.  Explicitly, $\Delta = 4\lambda/\chi+ 2(\theta_1-\theta_2-\pi)$.  Then the law of the pair $(\eta_1,\eta_2)$ is equal to the law of the corresponding pair from Figure~\ref{fig::two_path_martingale} after weighting by $M_{s,t}$ (and taking a limit as the initial segments tend to $0$).}
\end{figure}

\begin{figure}[ht!]
\begin{center}
\includegraphics[scale=0.85]{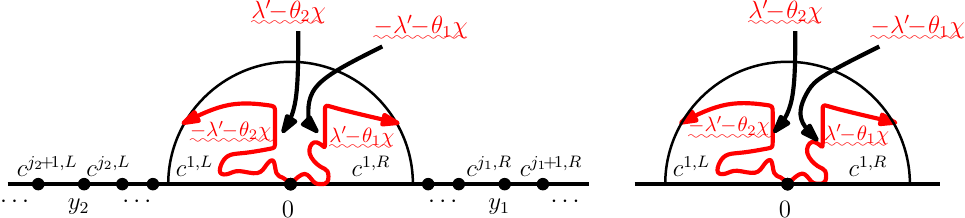}
\end{center}
\caption{\label{fig::two_half_discs2}
The strategy to prove Proposition~\ref{prop::conditionalunique2} is similar to that used to prove Proposition~\ref{prop::conditionalunique}.  In particular, we will show that the law of the pair of paths $(\eta_1,\eta_2)$ stopped at times $\tau_1,\tau_2$ (left side) is mutually absolutely continuous with respect to the another pair $(\wt{\eta}_1,\wt{\eta}_2)$ stopped at the corresponding times (right side) and simultaneously identify the Radon-Nikodym derivative explicitly.  For concreteness, we take $\tau_1,\tau_2$ to be the first time that $\eta_1,\eta_2$ exit $B(0,1)$, respectively.  We take $\wt{\eta}_1,\wt{\eta}_2$ to be the flow lines of the GFF whose boundary data is on the right side with angles $\theta_1,\theta_2$, respectively, conditioned to avoid each other (see Figure~\ref{fig::two_path_martingale} and Figure~\ref{fig::two_path_gff_rep}).  In the right figure, given initial segments of the paths we can accomplish the conditioning by reweighing by the martingale $M$ described in Figure~\ref{fig::two_path_martingale}.  In analogy with the proof of Proposition~\ref{prop::conditionalunique}, this means that the Radon-Nikodym derivative takes the form $FP/M$ where $F$ accounts for the change in boundary data of the GFFs and $P$ is the conditional probability that the paths make it to their target point without touching each other.}
\end{figure}

We will now present Proposition~\ref{prop::conditionalunique2}, which is a version of Proposition~\ref{prop::conditionalunique} but instead of having one path that we condition not to hit the boundary, we have two paths that we condition not to hit each other (see Figure~\ref{fig::two_paths_gibbs}).  The proof is very similar to that of Proposition~\ref{prop::conditionalunique}.  Since we will reference a modification of this result at the end of Section~\ref{sec::multiple_force_points}, we are careful to spell out the necessary modifications.

\begin{proposition}
\label{prop::conditionalunique2}
Fix points $\ul{x} = (\ul{x}^L;\ul{x}^R)$ in $\partial \h$ and constants $\ul{c} = (\ul{c}^L;\ul{c}^R)$.  We assume that $x^{1,L} = 0^-$, $x^{1,R} = 0^+$, $|\ul{x}^L| = |\ul{c}^L|$, and $|\ul{x}^R| = |\ul{c}^R|$.  Suppose that $y_1 = x^{j_1+1,R}$ and $y_2 = x^{j_2+1,L}$ for some $0 \leq j_1 \leq |\ul{x}^R|-1$ and $0 \leq j_2 \leq |\ul{x}^L|-1$.  There exists at most one measure on pairs of non-intersecting paths $(\eta_1,\eta_2)$ where $\eta_i$ connects $0$ to $y_i$ for $i=1,2$ such that for every pair of positive stopping times $\tau_1,\tau_2$ for $\eta_1,\eta_2$, respectively, the following is true.  The law of the pair $(\eta_1,\eta_2)$ conditional on $\eta_1([0,\tau_1])$ and $\eta_2([0,\tau_2])$ is given by the flow lines of the GFF $h$ on $\h \setminus (\eta_1([0,\tau_1]) \cup \eta_2([0,\tau_2]))$ starting at $\eta_i(\tau_i)$, $i=1,2$, with angles $\theta_i$ satisfying $\theta_2-\theta_1 < 2\lambda'/\chi$, whose boundary data is as depicted in Figure~\ref{fig::two_paths_gibbs} conditioned on the positive probability event that they do not hit each other and terminate at $y_1$, $y_2$, respectively.
\end{proposition}

We note that Proposition~\ref{prop::conditionalunique2} only has content in the case that the boundary data of $h$ is such that conditional on the initial segments of $\eta_1,\eta_2$, the corresponding flow lines of $h$ terminate at $y_1$ and $y_2$ with positive probability.  This translates into the following requirements (see Figure~\ref{fig::hittingrange}, Figure~\ref{fig::hittingsinglepoint}, and~\eqref{eqn::ac_eq_rel}):
\begin{enumerate}
\item $c^{j_2,L} + \theta_2 \chi < \lambda$ ($\eta_2$ does not hit the continuation threshold upon hitting $[x^{j_2,L},x^{j_2-1,L})$),
\item $c^{j_2+1,L} + \theta_2 \chi > -\lambda + 2\pi \chi$ (upon hitting $[x^{j_2+1,L},x^{j_2,L})$, $\eta_2$ can be continued towards $y_2$)
\item $c^{j_1,R} + \theta_1 \chi > -\lambda$ ($\eta_1$ does not hit the continuation threshold upon hitting $[x^{j_1-1,R},x^{j_1,R})$), and
\item $c^{j_1+1,R} + \theta_1 \chi < \lambda - 2\pi \chi$ (upon hitting $[x^{j_1,R},x^{j_1+1,R})$, $\eta_1$ can be continued towards $y_1$)
\end{enumerate}

\begin{proof}[Proof of Proposition~\ref{prop::conditionalunique2}]
We let $\Gamma$ be the set which consists of the pairs $(\gamma_1,\gamma_2)$ where $\gamma_i$ is a simple, continuous, path in $\ol{\h}$ which connects $0$ to $\partial B(0,1)$ for $i=1,2$, modulo reparameterization.  We also assume that $\gamma_1$ intersects $\gamma_2$ only at $0$ (though $\gamma_i$ may hit $\partial \h$, in contrast to the space $\Omega$ considered in the previous subsection).  We note that $\Gamma \subseteq \ol{\Omega}^2$ where $\ol{\Omega}$ is the completion of $\Omega$ under $d$, as in the previous subsection; we equip $\Gamma$ with the product metric induced by $\ol{\Omega}^2$, which we will also denote by $d$.  Let $\mu$ be a law on pairs $(\eta_1,\eta_2) \in \Gamma$ as described in the statement of the proposition, where the paths are stopped upon exiting $B(0,1)$, and let $\nu$ be the measure on pairs of paths as described in Figure~\ref{fig::two_path_gff_rep} (or Figure~\ref{fig::two_path_martingale} after weighting by the martingale), also stopped upon exiting $B(0,1)$.  As in the proof of Proposition~\ref{prop::conditionalunique}, we are going to prove the result by explicitly showing that $\mu$ is absolutely continuous with respect to $\nu$ and, at the same time, explicitly identifying the Radon-Nikodym derivative $d\mu/d\nu$.

Let $D \colon \Gamma \to \R$ be the function where $D(\gamma_1,\gamma_2)$ is the distance between the tips of the paths $\gamma_1,\gamma_2$ in $\partial B(0,1)$.  Then, using the same proof as in Lemma~\ref{lem::distance_continuous} (uniform convergence of Loewner driving functions implies the convergence of the tips of the paths on $\partial B(0,1)$), we see that $D$ is a continuous function on $(\Gamma,d)$.  Consequently, for each $\epsilon >0$ the set $\Gamma_\epsilon = \{ (\gamma_1,\gamma_2) \in \Gamma : D(\gamma_1,\gamma_2) \geq \epsilon\}$ is a closed subset of $\Gamma$ and $\partial \Gamma_\epsilon$ consists of the set of pairs of paths whose tips on $\partial B(0,1)$ have distance exactly $\epsilon$.  Let $M_{s,t} = M_{s,t}(\gamma_1,\gamma_2)$ be the functional on pairs in $\Gamma$ as defined in Figure~\ref{fig::two_path_martingale}.  The same proof as Lemma~\ref{lem::images_of_zero} implies that there exists positive and finite constants $C_\epsilon^1 \leq C_\epsilon^2$ such that $C_\epsilon^1 \leq M_{\tau_{1,1},\tau_{1,2}}(\gamma_1,\gamma_2) \leq C_\epsilon^2$ for all $(\gamma_1,\gamma_2) \in \Gamma_\epsilon$ where $\tau_{r,i}$ is the first time that $\gamma_i$ hits $\partial B(0,r)$ for $r > 0$ (simple Brownian motion estimates imply that the image points under $g_{\tau_{1,1},\tau_{1,2}}$ have uniformly positive and finite distance from each other for each fixed $\epsilon > 0$).  We let $P \colon \Gamma \to [0,1]$ be the probability that the pair of flow lines $(\eta_1,\eta_2)$ of a GFF in $\h \setminus (\gamma_1 \cup \gamma_2)$ with boundary data as described in the statement of the proposition starting at the tips of $\gamma_1,\gamma_2$ with angles $\theta_1,\theta_2$, respectively, terminate at $y_1,y_2$ without hitting each other.  Then the same proof as Lemma~\ref{lem::prob_escape_infinity} implies that there exists $p > 0$ which depends only on $\epsilon > 0$ and $\ul{x}$ such that $P|_{\Gamma_\epsilon} \geq p$.  Finally, we define $F_{s,t} = F_{s,t}(\gamma_1,\gamma_2)$ to be the two-sided martingale which reweights a pair $(\eta_1,\eta_2)$ of flow lines of a GFF on $\h$ with boundary data which is constant $c^{1,L}$ to the left of $0$ and constant $c^{1,R}$ to the right of $0$ with angles $\theta_1,\theta_2$ (as in the statement of the proposition) to the corresponding pair of flow lines of a GFF whose boundary data is as in the statement of the proposition (see Figure~\ref{fig::two_path_martingale} for a similar martingale).  By adjusting $C_\epsilon^1 \leq C_\epsilon^2$ if necessary, we also have that $C_\epsilon^1 \leq F_{\tau_{1,1},\tau_{1,2}} \leq C_\epsilon^2$ on $\Gamma_\epsilon$ (this can be seen from the explicit form of $F$; it can also be seen via the GFF as in Section~\ref{sec::multiple_force_points}).

Assume that $(\eta_1,\eta_2) \sim \mu$.  Let $\eta_{r,i} = \eta_i([0,\tau_{r,i}])$ where $\tau_{r,i}$ is as defined in the previous paragraph.  Let $\mu(\cdot;\eta_{r,1},\eta_{r,2})$ denote the conditional law of $(\eta_1,\eta_2)$ stopped upon exiting $B(0,1)$ given $\eta_{r,1}$ and $\eta_{r,2}$.  Let $\nu(\cdot;\eta_{r,1},\eta_{r,2})$ denote the law of the pair of flow lines as described in the right side of Figure~\ref{fig::two_half_discs2} given that their initial segments up to hitting $\partial B(0,r)$ are given by $\eta_{r,1},\eta_{r,2}$, respectively (i.e., conditioned not to hit by weighting by the martingale $M_{s,t}$).  The same argument as was used in the proof of Proposition~\ref{prop::conditionalunique} implies that
\[ Z_r := \frac{d\mu(\cdot;\eta_{r,1},\eta_{r,2})}{d\nu(\cdot;\eta_{r,1},\eta_{r,2})} = c_r \wt{Z} \quad\text{where}\quad \wt{Z} = \frac{F_{\tau_{1,1},\tau_{1,2}} P}{M_{\tau_{1,1},\tau_{1,2}}}\]
where $c_r$ is a normalizing constant.  The reference measure $\wt{\nu}(\cdot;\eta_{r,1},\eta_{r,2})$ that we use in this computation is given by the flow lines of a GFF on $\h \setminus (\eta_{r,1} \cup \eta_{r,2})$ with constant boundary data $c^{1,L},c^{1,R}$ to the left and right of $0$, respectively, stopped upon exiting $\partial B(0,1)$ though without any conditioning (see the right side of Figure~\ref{fig::two_half_discs2}).

As in the proof of Lemma~\ref{lem::weak_convergence}, the results of \cite[Section~2]{MS_IMAG} also imply that $\nu(\cdot;\eta_{r,1},\eta_{r,2}) \to \nu(\cdot)$ weakly as $r \downarrow 0$.  The reason is that these results imply that we have the desired weak convergence for the law of $\eta_1$ given $(\eta_{r,1},\eta_{r,2})$ as $r \to 0$ as well as the weak convergence of the conditional law of $\eta_2$ given $\eta_1$ and $\eta_{r,2}$ as $r \to 0$.  Since $\Gamma_\epsilon$ is a continuity set for $\nu(\cdot;\eta_{r,1},\eta_{r,2})$ for all $r \in (0,1)$ as well as for $\nu(\cdot)$ (conditional on $\eta_1([0,\tau_{1,1}])$, the probability that $\eta_2$ exits $\partial B(0,1) \cap \h$ at a point with distance exactly $\epsilon$ to $\eta_1(\tau_{1,1})$ is zero), we know, using the same argument as in the proof of Proposition~\ref{prop::conditionalunique}, that $c_r \leq Y$ for $Y < \infty$ non-random and all $r \in (0,\tfrac{1}{2})$.  The rest of the argument of Proposition~\ref{prop::conditionalunique} goes through verbatim.
\end{proof}

\begin{remark}
\label{rem::conditionalunique2_interval_exit}
The statement of Proposition~\ref{prop::conditionalunique2} also holds (by the same proof) in the case that $\eta_1$ and $\eta_2$ are conditioned to exit $\h$ in intervals $I_1$ and $I_2$, respectively.  The same result holds if we replace one (or both) of the $\eta_i$ with a counterflow line.  We will make use of this in Section~\ref{sec::multiple_force_points}.
\end{remark}

\subsection{Resampling properties and dual flow lines}

\begin{figure}[h!]
\begin{center}
\includegraphics[scale=0.85]{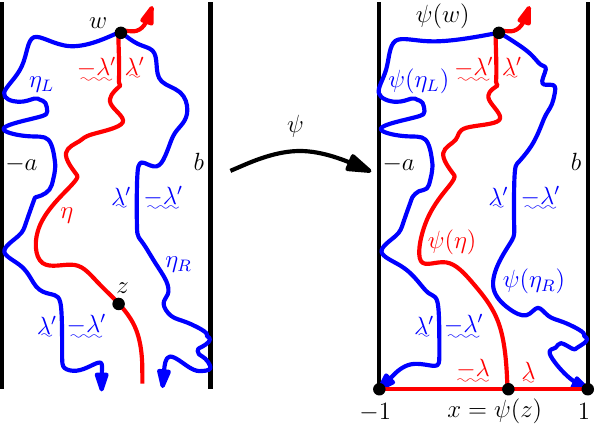}
\end{center}
\caption{\label{fig::triplepath}
Assume $a,b \in (-\lambda'-\pi \chi,\lambda')$; this is precisely the range of boundary data so that a flow line of a GFF with the boundary data above can hit both the left and right sides of the vertical strip $\vstrip$ (recall Figure~\ref{fig::hittingrange}).  In the left panel, $\eta$ is an $\SLE_\kappa(\refrho{\rho}^L;\refrho{\rho}^R)$ process from the bottom to the top of $\vstrip$ interpreted as a flow line of a GFF with the given boundary data \emph{conditioned} not to hit $\partial \vstrip$.  Let $\tau$ be a stopping time for $\eta$, $z = \eta(\tau)$, $\tau'$ a reverse stopping time for $\eta|_{[\tau,\infty)}$, and $w = \eta(\tau')$.  Given $\eta$, we let $\eta_L,\eta_R$ be flow lines from $w$ of a GFF in $\vstrip \setminus \eta$ with the boundary conditions as shown with angles $\pi$, $-\pi$, respectively.  We prove in Proposition~\ref{prop::middlepath} that given $\eta_L$, $\eta_R$, $\eta([0,\tau])$, and $\eta([\tau',\infty))$ the conditional law of $\eta$ is that of an $\SLE_\kappa(\tfrac{\kappa}{2}-2;\tfrac{\kappa}{2}-2)$ process in the region of $\vstrip \setminus \eta([0,\tau])$ which lies between $\eta_L$ and $\eta_R$.  We can resample each of the blue paths given both the other and $\eta|_{[\tau',\infty)}$ (and, as we just mentioned, we can resample the red path given both blue paths).  The same result holds when $\tau'=\infty$ so that the dual flow lines emanate from $\infty$ and/or when $\tau=0$ so that $z=-\infty$.  We will explain in the proof of Theorem~\ref{thm::usual_sle_reversible} that this resampling property characterizes the joint law of these paths.  In order to avoid mentioning $\tau$ and $\eta([0,\tau])$, throughout this section we will work on the half infinite vertical strip $\vhstrip$ by applying the conformal change of coordinates $\psi$ which takes $\vstrip \setminus \eta([0,\tau])$ to the half-infinite vertical strip $\vhstrip$ where $+\infty$ is fixed and the left and right sides of $-\infty$ are sent to $-1$ and $+1$, respectively; $x = \psi(z) \in (-1,1)$.}
\end{figure}

The purpose of this subsection is to establish another type of resampling result.  Throughout, we will often use the half-infinite vertical strip $\vhstrip = [-1,1] \times \R_+$ as our ambient domain.  We will write $\vhstripleft = -1 + i\R_+$ for its left boundary and $\vhstripright = 1 + i\R_+$ for its right boundary.  We let $\eta$ be an $\SLE_\kappa(\refrho{\rho}^L;\refrho{\rho}^R)$ process with $\refrho{\rho}^L,\refrho{\rho}^R > \tfrac{\kappa}{2}-2$ in $\vhstrip$ starting at $x \in (-1,1)$.  By Proposition~\ref{prop::conditionalunique}, we can view $\eta$ as a flow line of a GFF on $\vhstrip$ whose boundary data is as depicted in Figure~\ref{fig::local_dual_flow} with $a,b \in (-\lambda'-\pi \chi,\lambda')$ \emph{conditioned} not to hit the boundary.  Let $\tau'$ be any reverse stopping time for $\eta$, i.e.\ a stopping time for the filtration $\CF_t^\CR = \sigma(\eta(s) : s \geq t)$.  The main result of this subsection is Proposition~\ref{prop::middlepath} (stated in the slightly more general setting of an $\SLE_\kappa(\refrho{\rho}^L;\refrho{\rho}^R)$ process on the vertical strip $\vhstrip$), which gives us a way to resample the initial segment $\eta|_{[0,\tau']}$ of $\eta$.

\begin{proposition}
\label{prop::middlepath}
Suppose that we sample $\eta$ as an $\SLE_\kappa(\refrho{\rho}^L;\refrho{\rho}^R)$ process in the vertical strip $\vstrip$ from $-\infty$ to $\infty$, interpreted as a flow line of the GFF on $\vstrip$ with boundary conditions $-a$ and $b$ for $a,b \in (-\lambda'-\pi\chi, \lambda')$ conditioned not to hit the left and right boundaries $\vstripleft,\vstripright$, as in the left side of Figure~\ref{fig::triplepath}.  Let $\tau$ be a forward stopping for $\eta$ and let $\tau'$ be a reverse stopping time for $\eta|_{[\tau,\infty)}$.  Then, conditioned on $\eta$, we sample $\eta_L$ and $\eta_R$ as flow lines of a GFF in $\vstrip \setminus \eta$ starting from $\eta(\tau')$ targeted at $-\infty$ with the boundary conditions as shown in the left side of Figure~\ref{fig::triplepath} (the angles of $\eta_L,\eta_R$ are $\pm \pi$).  Given $\eta_L$, $\eta_R$, and $\eta([0,\tau])$, the conditional law of $\eta$ is that of an $\SLE_\kappa(\tfrac{\kappa}{2}-2;\tfrac{\kappa}{2}-2)$ process starting at $\eta(\tau)$ and targeted at $\eta(\tau')$ in the region of $\strip \setminus (\eta_L \cup \eta_R \cup \eta([0,\tau])$ bounded between $\eta_L$ and $\eta_R$ with force points at the left and right sides of $-\infty$.  The same statement also holds if we take $\tau' = \infty$, which corresponds to drawing the dual flow lines $\eta_L$ and $\eta_R$ all the way from the top to the bottom of $\vstrip$ (one on either side of $\eta$), and/or we take $\tau=0$.
\end{proposition}

In the setting of Proposition~\ref{prop::middlepath}, once we have fixed a forward stopping time $\tau$, we can apply a conformal map $\psi \colon \vstrip \setminus \eta([0,\tau]) \to \vhstrip$ which takes the left and right sides of $-\infty$ to $\pm 1$, respectively, and fixes $+\infty$.  Then $\psi(\eta|_{[\tau,\infty)})$ is an $\SLE_\kappa(\refrho{\rho}^L;\refrho{\rho}^R)$ process in $\vhstrip$ starting from $x = \psi(\eta(\tau)) \in (-1,1)$ to $+\infty$, as is depicted on the right side of Figure~\ref{fig::triplepath}.  This perspective has the notational advantage that we do not need to refer to the stopping time $\tau$ or the initial segment of $\eta$.  For this reason, we will work on $\vhstrip$ throughout the rest of this section and think of $\eta$ as an $\SLE_\kappa(\refrho{\rho}^L;\refrho{\rho}^R)$ process on $\vhstrip$ from a point $x \in (-1,1)$ to $\infty$ with force points located at $\pm 1$.

We remark that if we instead sample $\eta$ as an ordinary flow line (so that we are not conditioning $\eta$ to be boundary avoiding), then due to the choice of boundary data in Proposition~\ref{prop::middlepath}, the dual flow lines $\eta_L$ and $\eta_R$ may intersect both $\vhstripleft$ and $\vhstripright$, but almost surely terminate upon reaching $\pm 1$ (see Figure~\ref{fig::hittingrange}).  The proof of Proposition~\ref{prop::middlepath} has two steps.  The first is Lemma~\ref{lem::local_dual_flow} (see Figure~\ref{fig::local_dual_flow}), which proves a statement analogous to Proposition~\ref{prop::middlepath} except for $\eta \sim \SLE_\kappa(\rho^L;\rho^R)$ rather than $\eta \sim \SLE_\kappa(\refrho{\rho}^L;\refrho{\rho}^R)$ (recall that $\rho,\refrho{\rho}$ are related to each other by $\refrho{\rho} = \kappa-4-\rho$ and $\refrho{\rho} \geq \tfrac{\kappa}{2}-2$ corresponds to the non-boundary intersecting case).  This can be thought of as the analogous result but for the \emph{unconditioned path}.  The bulk of the remainder of the proof is contained in Lemmas~\ref{lem::dualflow_limit}--\ref{lem::resampling_martingale}, which we will show ultimately imply that the result of Lemma~\ref{lem::local_dual_flow} holds even when we condition $\eta$ not to hit the boundary (in the sense of Proposition~\ref{prop::conditionalunique}).

\begin{figure}[ht!]
\begin{center}
\includegraphics[scale=0.85]{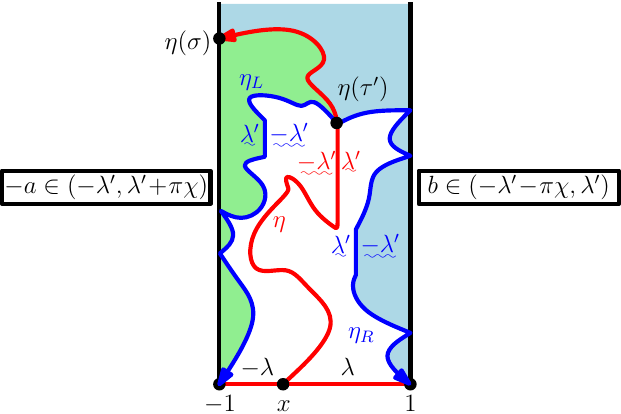}
\end{center}
\caption{\label{fig::local_dual_flow} Suppose that $h$ is a GFF on the half-infinite vertical strip $\vhstrip$ with boundary data as shown above.  Let $\eta$ be the flow line of $h$ from $x$ targeted at $\infty$.  Then $\eta$ is an $\SLE_\kappa(\rho^L;\rho^R)$ with force points at $-1$ and $1$ which almost surely hits either the left or the right side of the strip ($\vhstripleft$, $\vhstripright$, respectively).  Let $\sigma$ be the first time that this occurs, let $\tau'$ be any reverse stopping time for $\eta|_{[0,\sigma]}$, and let $\eta_L,\eta_R$ be the left and right dual flow lines of $h$ which start at $\eta(\tau')$ (i.e., flowing at angles $\pi$ and $-\pi$, respectively) and are targeted at $-1$ and $1$, respectively.  Due to the choice of boundary data, $\eta_L,\eta_R$ almost surely reach $\pm 1$, although they may bounce off both $\vhstripleft$ and $\vhstripright$.  We prove in Lemma~\ref{lem::local_dual_flow} that the conditional law of $\eta$ given $\eta_L$ and $\eta_R$ is that of an $\SLE_\kappa(\tfrac{\kappa}{2}-2;\tfrac{\kappa}{2}-2)$ process targeted at $\eta(\tau')$ in the connected component of $\vhstrip \setminus (\eta_L \cup \eta_R)$ which contains $x$.  The idea of the proof is that $\eta_L$ can be realized as the left outer boundary of a counterflow line $\eta_L'$ of $h$ starting from $-1$ and targeted at $x$.  That is, if $\tau_L'$ is the first time that $\eta_L'$ hits $\eta(\tau')$, then $\eta_L$ is equal to the segment of the boundary of the component of $\vhstrip \setminus \eta_L'([0,\tau_L'])$ with $x$ on its boundary which runs from $-1$ to $\eta'(\tau')$ in the clockwise direction.  (The hull of $\eta_L'([0,\tau_L'])$ viewed as a path targeted at $x$ is shaded green in the illustration.)  Likewise, we can view $\eta_R$ as the right outer boundary of a counterflow line $\eta_R'$ of $h$ starting from $1$ and targeted at $x$ stopped at the first time $\tau_R'$ that it hits $\eta(\tau')$.  (The hull of $\eta_R'([0,\tau_R'])$ viewed as a path targeted at $x$ is shaded in light blue in the illustration.)  The common segment of the outer boundaries of $\eta_L',\eta_R'$ is equal to $\eta$ (recall Figure~\ref{fig::counterflowline}).}
\end{figure}

\begin{lemma}
\label{lem::local_dual_flow}
Suppose that $h$ is a GFF on $\vhstrip$ with boundary data as in Figure~\ref{fig::local_dual_flow} with $a,b \in (-\lambda'-\pi \chi,\lambda')$.  Recall that the flow line~$\eta$ of~$h$ starting from~$x$ is an $\SLE_\kappa(\rho^L;\rho^R)$ process with force points located at~$-1$ and~$1$, respectively.  Moreover, $\eta$ almost surely intersects either $\vhstripleft$ or $\vhstripright$; let $\sigma$ be the time of the first such intersection.  Let $\tau'$ be any reverse stopping time for $\eta|_{[0,\sigma]}$, i.e.\ a stopping time for the filtration $\CF_t^\CR = \sigma(\eta(s) : s \in [t,\sigma])$, and let $\eta_L$ and $\eta_R$ be the left and right dual flow lines of $h$ starting at $\eta(\tau')$.  The conditional law of $\eta|_{[0,\tau']}$ given $\eta_L$, $\eta_R$, and $\eta|_{[\tau',\infty)}$ is an $\SLE_\kappa(\tfrac{\kappa}{2}-2;\tfrac{\kappa}{2}-2)$ process in the connected component of $\vhstrip \setminus (\eta_L \cup \eta_R)$ which contains $x$.
\end{lemma}

Recall Figure~\ref{fig::counterflowline}, which tells us that a flow line of a GFF $h$ from $x$ to $y$ with angle $-\tfrac{\pi}{2}$ is equal to the right boundary of the counterflow line of the same GFF starting at $y$.  An equivalent formulation of this fact is that the zero angle flow line is equal to the right boundary of the counterflow line $\eta_L'$ of $h+\tfrac{\pi}{2}\chi$.  Moreover, the left boundary of this counterflow line is the flow line with angle $\pi$ --- the left dual flow line emanating from $x$.  Analogously, the zero angle flow line is equal to the left boundary of the counterflow line $\eta_R'$ of $h-\tfrac{\pi}{2}$ and the right boundary of this counterflow line is a flow line with angle $-\pi$.  This means that $\eta_L'$ and $\eta_R'$ together determine both $\eta$ as well as the dual flow lines emanating from $x$, the seed of $\eta$.  The idea of the proof of Lemma~\ref{lem::local_dual_flow} is to extend this one step further, to show that we can generate $\eta$ up to any reverse stopping time $\tau'$ along with the dual flow lines $\eta_L$ and $\eta_R$ emanating from $\eta(\tau')$ by stopping $\eta_L',\eta_R'$ at appropriate times and showing that the result seeing $K$ is local for $h$.  Then we can quote the results of \cite{MS_IMAG} (which are summarized in Section~\ref{sec::preliminaries}), which allow us to compute the conditional law of $\eta|_{[0,\tau']}$ given $K$ and $h|_K$ (this corresponds to conditioning on a larger $\sigma$-algebra than the one generated by just $\eta|_{[\tau',\infty)}$, $\eta_L$, and $\eta_R$).

One particularly interesting aspect of the proof is the following.  Even though the coupling of $\SLE_\kappa(\ul{\rho})$ with the GFF is non-reversible in the sense that reparameterizing the path in the opposite direction does not yield a flow line of the field (recall Section~\ref{subsec::naive_time_reversal}), it is still nevertheless possible to generate a flow line in the reverse direction by looking at the values of the field in increasing neighborhoods of its terminal point (unlike when we generate the path in the forward direction, we need to see values of the field which lie off the path).

\begin{proof}[Proof of Lemma~\ref{lem::local_dual_flow}]
See Figure~\ref{fig::local_dual_flow} for an illustration of the setup of the proof (recall also Figure~\ref{fig::counterflowline}).  Let $\eta_L'$ be the counterflow line of $h-\tfrac{\pi}{2} \chi$ starting at $-1$ and let $\eta_R'$ be the counterflow line of $h + \tfrac{\pi}{2} \chi$ starting at $1$.  We assume that both $\eta_L'$ and $\eta_R'$ are targeted at~$x$.

Both $\eta_L'$ and $\eta_R'$ are $\SLE_{\kappa'}(\ul{\rho})$ processes.  For the convenience of the reader (these exact values are not important for the rest of the proof), we are now going to determine the $\rho$ values for $\eta_L'$; one computes the $\rho$ values for $\eta_R'$ in an analogous manner.  We begin by mapping $\vhstrip$ to $\h$ (with $-1$ sent to $0$, $x$ sent to $\infty$, $1$ sent to $-1$) via the conformal transformation $\psi$ and then apply~\eqref{eqn::ac_eq_rel} to $h$.  Note that $\psi(\infty) \in (-1,0)$.  We then obtain a GFF on~$\h$ with boundary conditions given by:
\begin{itemize}
\item $-a+\tfrac{\pi \chi}{2}$ on $(\psi(\infty),0]$,
\item $b+\tfrac{3\pi \chi}{2}$ on $(-1,\psi(\infty)]$,
\item $\lambda + 2\pi \chi$ on $(-\infty,-1)$, and
\item $-\lambda$ on $(0,\infty)$.
\end{itemize}
Recall that $\eta_L'$ (resp.\ $\eta_R'$) is the counterflow line of the field minus $\tfrac{\pi \chi}{2}$ (resp.\ plus $\tfrac{\pi \chi}{2}$).  We sent $x$ to $\infty$ because we want to view $\psi(\eta_L')$ as the counterflow line targeted at~$\infty$.  That is, the boundary data for $\psi(\eta_L')$ is given by
\begin{itemize}
\item $-a$ on $(\psi(\infty),0]$,
\item $b+\pi \chi$ on $(-1,\psi(\infty)]$,
\item $\lambda + \tfrac{3\pi \chi}{2} = \lambda' + 2\pi \chi$ on $(-\infty,-1)$, and
\item $-\lambda - \tfrac{\pi \chi}{2} = -\lambda' - \pi \chi$ on $(0,\infty)$.
\end{itemize}
This means that $\eta_L'$ is an $\SLE_{\kappa'}(\rho^{3,L},\rho^{2,L},\rho^{1,L}; \rho^{1,R})$ process where:
\begin{align*}
   \rho^{1,R} &= \frac{\pi \chi}{\lambda'} = \frac{\kappa'}{2}-2\\
   \rho^{1,L} &= -\frac{a}{\lambda'}-1 \in \left(-2,\frac{\kappa'}{2}-2\right)\\
   \rho^{1,L}+\rho^{2,L} &= \frac{b+\pi\chi}{\lambda'} -1 \in \left(-2,\frac{\kappa'}{2}-2 \right) \\
   \rho^{1,L} + \rho^{2,L} + \rho^{3,L} &= \frac{2\pi \chi}{\lambda'} = \kappa'-4
   \end{align*}
The force point for~$\eta_L'$ which corresponds to~$\rho^{1,R}$ is located immediately to the right of~$-1$, the one to~$\rho^{1,L}$ is immediately to the left of~$-1$, the one to~$\rho^{2,L}$ is at~$\infty$, and the one to~$\rho^{3,L}$ is at~$1$.  These values imply that~$\eta_L'$ cannot hit~$(-1,x)$ and also cannot hit~$(x,1)$ but it eventually terminates at~$x$.  On the other hand, $\eta_L'$ necessarily hits both $\vhstripleft$ and $\vhstripright$.  (Recall Figure~\ref{fig::hittingrange}.)

One can similarly calculate the $\rho$ values for $\eta_R'$ and see that $\eta_R'$ cannot hit $(-1,x)$ or $(x,1)$ but eventually terminates at $x$ and that $\eta_R'$ hits both~$\vhstripleft$ and~$\vhstripright$.  This, of course, makes sense in view of what is proved below: that $\eta$ is given by the intersection of the right boundary of $\eta_L'$ and the left boundary of $\eta_R'$ and $a,b$ are such that $\eta$ can hit both the left and right sides of $\partial \vhstrip$.

It follows from \cite[Theorem~1.4]{MS_IMAG} that $\eta([0,\sigma])$ is equal to the right outer boundary of $\eta_L'$ as well as to the left outer boundary of $\eta_R'$ (see also Figure~\ref{fig::counterflowline}).  That is, $\eta([0,\sigma])$ is equal to the segment of the boundary of the component of $\vhstrip \setminus \eta_L'$ with $x$ on its boundary which runs in the counterclockwise direction from $x$ to $\eta(\sigma)$ and $\eta([0,\sigma])$ is also equal to the segment of the boundary of the component of $\vhstrip \setminus \eta_R'$ with $x$ on its boundary which runs in the clockwise direction from $x$ to $\eta(\sigma)$.  Let $\tau_q'$ be the first time $t$ that $\eta_q'$ hits $\eta(\tau')$ for $q \in \{L,R\}$.  Then $\eta_L$ is equal to the segment of the boundary of the connected component $A_L'$ of $\vhstrip \setminus \eta_L'([0,\tau_L'])$ which contains $x$ on its boundary which runs from $-1$ to $\eta'(\tau_L') = \eta(\tau')$ in the counterclockwise direction.  Similarly, $\eta_R$ is equal to the segment of the boundary of the connected component $A_R'$ of $\vhstrip \setminus \eta_R'([0,\tau_R'])$ which contains $x$ on its boundary which runs from $1$ to $\eta_R'(\tau_R') = \eta(\tau')$ in the clockwise direction.  Moreover, we also have that the right outer boundary of $\eta_L'([0,\tau_L'])$ is equal to $\eta([\tau',\sigma])$.  That is, $\eta([\tau',\sigma])$ is equal to the segment of $\partial A_L'$ which runs in the counterclockwise direction from $\eta_L'(\tau_L') = \eta(\tau')$ to $\eta(\sigma)$.  Likewise, the left outer boundary of $\eta_R'([0,\tau_R'])$ is also equal to $\eta([\tau',\sigma])$.  That is, $\eta([\tau',\sigma])$ is equal to the segment of $\partial A_R'$ which runs in the clockwise direction from $\eta_R'(\tau_R') = \eta(\tau')$ to $\eta(\sigma)$.  Combining, we have that $\eta([\tau',\sigma])$ is equal to $\partial A_L' \cap \partial A_R'$, i.e., the common part of the outer boundaries of $\eta_L'([0,\tau_L'])$ and $\eta_R'([0,\tau_R'])$.

We are now going to compute the conditional law of $\eta|_{[0,\tau']}$ given $K = K(\tau') = \eta_L'([0,\tau_L']) \cup \eta_R'([0,\tau_R'])$.  The first step is to show that $K$ is a local set for $h$ (the notion of a local set is explained in \cite[Section~3.2]{MS_IMAG}; an argument similar to the one that we will give here is used to prove \cite[Lemma~7.7]{MS_IMAG}).  To see this, we will check the criteria of \cite[Lemma~3.6]{MS_IMAG}.  Fix any open set $U \subseteq \vhstrip$ and, for $q \in \{L,R\}$, we let $(\tau_q^U)'$ be the first time $t$ that $\eta_q'$ hits $U$.  Then $\eta_L'|_{[0,(\tau_L^U)']}$ and $\eta_R'|_{[0,(\tau_R^U)']}$ are both determined by $h|_{U^c}$ by \cite[Theorem~1.2]{MS_IMAG}.  (In particular, this implies that given $h|_{U^c}$, $\eta_L'|_{[0,(\tau_L^U)']}$ and $\eta_R'|_{[0,(\tau_R^U)']}$ are independent of the orthogonal projection of $h$ onto the closure of the space of functions compactly supported in $U$.)  Since $\eta_L'$ and $\eta_R'$ hit the points of $\eta|_{[0,\sigma]}$ in reverse chronological order, it follows that the event $\{K \cap U = \emptyset\}$ is determined by $\eta_L'|_{[0,(\tau_L^U)']}$ and $\eta_R'|_{[0,(\tau_R^U)']}$ hence also by $h|_{U^c}$ (by looking at the intersection of the right boundary of $\eta_L'|_{[0,(\tau_L^U)']}$ and the left boundary of $\eta_R'|_{[0,(\tau_R^U)']}$ we can tell whether or not $\eta([\tau',\sigma])$ is contained in $\vhstrip \setminus U$).  Therefore $K$ is a local set for $h$, as desired.

Let $D$ be the complementary connected component of $K$ which contains $x$ (note that this is the same as the complementary connected component of $\eta_L \cup \eta_R$ which contains $x$).  Then the law of $h|_D$ conditional on $K$ and $h|_K$ is that of a GFF on $D$ whose boundary conditions are as depicted in Figure~\ref{fig::local_dual_flow}.  Indeed, \cite[Proposition~3.8]{MS_IMAG} implies that $h|_D$ does in fact have this boundary behavior at every point, except possibly at $\eta(\tau')$, since at other boundary points we can compare the conditional mean of $h|_D$ with that of $h$ given $\eta_q'([0,s])$ for $q \in \{L,R\}$ and $s \geq 0$.  One could worry that the boundary behavior exhibits pathological behavior right at $\eta(\tau')$, though the argument used to prove \cite[Lemma~7.8]{MS_IMAG} rules this out (we draw $\eta_L'$ up to time $r > \tau_L'$, $\eta_R'$ up to time $s < \tau_R'$, and then use \cite[Proposition~6.5]{MS_IMAG} to get the continuity in the conditional mean as we first take the limit $s \uparrow \tau_R'$ and then take the limit $t \downarrow \tau_L'$).  Since $\eta_L$ and $\eta_R$ almost surely do not hit $\eta|_{[0,\tau']}$, it follows from \cite[Proposition~6.12]{MS_IMAG} that $\eta|_{[0,\tau']}$ has a continuous Loewner driving function viewed as a path in $D$.  Thus, \cite[Theorem~2.4 and Proposition~6.5]{MS_IMAG} together imply that the conditional law of $\eta$ given $K$ and $h|_K$ is that of an $\SLE_\kappa(\tfrac{\kappa}{2}-2;\tfrac{\kappa}{2}-2)$ process in the complementary component which contains $x$, where the extra force points are as described in the lemma statement (recall also Figure~\ref{fig::monotonicity} and Figure~\ref{fig::different_starting_point}).

Moreover, this holds when we condition on $\eta|_{[\tau',\sigma]}$ because \cite[Proposition~3.9]{MS_IMAG} implies that the conditional law of $h|_D$ given $K$ and $h|_K$ is equal to the conditional law of $h|_D$ given $K$, $h|_K$, along with $K'$ and $h|_{K'}$ where $K'$ is another local set for $h$ with $K' \cap D = \emptyset$ almost surely.  In particular, if $\tau''$ is any reverse stopping time for $\eta|_{[0,\sigma]}$ with $\tau'' < \tau'$, the conditional law of $h|_D$ given $K$ and $h|_K$ is equal to the conditional law of $h|_D$ given $K$, $h|_K$, $K(\tau'')$, and $h|_{K(\tau'')}$.  This implies that the conditional law of $\eta|_{[0,\tau']}$ given $K$ and $h|_K$ is the same as the conditional law of $\eta|_{[0,\tau']}$ given $K$, $h|_K$, $K(\tau'')$, and $h|_{K(\tau'')}$.  The claim follows since we can apply this to the collection of reverse stopping times which are of the form $r \vee \tau'$ for $r$ is a positive rational and it is clear that $\eta|_{[\tau',\sigma]}$ is determined by $\sigma(K(r \vee \tau'),\ h|_{K(r \vee \tau')} : r \in \Q,\ r > 0)$.  This completes the proof.
\end{proof}

\begin{figure}[ht!]
\begin{center}
\includegraphics[height=0.32\textheight]{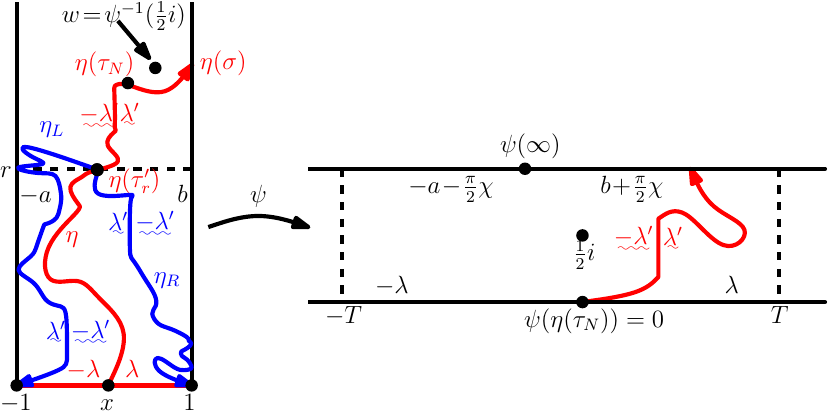}
\end{center}
\caption{\label{fig::approximatetriplepath} The setup for the proof of Lemma~\ref{lem::dualflow_limit}.  Suppose that $a,b \in (-\lambda'-\pi\chi,\lambda')$ and let $h$ be a GFF whose boundary data is as depicted on the left side.  Then the flow line $\eta$ of $h$ starting at $x$ is an $\SLE_\kappa(\rho^L;\rho^R)$ process with force points at $-1$ and $1$ of weight $\rho^L,\rho^R \in (\tfrac{\kappa}{2}-4,\tfrac{\kappa}{2}-2)$.  Let $M_t$ be the martingale as in the statement of Proposition~\ref{prop::mtweight} which upon reweighting the law of $\eta$ by yields an $\SLE_\kappa(\refrho{\rho}^L;\refrho{\rho}^R)$ process.  For each $N \geq 0$, let $\tau_N$ be the first time that $M$ hits $N$ and let $E_N = \{\tau_N < \tau_0\}$.  Let $\sigma$ be the first time that $\eta$ hits either $\vhstripleft$ or $\vhstripright$ and let $\tau_r'$ be the largest time that $\eta|_{[0,\sigma]}$ hits the horizontal line with height $r$.  Let $\eta_L$ and $\eta_R$ be the dual flow lines emanating from $\eta(\tau_r')$.  We prove in Lemma~\ref{lem::dualflow_limit} that the law of $\eta_L$ and $\eta_R$ converge to what it would be if $\eta$ were an $\SLE_\kappa(\refrho{\rho}^L;\refrho{\rho}^R)$ process.  The idea of the proof is to argue that $\im(\eta(\tau_N))$ is likely to be very large compared to $r$ and that $\eta$ quickly exits $\vhstrip$ after time $\tau_N$ so that the law of the dual flow lines is not affected by $\eta|_{[\tau_N,\sigma]}$.  The argument for this last point is illustrated above.  In particular, we let $\psi$ be the conformal map from $\vhstrip \setminus \eta([0,\tau_N])$ to the horizontal strip $\strip = \R \times [0,1]$ with $\psi(\eta(\tau_N)) = 0$, $\psi(-1) = -\infty$, and $\psi(1) = \infty$.  Brownian motion estimates then imply that the part of $\vhstrip$ which lies below height $r$ is mapped by $\psi$ very far away from $0$.  The claim follows since $\psi(\eta|_{[\tau_N,\sigma]})$ is an $\SLE_\kappa(\rho^L;\rho^R)$ process on $\strip$ which almost surely hits its upper boundary in finite time; by continuity, it with high probability does so without traveling far in the lateral direction.}
\end{figure}

The purpose of the next series of results, Lemmas~\ref{lem::dualflow_limit}--\ref{lem::resampling_martingale}, is to justify that the result of Lemma~\ref{lem::local_dual_flow} holds when we condition $\eta$ not to hit the boundary (in the sense of Proposition~\ref{prop::conditionalunique}).  For simplicity, we will restrict our attention to reverse stopping times of the form $\tau_r'$ --- the first time that the time-reversal of $\eta|_{[0,\sigma]}$ hits the horizontal line through $ir$.  We will first prove in Lemma~\ref{lem::dualflow_limit} that the law of the triple $(\eta,\eta_L,\eta_R)$ conditional on the event that the martingale $M_t$ of Proposition~\ref{prop::mtweight} hits level $N$ before hitting $0$ converges as $N \to \infty$ to the corresponding triple where $\eta$ is replaced with an $\SLE_\kappa(\refrho{\rho}^L;\refrho{\rho}^R)$ process.  Lemma~\ref{lem::martingale_gets_large} is a basic result about continuous time martingales which will then be used to prove Lemma~\ref{lem::functional_not_changed}.  The latter states that $M_t(\gamma)$ (the functional of Proposition~\ref{prop::mtweight} applied to a simple path $\gamma$) does not change much when we jiggle the initial segment of $\gamma$.  This is then employed in Lemma~\ref{lem::resampling_martingale} to show that resampling $\eta|_{[0,\tau_r']}$ does not affect the probability that $M_t = M_t(\eta)$ will exceed a certain large threshold.  We then combine all of these ingredients to complete the proof of Proposition~\ref{prop::middlepath} at the end of the subsection.

\begin{lemma}
\label{lem::dualflow_limit}
Suppose that $\eta$ is an $\SLE_\kappa(\rho^L;\rho^R)$ process interpreted as a flow line of a GFF on $\vhstrip$ with boundary conditions $-a$ and $b$ for $a,b \in (-\lambda'-\pi \chi,\lambda')$ as in Figure~\ref{fig::local_dual_flow}.  Let $M_t$ be the martingale as in Proposition~\ref{prop::mtweight} which, upon reweighting the law of $\eta$ by $M$, yields an $\SLE_\kappa(\refrho{\rho}^L;\refrho{\rho}^R)$ process.  For each $N \geq 0$, let $\tau_N = \inf\{t \geq 0 : M_t = N\}$ and let $\sigma$ be the first time that $\eta$ hits either $\vhstripleft$ or $\vhstripright$.  For each $r \geq 0$, we let $\tau_r'$ be the largest time $t \leq \sigma$ that $\im(\eta(t)) = r$ (this corresponds to the first time that the time-reversal of $\eta|_{[0,\sigma]}$ hits the horizontal line through $ir$).  Let $\eta_L$ and $\eta_R$ be the left and right dual flow lines of $h$ conditional on $\eta([0,\sigma])$ emanating from $\eta(\tau_r')$ (i.e., the flow lines with angles $\pi$ and $-\pi$).  Then the law of the triple $(\eta,\eta_L,\eta_R)$ conditional on $E_N = \{\tau_N < \tau_0\}$ converges weakly as $N \to \infty$ to the corresponding triple with $\eta$ replaced by an $\SLE_\kappa(\refrho{\rho}^L;\refrho{\rho}^R)$ process with respect to the topology of local uniform convergence, modulo parameterization.
\end{lemma}
\begin{proof}
Conditional on $E_N$, Proposition~\ref{prop::mtweight} implies that $\eta$ evolves as an $\SLE_\kappa(\refrho{\rho}^L;\refrho{\rho}^R)$ process up until time $\tau_N$; afterwards it evolves as an $\SLE_\kappa(\rho^L;\rho^R)$ process.  We first observe that, for every $t > 0$,
\[ \liminf_{N \to \infty} \p[ \tau_N \geq t \giv E_N] = 1.\]
Indeed, the reason for this is that if $M_t \geq N$ then either $|V_t^L-W_t| \geq N^{\alpha}$ or $|V_t^R - W_t| \geq N^{\alpha}$ for $\alpha = \alpha(\rho^L,\rho^R) > 0$.  It follows from the definition of the driving function of an $\SLE_\kappa(\refrho{\rho}^L;\refrho{\rho}^R)$ process \cite[Section~2]{MS_IMAG} that with $\xi_N = \inf\{t \geq 0 : |V_t^L-W_t| \wedge |V_t^R - W_t| \geq N^\alpha\}$ we have that $\lim_{N \to \infty} \p[ \xi_N \geq t \giv E_N] = 1$ for every $t > 0$.  Let $\vhstrip_r = \{z \in \vhstrip : \im(z) \leq r\}$ and $F_{r,N} = \{ \eta(\tau_N) \notin \vhstrip_r\}$.  By the transience of $\SLE_\kappa(\wh{\rho}^L;\wh{\rho}^R)$ processes \cite[Theorem~1.3]{MS_IMAG}, it thus follows that
\[ \liminf_{N \to \infty} \p[F_{r,N} \giv E_N] = 1 \quad\text{for every}\quad r > 0.\]

With $G_{r,N} = \{\eta([\tau_N,\sigma]) \cap \vhstrip_r = \emptyset\}$, we further claim that
\begin{equation}
\label{eqn::tail_far_away_whp}
 \liminf_{N \to \infty} \p[G_{r,N} \giv E_N] = 1 \quad\text{for every}\quad r > 0.
\end{equation}
To see this, we let $\psi$ be the conformal map from $\vhstrip \setminus \eta([0,\tau_N])$ to the horizontal strip $\strip$ with $\psi(\eta(\tau_N)) = 0$, $\psi(-1)=-\infty$, and $\psi(1) = \infty$ (see Figure~\ref{fig::approximatetriplepath}).  For $T > 0$, we let $\strip_T = \{z \in \strip : |\re(z)| \leq T\}$.  We are going to use Brownian motion estimates and the conformal invariance of Brownian motion to show that $\p[\psi(\vhstrip_r) \subseteq \strip_T^c \giv E_N] \to 1$ as $N \to \infty$ but with $r,T > 0$ fixed.  Let $w = \psi^{-1}(\tfrac{1}{2}i)$.  By symmetry, the probability that a Brownian motion starting at $\tfrac{1}{2}i$ first exits $\strip$ in $(-\infty,0)$ is $\tfrac{1}{4}$ (the same is likewise true for $(0,\infty)$).  Consequently, by the conformal invariance of Brownian motion, the probability that a Brownian motion starting at $w$ first exits $\vhstrip \setminus \eta([0,\tau_N])$ on the left side of $\eta([0,\tau_N])$ or in $[-1,x]$ is $\tfrac{1}{4}$ (the same is likewise true for the right side of $\eta([0,\tau_N])$ or $[x,1]$).  Moreover, the probability that a Brownian motion starting at $w$ exits first in $\vhstripleft \cup \vhstripright$ is $\tfrac{1}{2}$ (since a Brownian motion starting at $\tfrac{1}{2}i$ in $\strip$ first exits $\strip$ in $\R+i$ with probability $\tfrac{1}{2}$).  The Beurling estimate \cite[Theorem~3.69]{LAW05} thus implies that there exists a universal constant $d_0 > 0$ such that
\begin{equation}
\label{eqn::height_bound}
|\im(\eta(\tau_N))-\im(w)| \leq d_0.
\end{equation}
Indeed, if $\im(w)$ differs from $\im(\eta(\tau_N))$ by too much, the Beurling estimate implies that it is much more likely for a Brownian motion starting from $w$ either to hit one side of $\eta([0,\tau_N])$ or $\vhstripleft \cup \vhstripright$ before hitting the other side of $\eta([0,\tau_N])$.  Using the Beurling estimate again, we also see that for every $\epsilon > 0$ there exists $r_0' = r_0'(r,\epsilon)$ such that for all $r' \geq r_0'$ on the event $F_{r',N}$ the probability that a Brownian motion starting at $w$ first exits $\vhstrip \setminus \eta([0,\tau_N])$ in $\vhstrip_r$ is at most $\epsilon$.  The reason for this is that~\eqref{eqn::height_bound} implies that, on $F_{r',N}$, we have that $\im(w) \geq r'-d_0$.  Consequently, conformal invariance of Brownian motion implies that, for each $T > 0$, there exists $\epsilon_0 = \epsilon_0(T) > 0$ such that for all $\epsilon \in (0,\epsilon_0)$ we have that $\psi(\vhstrip_r) \subseteq \strip_T^c$ on $F_{r',N}$ for $r' \geq r_0$.

Note that $\psi(\eta|_{[\tau_N,\sigma]})$ has the law of an $\SLE_\kappa(\rho^L;\rho^R)$ process in $\strip$ with force points at $\pm \infty$ which almost surely hits the upper boundary of $\strip$.  By the continuity of such processes \cite[Theorem~1.3]{MS_IMAG}, it follows that $\p[\psi(\eta|_{[\tau_N,\sigma]}) \subseteq \strip_T] \to 1$ as $T \to \infty$.  This completes the proof of~\eqref{eqn::tail_far_away_whp}.

It follows from what we have shown far that, for each $S > 0$, the conditional law of the GFF $h$ given $\eta$ and $E_N$ restricted to $\vhstrip_S \setminus \eta$ converges in the $N \to \infty$ limit to the corresponding GFF with $\eta$ replaced by an $\SLE_\kappa(\refrho{\rho}^L;\refrho{\rho}^R)$ process.  The result follows since the continuity of $\eta_L$ and $\eta_R$ \cite[Theorem~1.3]{MS_IMAG} implies that $\p[ \eta_L \cup \eta_R \subseteq \vhstrip_S \giv E_N]$ can be made as close to $1$ as we like for large enough $S$.
\end{proof}

\begin{lemma}
\label{lem::martingale_gets_large}
Suppose that $M_t$ is a continuous time martingale taking values in $[0,\infty)$ with continuous sample paths.  For each $\alpha \geq 0$, let $\tau_\alpha = \inf\{t \geq 0 : M_t = \alpha\}$.  Assume that $\p[\tau_\alpha < \infty] > 0$ for all $\alpha > 0$.  Then for each $\epsilon > 0$ we have that
\[ \p[ \tau_{(1+\epsilon)\alpha} < \tau_0 \giv \tau_\alpha < \tau_0] = \frac{1}{1+\epsilon}.\]
\end{lemma}
\begin{proof}
This is a basic fact about martingales.  Indeed, we let $\CF_t$ be the filtration generated by $M$.  The optional stopping theorem implies that
\[ \E[ M_{\tau_{(1+\epsilon)\alpha} \wedge \tau_0} \giv \CF_{\tau_\alpha}]\one_{\{\tau_\alpha < \tau_0\}} = \alpha \one_{\{\tau_\alpha < \tau_0\}}.\]
On the other hand, we also have that
\[ \E[ M_{\tau_{(1+\epsilon)\alpha} \wedge \tau_0} \giv \CF_{\tau_\alpha}]\one_{\{\tau_\alpha < \tau_0\}} = (1+\epsilon)\alpha \p[\tau_{(1+\epsilon)\alpha} < \tau_0 \giv \CF_{\tau_\alpha}] \one_{\{\tau_\alpha < \tau_0\}}.\]
Combining the two equations implies the result.
\end{proof}

\begin{figure}[ht!]
\begin{center}
\includegraphics[scale=0.85]{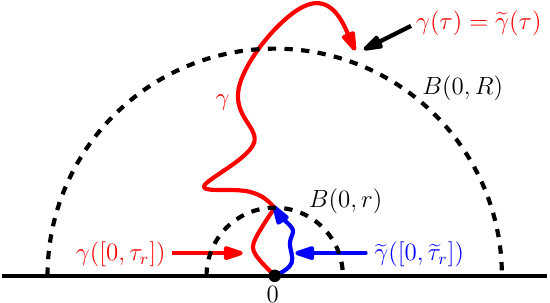}
\end{center}
\caption{\label{fig::functional_not_changed}  Suppose that $r > 0$ and that $\gamma,\wt{\gamma}$ are simple paths in $\ol{\h}$ starting at $0$ that agree with each other after leaving $B(0,r)$.  Let $\tau_r,\wt{\tau}_r$ be the first time that $\gamma,\wt{\gamma}$ exit $B(0,r)$, respectively.  By reparameterizing $\gamma,\wt{\gamma}$, we may assume without loss of generality that $\tau_r = \wt{\tau}_r$ and that $\gamma|_{[\tau_r,\infty)} = \wt{\gamma}_{[\tau_r,\infty)}$.  Let $M_t(\gamma)$ be the functional on paths described in Proposition~\ref{prop::mtweight} applied to $\gamma([0,t])$.  That is, if $W_t(\gamma)$ is the Loewner driving function of $\gamma$ and $V_t^L(\gamma)$ (resp.\ $V_t^R(\gamma)$) denotes the image of the leftmost (resp.\ rightmost) point of $\gamma([0,t]) \cap \R$ under the Loewner flow, then $M_t(\gamma) = |V_t^L(\gamma)-W_t(\gamma)|^{\alpha_L} |V_t^R(\gamma) - W_t(\gamma)|^{\alpha_R} | V_t^L(\gamma) - V_t^R(\gamma)|^{\alpha_{LR}}$ for constants $\alpha_L,\alpha_R,\alpha_{LR} \in \R$.   We prove in Lemma~\ref{lem::functional_not_changed} that for every $\epsilon,r > 0$ there exists $R_0 = R_0(r,\epsilon) > 0$ such that for every $R \geq R_0$ we have that $1-\epsilon \leq M_{\tau}(\gamma) / M_{\tau}(\wt{\gamma}) \leq 1+\epsilon$ where $\tau \geq \tau_r$ is a time such that $\gamma(\tau) = \wt{\gamma}(\tau) \notin B(0,R)$.}
\end{figure}

\begin{lemma}
\label{lem::functional_not_changed}
Suppose that $\gamma$ is a simple path in $\ol{\h}$ which admits a continuous Loewner driving function $W(\gamma)$; let $g_t$ be the corresponding family of conformal maps.  For each $t$, let $V_t^L(\gamma)$ (resp.\ $V_t^R(\gamma)$) be the image under $g_t$ of the leftmost (resp.\ rightmost) point of $\gamma([0,t]) \cap \R$.  Let
\[ M_t(\gamma) = |V_t^L(\gamma) - W_t(\gamma)|^{\alpha_L} | V_t^R(\gamma) - W_t(\gamma)|^{\alpha_R} | V_t^L - V_t^R|^{\alpha_{LR}}\]
for constants $\alpha_L,\alpha_R,\alpha_{LR} \in \R$ (this should be thought of as the functional on paths described in Proposition~\ref{prop::mtweight} applied to $\gamma([0,t])$).  For every $r,\epsilon > 0$, there exists $R_0 = R_0(r,\epsilon) > 0$ such that for every $R \geq R_0$ the following is true.  Suppose that $\gamma,\wt{\gamma}$ are simple paths in $\ol{\h}$ which agree with each other after exiting $B(0,r)$.  Assume further that $\gamma, \wt{\gamma}$ are parameterized so that with $\tau_r,\wt{\tau}_r$ the first time that each exits $B(0,r)$, we have that $\tau_r = \wt{\tau}_r$ and $\gamma|_{[\tau_r,\infty)} = \wt{\gamma}|_{[\wt{\tau}_r,\infty)}$.  Let $\tau \geq \tau_r$ be a time so that $\gamma(\tau) = \wt{\gamma}(\tau) \notin B(0,R)$.  Then
\[ 1-\epsilon \leq \frac{M_{\tau}(\gamma)}{M_{\tau}(\wt{\gamma})} \leq 1+\epsilon.\]
\end{lemma}
See Figure~\ref{fig::functional_not_changed} for an explanation of the setup of Lemma~\ref{lem::functional_not_changed}.  The idea of the proof is to use a Brownian motion estimate to show that the ratio of $|V_{\tau}^q(\gamma) - W_{\tau}(\gamma)|/|V_{\wt{\tau}}^q(\wt{\gamma}) - W_{\wt{\tau}}(\wt{\gamma})|$ is close to $1$ for $q \in \{L,R\}$.
\begin{proof}[Proof of Lemma~\ref{lem::functional_not_changed}]
Since $|V_t^R(\gamma) - V_t^L(\gamma)| = |V_t^R(\gamma) - W_t(\gamma)| + |W_t(\gamma) - V_t^L(\gamma)|$ (and likewise with $\wt{\gamma}$ in place of $\gamma$), it suffices to show that there exists $R_0 = R_0(r,\epsilon) > 0$ such that $R \geq R_0$ implies
\[ 1-\epsilon \leq \frac{|V_{\tau}^q(\gamma) - W_{\tau}(\gamma)|}{|V_{\tau}^q(\wt{\gamma}) - W_{\tau}(\wt{\gamma})|} \leq 1+\epsilon \quad\text{for}\quad q \in \{L,R\}.\]
Let $L_t(\gamma)$ (resp.\ $R_t(\gamma)$) be the left (resp.\ right) side of $\gamma([0,t])$ and let $L_t(\wt{\gamma}),R_t(\wt{\gamma})$ denote the analogous quantities with $\wt{\gamma}$ in place of $\gamma$.  Let $\p_{iy}$ denote the law of a Brownian motion $B$ starting at $iy$ and let $\zeta(t;\gamma)$ be the first time that $B$ exits $\h \setminus \gamma([0,t])$.  The conformal invariance of Brownian motion implies that
\[ |V_t^q(\gamma) - W_t(\gamma)| = \lim_{y \to \infty} y\p_{iy}[ B_{\zeta(t;\gamma)} \in q_t(\gamma)] \quad\text{for}\quad q \in \{L,R\}\]
and likewise with $\wt{\gamma}$ in place of $\gamma$.  Thus it suffices to show that there exists $R_0 = R_0(r,\epsilon) > 0$ such that $R \geq R_0$ implies
\begin{equation}
\label{eqn::bm_exit_ratio}
 1-\epsilon \leq \frac{\p_{iy}[B_{\zeta(\tau;\gamma)} \in q_{\tau}(\gamma)]}{\p_{iy}[B_{\zeta(\tau;\wt{\gamma})} \in q_{\tau}(\wt{\gamma})]} \leq 1+\epsilon
\end{equation}
for all large $y > 0$ and $q \in \{L,R\}$.

To see that~\eqref{eqn::bm_exit_ratio} holds, we are going to argue that by making $R$ large enough, the conditional probability that $B$ first exits $\h \setminus \gamma([0,\tau])$ in the left side of $\gamma([0,\tau_r])$ given that it first exits in $R_\tau(\gamma)$ is at most $\epsilon$ (the case when $R$ and $L$ are swapped is analogous).  Let $\zeta_R$ be the first time $t$ that $B$ hits the connected component of $\partial B(0,\tfrac{1}{2}R) \setminus \gamma([0,\tau])$ which lies to the right of $\gamma([0,\tau])$.  We are going to establish our claim by arguing that there exists constants $C_1,C_2 > 0$ such that
\[ \p_{B_{\zeta_R}}[B_{\zeta(\tau;\gamma)} \in R_\tau(\gamma)] \geq C_1 \arg(B_{\zeta_R})\]
and
\[ \p_{B_{\zeta_R}}[B_{\zeta(\tau;\gamma)} \in R_{\tau_r}(\gamma))] \leq C_2 \big(\arg(B_{\zeta_R}-r) - \arg(B_{\zeta_R} + r) \big).\]
This suffices because by making $R$ sufficiently large, we can make the latter as small as we want relative to the former uniformly in the realization of $B_{\zeta_R}$.  To see the former, we note that it is easy to see that $\p_{B_{\zeta_R}}[B_{\zeta(\tau;\gamma)} \in R_\tau(\gamma)]$ is at least a constant $C_1 > 0$ times the probability that a Brownian motion starting at $B_{\zeta_R}$ first exits $\h$ in $(-\infty,0]$.  This probability is explicitly given by $\tfrac{1}{\pi}\arg(B_{\zeta_R})$.  To see the latter, we also note that it is also easy to see that $\p_{B_{\zeta_R}}[B_{\zeta(\tau;\gamma)} \in R_{\tau_r}(\gamma))]$ is at most a constant $C_2 > 0$ times the probability a Brownian motion starting at $B_{\zeta_R}$ first exits $\h$ in $[-r,r]$.  This probability is explicitly given by $\tfrac{1}{\pi}(\arg(B_{\zeta_R}-r) - \arg(B_{\zeta_R}+r))$, which completes the proof.
\end{proof}

\begin{lemma}
\label{lem::resampling_martingale}
We suppose that we have the same setup as described in the statement of Lemma~\ref{lem::dualflow_limit}.  Let $M_t(\gamma)$ be the functional on paths described in Proposition~\ref{prop::mtweight} (so that $M_t := M_t(\eta)$ is a martingale for $\eta \sim \SLE_\kappa(\rho^L;\rho^R)$).  Fix $r > 0$ and let $\eta_L,\eta_R$ be the left and right dual flow lines emanating from $\eta(\tau_r')$.  For a simple path $\gamma$, let $\tau_N(\gamma) = \inf\{t \geq 0 : M_t(\gamma) = N\}$.  Let $\ol{E}_{N,r}$ be the event that $\cap_\gamma \{\tau_N(\gamma) < \tau_0(\gamma)\}$ where the intersection is over the set of paths $\gamma$ which can be written as a concatenation of a simple path which connects $x$ to $\eta(\tau_r')$ and lies in the region of $\vhstrip$ between $\eta_L$ and $\eta_R$ with $\eta|_{[\tau_r',\infty)}$.  Then $\liminf_{N \to \infty} \p[\ol{E}_{N,r} \giv E_N] = 1$.
\end{lemma}
\begin{proof}
This follows by combining Lemmas~\ref{lem::dualflow_limit}-\ref{lem::functional_not_changed}. Indeed, Lemma~\ref{lem::dualflow_limit} implies that the law of the dual flow lines $\eta_L$ and $\eta_R$ conditional on $E_N$ has a limit as $N \to \infty$.  This in turn implies that the maximal height $Y_r$ reached by $\eta_L$ and $\eta_R$ before exiting $\vhstrip$ is tight conditional on $E_N$ as $N \to \infty$.  By the argument in the proof of Lemma~\ref{lem::dualflow_limit}, we also know that $\im(\eta(\tau_N))$ converges to $\infty$ conditional on $E_N$ as $N \to \infty$.  Consequently, Lemma~\ref{lem::functional_not_changed} implies that if $\gamma$ is any path which arises by concatenating any simple path which lies between $\eta_L$ and $\eta_R$ and connects $x$ to $\eta(\tau_r')$ with $\eta|_{[\tau_r',\sigma]}$ then
\begin{equation}
\label{eqn::martingale_bound}
 \frac{M_{\tau_N}(\gamma)}{M_{\tau_N}(\eta)} = 1+o(1)
\end{equation}
with high probability conditional on $E_N$ as $N \to \infty$.  The result now follows since Lemma~\ref{lem::martingale_gets_large} implies that, on the event $\{\tau_N < \tau_0\}$ we have that $\{\tau_{(1+\epsilon)N} < \tau_0\}$ occurs with probability $(1+\epsilon)^{-1}$.  Combining this with~\eqref{eqn::martingale_bound} implies that $\liminf_{N \to \infty} \p[ \ol{E}_{N,r} \giv E_N] \geq (1+\epsilon)^{-1}$.  This completes the proof of the result since this holds for all $\epsilon > 0$.
\end{proof}

We now have all of the ingredients to complete the proof of Proposition~\ref{prop::middlepath}.

\begin{proof}[Proof of Proposition~\ref{prop::middlepath}]
Fix $r > 0$.  We first assume that our reverse stopping time $\tau'$ satisfies $\tau' \leq \tau_r'$, the first time that the reversed path hits the horizontal line through $ir$.  Once we prove the result for reverse stopping times of this form, the result follows for general reverse stopping times because we can take a limit as $r \to \infty$.  Let $\eta$ be the flow line of a GFF $h$ whose boundary data is as depicted in Figure~\ref{fig::approximatetriplepath} and we let $M_t$ be the martingale as in Proposition~\ref{prop::mtweight}.  For each $N \geq 0$, let $\tau_N = \{t \geq 0: M_t = N\}$ and let $E_N = \{ \tau_N < \tau_0\}$.  Then, for $N \geq 1$, the law of $\eta$ conditional on $E_N$ is an $\SLE_\kappa(\refrho{\rho}^L;\refrho{\rho}^R)$ process until time $\tau_N$, after which it evolves as an $\SLE_\kappa(\rho^L;\rho^R)$ process.  Lemma~\ref{lem::local_dual_flow} implies that the law of the (unconditioned) triple $(\eta,\eta_L,\eta_R)$ is invariant under the kernel $\CK$ which resamples $\eta|_{[0,\tau']}$ as an $\SLE_\kappa(\tfrac{\kappa}{2}-2;\tfrac{\kappa}{2}-2)$ process in the connected component of $\vhstrip \setminus (\eta_L \cup \eta_R)$ which contains $x$.  This implies that the law of the triple $(\eta,\eta_L,\eta_R)$ conditional on $E_N$ is invariant under the kernel $\CK^N$ which resamples $\eta|_{[0,\tau']}$ as an $\SLE_\kappa(\tfrac{\kappa}{2}-2;\tfrac{\kappa}{2}-2)$ process but restricted to the set of paths which do not affect whether the event $E_N$ occurs.  The result then follows from Lemma~\ref{lem::resampling_martingale}.
\end{proof}


\section{The time-reversal satisfies the conformal Markov property}
\label{sec::time_reversal_conformal_markov}

\begin{figure}[h!]
\begin{center}
\includegraphics[scale=0.85]{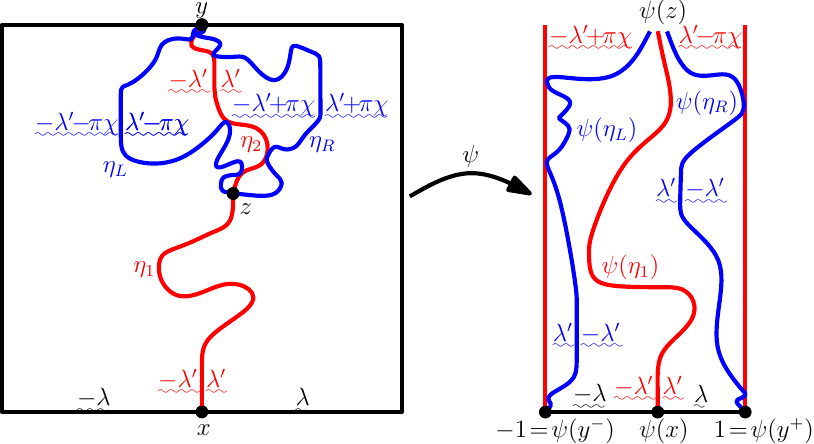}
\end{center}
\caption{\label{fig::reverse_stopped_plus_dual}  Suppose that $\eta$ is an $\SLE_\kappa$ process in a Jordan domain $D$ from $x$ to $y$ with $x,y \in \partial D$ distinct.  Let $\tau'$ be a reverse stopping time for $\eta$ and let $\eta_1 = \eta|_{[0,\tau']}$ and $\eta_2 = \eta|_{[\tau',\infty)}$.  Let $\eta_L$ and $\eta_R$ be the dual flow lines starting at $z = \eta(\tau')$ as in the statement of Proposition~\ref{prop::middlepath} (conditional on $\eta$, we sample a GFF in the left and right complementary connected components with the boundary data as shown and take $\eta_L,\eta_R$ to be the flow lines starting at $\eta(\tau')$ with angles $\pi,-\pi$, respectively).  Proposition~\ref{prop::middlepath} implies that if we condition on $\eta_2 \cup \eta_L \cup \eta_R$, we can resample $\eta_1$ as an $\SLE_\kappa(\tfrac{\kappa}{2}-2;\tfrac{\kappa}{2}-2)$ process.  Similarly, if we condition on $\eta_1 \cup \eta_2$, we can resample either $\eta_L$ or $\eta_R$ as described above.  As explained in the proof of Theorem~\ref{thm::usual_sle_reversible}, this invariance under resampling uniquely determines the conditional law of the triple $(\eta_L, \eta_R, \eta_1)$, given $\eta_2$.  In particular, this implies that $\eta_1$ is an $\SLE_\kappa$ process in $D \setminus \eta_2$ from $x$ to $z$ since $\SLE_\kappa$ from $x$ to $z$ in $D \setminus \eta_2$ satisfies the same resampling property.}
\end{figure}

We are now ready to prove the special case of Theorem~\ref{thm::all_reversible} for ordinary $\SLE_\kappa$ by showing that the time-reversal of $\SLE_\kappa$ satisfies the domain Markov property and is conformally invariant.  This suffices since it was proved by Schramm \cite{S0} that these properties characterize $\SLE_\kappa$ (the value of $\kappa$ is preserved because the almost sure Hausdorff dimension of the $\SLE_\kappa$ trace is $1+\tfrac{\kappa}{8}$ \cite{BEF_DIM} and this property is obviously invariant under time-reversal).  We will subsequently explain how this argument can be extended to prove that the time-reversal of $\SLE_\kappa(\rho)$ for $\rho \in (-2,0]$ also satisfies the domain Markov property and is conformally invariant.  The proof that $\SLE_\kappa(\rho)$ (and more generally that $\SLE_\kappa(\rho_1;\rho_2)$) is reversible will be postponed until the next section, however, since we will need some additional machinery.

\begin{theorem}
\label{thm::usual_sle_reversible}
Suppose that $D$ is a Jordan domain and $x,y \in \partial D$ are distinct.  Suppose that $\eta$ is an $\SLE_\kappa$ process in $D$ from $x$ to $y$.  Then the law of the time-reversal $\CR(\eta)$ of $\eta$ has the law of an $\SLE_\kappa$ process from $y$ to $x$ in $D$, up to reparameterization.  In other words, ordinary $\SLE_\kappa$ has time-reversal symmetry.
\end{theorem}
\begin{proof}
It is enough to show a reverse domain Markov property by \cite{S0} namely that if (in the setting of Figure~\ref{fig::reverse_stopped_plus_dual}) we condition on $\eta_2 = \eta|_{[\tau',\infty)}$ where $\tau'$ is a reverse stopping time for $\eta$, then the conditional law of $\eta_1 = \eta|_{[0,\tau']}$ is that of an ordinary $\SLE_\kappa$ in $D \setminus \eta_2$ from $x$ to $z=\eta(\tau')$.  Let $\eta_L,\eta_R$ be dual flow lines emanating from $z = \eta(\tau')$ as in the statement of Proposition~\ref{prop::middlepath} (see Figure~\ref{fig::reverse_stopped_plus_dual} for further explanation of the setup).

We know how to resample any one of $\eta_1$ or $\eta_L$ or $\eta_R$ given the realization of the other two ($\eta_2$ is fixed).  We claim that these resampling rules determine the joint law of these three paths.  Indeed, if we condition on initial segments of $\eta_L$ and $\eta_R$, then the fact that the resampling law determines the law of the remainder of these paths and of $\eta_1$ follows from argument used to prove Theorem~\ref{thm::bi_chordal}.  To be more explicit, the conditional law of $(\eta_1,\eta_L,\eta_R)$ (given the two initial segments) must be that of ordinary flow lines off a GFF conditioned on the positive probability event that $\eta_L$ and $\eta_R$ do not intersect (since this conditioned law satisfies the same resampling properties).  In particular, the law of $\eta_L$ and $\eta_R$ is in the setting of Proposition~\ref{prop::conditionalunique2}, so it is indeed determined; and given these two paths, we know how to sample $\eta_1$.  This completes the proof of the domain Markov property because we know that ordinary $\SLE_\kappa$ in $D \setminus \eta_2$ from $x$ to $z$ satisfies the same resampling rule.

The conformal Markov characterization implies that $\CR(\eta)$ is an $\SLE_{\wt{\kappa}}$ process from some $\wt{\kappa}$.  We note that $\wt{\kappa} = \kappa$ (the same value as for $\eta$) because the dimension of an $\SLE_\kappa$ process is almost surely $1+\tfrac{\kappa}{8}$ \cite{BEF_DIM} and this property is obviously preserved under time-reversal.
\end{proof}

\begin{figure}[h!]
\begin{center}
\includegraphics[scale=0.85]{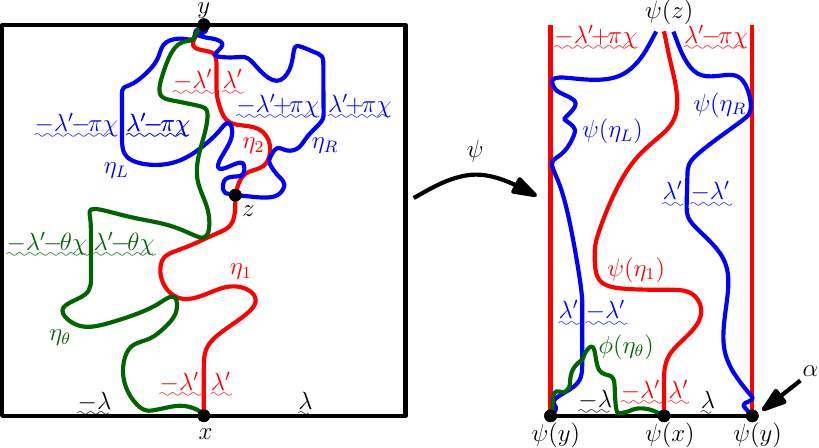}
\end{center}
\caption{\label{fig::reverse_stopped_plus_dual_theta} The left side of the figure is the same as Figure~\ref{fig::reverse_stopped_plus_dual} except with an additional path: a flow line of angle $\theta > 0$ from $x$ to $y$ that goes to the left $\eta_1 \cup \eta_2$.  (Note that $\eta_\theta$ either crosses $\eta_L$, as shown in the figure, or stays to its left, depending on $\theta$.)  We can conformally map $D \setminus \eta_2$ to the half-infinite vertical strip $\vhstrip$ in the right panel, just as in Figure~\ref{fig::reverse_stopped_plus_dual}.  The right panel has four marked boundary points, and we already have a recipe for producing the paths: namely, we draw $\psi(\eta_1)$ as an ordinary $\SLE_\kappa$, then we sample the GFF with the illustrated boundary conditions, and then draw the other paths as flow lines with illustrated angles.  This recipe immediately implies that law of the triple $(\psi(\eta_\theta), \psi(\eta_1), \psi(\eta_L))$ is a conformal invariant that does not depend on the point labeled $\alpha$ on the right.  In particular, this implies that the conditional law of $\psi(\eta_1)$ given $\psi(\eta_\theta)$ does not depend on $\alpha$.  We have not given an explicit description of this law but we will not need one.  One indirect way to describe it (that actually uses $\alpha$) is to fix $\psi(\eta_\theta)$ and note that we have a recipe for resampling each of the paths $\psi(\eta_1)$, $\psi(\eta_R)$ and $\psi(\eta_L)$ (stopped when it first hits $\psi(\eta_\theta)$) as GFF flow lines once we condition on the other two.  By Proposition~\ref{prop::middlepath} this determines the law of the triple given $\psi(\eta_\theta)$, and in particular the law of $\psi(\eta_1)$ given $\psi(\eta_\theta)$.}
\end{figure}

We are now going to extend the proof of Theorem~\ref{thm::usual_sle_reversible} to show that the time-reversal of an $\SLE_\kappa(\rho)$ process for $\rho \in (-2,0]$ is conformally invariant and satisfies the domain Markov property with one extra marked point (see also Figure~\ref{fig::reverse_stopped_plus_dual_theta}).

\begin{theorem}
\label{thm::time_reversal_conformal_markov}
Let $D \subseteq \C$ be a Jordan domain and fix $x,y \in \partial D$ distinct.  Suppose that $\kappa \in (0,4)$ and $\rho \in (-2,0]$.  Let $\eta$ be an $\SLE_\kappa(\rho)$ process in $D$ from $x$ to $y$ with the single force point located at $x$.  Then the time-reversal $\CR(\eta)$ of $\eta$ is conformally invariant and satisfies the domain Markov property with one extra marked point.
\end{theorem}

We remark that it is possible to extend the proof of Theorem~\ref{thm::time_reversal_conformal_markov} to include the case that $\rho > 0$.  The reason we chose to omit this particular case is that, using a technique which will be relevant for proving the reversibility of $\SLE_\kappa(\rho^L;\rho^R)$, we are also able to deduce the reversibility of $\SLE_\kappa(\rho)$ for $\rho > 0$ from the reversibility of $\SLE_\kappa(\rho)$ for $\rho \in (-2,0]$.  We also note that it is possible to extend the proof of Theorem~\ref{thm::time_reversal_conformal_markov} to show that the law of an $\SLE_\kappa(\rho^L;\rho^R)$ process given the realization of a terminal segment is conformally invariant and satisfies the domain Markov property with two extra marked points.  We do not include this here because these properties actually do not single out the class of $\SLE_\kappa(\rho_1;\rho_2)$ processes and we use another approach in the next section to prove reversibility in this setting.

\begin{proof}[Proof of Theorem~\ref{thm::time_reversal_conformal_markov}]
As explained in the caption of Figure~\ref{fig::reverse_stopped_plus_dual_theta}, we can take $\eta = \eta_1 \cup \eta_2$ to be an ordinary $\SLE_\kappa$ process from $x$ to $y$ in $D$, coupled with the GFF in the usual way, and then draw a flow line $\eta_\theta$ with angle $\theta > 0$ to its left.  By Figure~\ref{fig::monotonicity}, we know that the conditional law of $\eta$ given $\eta_\theta$ is that of an $\SLE_\kappa(\theta \chi/\lambda-2)$ process in the right connected component of $D \setminus \eta_\theta$.  We want to show that the law of $\eta_1$ given $\eta_2 \cup \eta_\theta$ is a conformal invariant of the connected component of the domain $D \setminus (\eta_2 \cup \eta_\theta)$ which contains $x$ that depends only on the three points on the boundary: $x$, $z$, and the first place $w$ after $x$ at which $\eta_\theta$ hits $\eta_2$.  Since we know how to resample $\eta_1$ given $\eta_{\theta}$ and $\eta_2$ (as explained in Figure~\ref{fig::reverse_stopped_plus_dual}, it follows that the law of $\eta_1$ given $\eta_2 \cup \eta_\theta$ is conformally invariant and depends on four boundary points ($x$, $z$, $w$, and the right side of $y$).  Thus we just need to show that the right side of $y$ is irrelevant.  This is explained in caption of Figure~\ref{fig::reverse_stopped_plus_dual_theta}, which completes the proof.
\end{proof}

\section{Proof of Theorem~\ref{thm::all_reversible}}
\label{sec::proof}

The completion of the proof of Theorem~\ref{thm::all_reversible} will require several steps.  First, we will observe that with a little thought, Theorem~\ref{thm::time_reversal_conformal_markov} implies the reversibility of $\SLE_\kappa(\rho)$ with a single force point of any weight $\rho \in (-2,0]$.  Then, we will make use of more conditioning tricks using flow lines of the GFF in order to obtain the result for all $\rho > -2$ and, more generally, two force points $\rho^L,\rho^R > -2$.

\subsection{One boundary force point; $\rho \in (-2,0]$}

Suppose that $\eta \sim \SLE_\kappa(\rho)$ in a Jordan domain $D$ from $x$ to $y$ with a single force point at $x^+$ with weight $\rho \in (-2,0]$.  By Theorem~\ref{thm::time_reversal_conformal_markov}, we know that the time-reversal $\CR(\eta)$ of $\eta$ is conformally invariant and satisfies the domain Markov property with one force point.  Moreover, \cite[Lemma~7.16]{MS_IMAG} implies that $\eta \cap D$ almost surely has zero Lebesgue measure if $\partial D$ is smooth.  Therefore Theorem~\ref{thm::conformal_markov} implies that there exists $\wt{\rho}$ such that $\CR(\eta) \sim \SLE_\kappa(\wt{\rho})$ in $D$ from $y$ to $x$.  We know that the value of $\kappa$ associated with $\CR(\eta)$ is the same as that associated with $\eta$ since the Girsanov theorem implies that the law of $\eta$ is mutually absolutely continuous with respect to the law of ordinary $\SLE_\kappa$ when it is not hitting the boundary.  Since the latter almost surely has Hausdorff dimension $1+\tfrac{\kappa}{8}$ \cite{BEF_DIM}, so does the former.  The claim follows because this property is obviously invariant under time-reversal.  Thus to complete the proof of Theorem~\ref{thm::all_reversible} in this case, we just need to show that $\rho = \wt{\rho}$.

Let $P = P(\kappa)$ be the set weights $\rho > -2$ such that there exists $\wt{\rho} > -2$ so that the time-reversal of an $\SLE_\kappa(\rho)$ process is an $\SLE_\kappa(\wt{\rho})$ process.  Note that $(-2,0] \subseteq P$.  Let $R \colon P \to P$ be the function that maps $\rho \in P$ to the corresponding weight for the time-reversal.  We are now going to argue that $R$ is the identity.  Observe that $R$ has the following properties:

\begin{enumerate}[(1)]
\item $R(R(\rho)) = \rho$ [the time-reversal of the time-reversal is the original path]
\item $R$ is injective [if $R(\rho) = R(\wt\rho)$ then $\rho = R(R(\rho)) = R(R(\wt\rho)) = \wt\rho$]
\item $R$ is continuous [if $\rho_n \to \rho$ and $\eta_n \sim \SLE_\kappa(\rho_n)$, then $\eta_n \stackrel{d}{\to} \eta \sim \SLE_\kappa(\rho)$ so $\CR(\eta_n) \stackrel{d}{\to} \CR(\eta)$; the convergence is weak convergence with the topology induced by local uniform convergence on the Loewner driving function.]
\item $R(0) = 0$ [the time-reversal symmetry of ordinary $\SLE_\kappa$ is Theorem~\ref{thm::usual_sle_reversible} and was already known \cite{Z_R_KAPPA, DUB_DUAL}.]
\item $\lim_{\rho \downarrow -2} R(\rho) = -2$ [if $\rho_n \downarrow -2$ and $\eta_n \sim \SLE_\kappa(\rho_n)$, then $\eta_n$ converges to $\R_+$ so that $\CR(\eta_n)$ converges to $\R^+$; recall Figure~\ref{fig::monotonicity}.]
\end{enumerate}
Properties (2) and (3) imply that $R|_{(-2,0]}$ is monotonic and properties (4) and (5) imply that $R|_{(-2,0]}$ is increasing.  The only increasing function on $(-2,0]$ satisfying (1) is the identity.  Indeed, if $R(\rho) > \rho$ for some $\rho > -2$, then since $R$ is increasing, we must have that $R(R(\rho)) > R(\rho)$.  But $R(R(\rho)) = \rho$, an obvious contradiction.  Likewise, it cannot be that $\rho > R(\rho)$.

\qed

\subsection{Two boundary force points}

\begin{figure}[ht!]
\begin{center}
\includegraphics[scale=0.85]{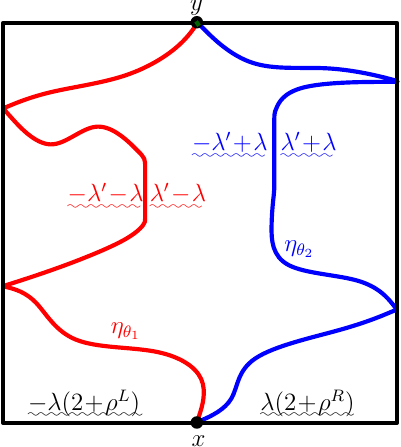}
\end{center}
\caption{\label{fig::two_force_point_non_boundary}  Fix $\rho^L,\rho^R \in (-2,0]$.  The configuration of paths used to prove the reversibility of $\SLE_\kappa(\rho^L;\rho^R+2)$ and $\SLE_\kappa(\rho^L+2;\rho^R)$ is illustrated above.  Suppose that $h$ is a GFF on a Jordan domain $D$ with the boundary data depicted above.  Let $\eta_{\theta_1},\eta_{\theta_2}$ be the flow lines of $h$ with angles $\theta_1= \lambda/\chi$ and $\theta_2=-\lambda/\chi$.  Then $\eta_{\theta_1} \sim \SLE_\kappa(\rho^L;2+\rho^R)$.  Moreover, the conditional law of $\eta_{\theta_1}$ given $\eta_{\theta_2}$ is an $\SLE_\kappa(\rho^L)$ and the conditional law of $\eta_{\theta_2}$ given $\eta_{\theta_1}$ is an $\SLE_\kappa(\rho^R)$.  Consequently, by the one boundary force point case for $\rho \in (-2,0]$, the conditional law of the time-reversal $\CR(\eta_{\theta_1})$ given the time-reversal $\CR(\eta_{\theta_2})$ of $\eta_{\theta_2}$ is an $\SLE_\kappa(\rho^L)$ and the conditional law of $\CR(\eta_{\theta_2})$ given $\CR(\eta_{\theta_1})$ is an $\SLE_\kappa(\rho^R)$.  By Theorem~\ref{thm::bi_chordal}, this characterizes the law of $(\CR(\eta_{\theta_1}),\CR(\eta_{\theta_2}))$ and, in particular, implies that $\CR(\eta_{\theta_1}) \sim \SLE_\kappa(2+\rho^R;\rho^L)$ from $y$ to $x$ in $D$ and similarly $\CR(\eta_{\theta_2}) \sim \SLE_\kappa(\rho^R;2+\rho^L)$.  By iterating this argument, by can prove the reversibility of $\SLE_\kappa(\rho^L;\rho^R)$ for all $\rho^L > -2$ and $\rho^R \geq 0$.
}
\end{figure}

We are now going to complete the proof of Theorem~\ref{thm::all_reversible} by deducing the reversibility of $\SLE_\kappa(\rho^L;\rho^R)$ for all $\rho^L,\rho^R > -2$ from the reversibility of $\SLE_\kappa(\rho)$ for $\rho \in (-2,0]$.  This will require two steps.  We will first prove the reversibility of $\SLE_\kappa(\rho^L;\rho^R)$ when $\rho^L > -2$ and $\rho^R \geq 0$.  This, in particular, will fill in the missing gap of the reversibility of $\SLE_\kappa(\rho)$ when $\rho \geq 0$.  We will prove this by considering a configuration of flow lines $\eta_{\theta_1},\eta_{\theta_2}$ of a GFF such that $\eta_{\theta_1}$ is reversible given $\eta_{\theta_2}$ and $\eta_{\theta_2}$ is reversible given $\eta_{\theta_1}$ and then use Theorem~\ref{thm::bi_chordal}, our characterization of bi-chordal $\SLE$, to deduce the reversibility of the whole configuration.  Second, we will deduce from this the reversibility of $\SLE_\kappa(\rho^L;\rho^R)$ for general $\rho^L,\rho^R > -2$ by considering another configuration of flow lines.

\begin{lemma}
\label{lem::two_force_point_non_boundary}
Suppose that $D \subseteq \C$ is a Jordan domain and fix $x,y \in \partial D$ distinct.  Suppose that $\eta \sim \SLE_\kappa(\rho^L;\rho^R)$ from $x$ to $y$ with $\rho^L > -2$ and $\rho^R \geq 0$ and with the weights located at $x^-,x^+$, respectively.  The time-reversal $\CR(\eta)$ of $\eta$ is an $\SLE_\kappa(\rho^R;\rho^L)$ process in $D$ with the weights located at $y^-,y^+$, respectively.
\end{lemma}
\begin{proof}
Fix $\rho^L,\rho^R \in (-2,0]$.  Suppose that $h$ is a GFF on $D$ whose boundary data is as described in Figure~\ref{fig::two_force_point_non_boundary}.  Let $\theta_1 =  \lambda/\chi$ and $\theta_2 = -\lambda/\chi$.  For $i=1,2$, let $\eta_{\theta_i}$ be the flow line of $h$ with angle $\theta_i$.  Then $\eta_{\theta_1}$ is an $\SLE_\kappa(\rho^L;2+\rho^R)$ process in $D$ from $x$ to $y$ with the force points located at $x^-,x^+$, respectively, and likewise $\eta_{\theta_2}$ is an $\SLE_\kappa(2+\rho^L;\rho^R)$ process.  Moreover, the law of $\eta_{\theta_1}$ given $\eta_{\theta_2}$ is an $\SLE_\kappa(\rho^L;0)$ in the left connected component $D_2$ of $D \setminus \eta_{\theta_2}$ and the law of $\eta_{\theta_2}$ given $\eta_{\theta_1}$ is an $\SLE_\kappa(0;\rho^R)$ in the right connected component $D_1$ of $D \setminus \eta_{\theta_1}$ (see Figure~\ref{fig::monotonicity}).  Therefore by Theorem~\ref{thm::all_reversible} for one boundary force point whose weight is in $(-2,0]$, which we proved in the previous subsection, we know that the law of the time-reversal $\CR(\eta_{\theta_1})$ of $\eta_{\theta_1}$ given the time-reversal $\CR(\eta_{\theta_2})$ of $\eta_{\theta_2}$ has the law of an $\SLE_\kappa(0;\rho^L)$ in $D_2$ from $y$ to $x$.   Likewise, the law of $\CR(\eta_{\theta_2})$ given $\CR(\eta_{\theta_1})$ is that of an $\SLE_\kappa(\rho^R;0)$ in $D_1$ from $y$ to $x$.  By Theorem~\ref{thm::bi_chordal}, we therefore have that $(\CR(\eta_{\theta_1}),\CR(\eta_{\theta_2})) \stackrel{d}{=} (\wt{\eta}_{1},\wt{\eta}_{2})$ where $\wt{\eta}_1 \sim \SLE_\kappa(2+\rho^R;\rho^L)$ and $\wt{\eta}_2 \sim \SLE_\kappa(\rho^R;2+\rho^L)$.  This proves the statement of the lemma when $\rho^L \in (-2,0]$ and $\rho^R \in [0,2)$.

The argument in the previous paragraph in particular implies the reversibility of $\SLE_\kappa(\rho)$ processes for $\rho \in (0,2)$.  Thus we can re-run the argument but now with the restriction $\rho^L,\rho^R \in (-2,2)$.  Iterating this proves the result for all $\rho^L > -2$ and $\rho^R \geq 0$.
\end{proof}

\begin{figure}[h!]
\begin{center}
\includegraphics[scale=0.85]{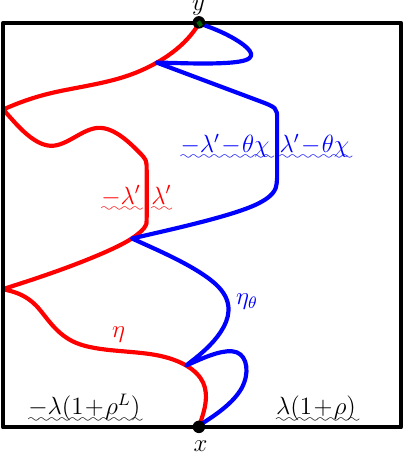}
\end{center}
\caption{\label{fig::two_force_point_boundary}  The configuration of paths used to prove the reversibility of $\SLE_\kappa(\rho^L;\rho^R)$ process for $\rho^L,\rho^R > -2$ assuming Lemma~\ref{lem::two_force_point_non_boundary}.  Suppose that $h$ is a GFF on a Jordan domain $D$ with the boundary data depicted above for $\rho^L > -2$ and $\rho > 0$ very large.  Let $\eta$ be the flow line of $h$ and $\eta_\theta$ the flow line with angle $\theta < 0$ chosen so that $\rho^R = -\theta \chi/ \lambda -2$.  Then $\eta \sim \SLE_\kappa(\rho^L;\rho)$, so $\CR(\eta) \sim \SLE_\kappa(\rho;\rho^L)$ since $\rho$ is large.  Moreover, conditionally on $\CR(\eta)$, we know from Figure~\ref{fig::monotonicity} that $\eta_\theta \sim \SLE_\kappa(-\theta \chi /\lambda -2;\rho+ \theta \chi/\lambda)$.  Thus the law of $\CR(\eta_\theta)$ conditionally on $\CR(\eta)$ is an $\SLE_\kappa(\rho+ \theta \chi/\lambda;- \theta \chi/\lambda-2)$.  Therefore we know the joint law of $(\CR(\eta),\CR(\eta_\theta))$, from which it is easy to see that $\CR(\eta)$ conditionally on $\CR(\eta_\theta)$ is an $\SLE_\kappa(\rho^R;\rho^L)$ process.  This proves the theorem since $\eta$ conditionally on $\eta_\theta$ is an $\SLE_\kappa(\rho^L;\rho^R)$ process (again by Figure~\ref{fig::monotonicity}).}
\end{figure}

We now have all of the ingredients needed to complete the proof of Theorem~\ref{thm::all_reversible}.

\begin{proof}[Proof of Theorem~\ref{thm::all_reversible}]
Suppose that $h$ is a GFF on $D$ whose boundary data is as indicated in Figure~\ref{fig::two_force_point_boundary} where $\rho > 0$ is very large and $\rho^L > -2$.  Fix $\rho^R > -2$ and let $\theta < 0$ be such that $\rho^R =  -\theta \chi/\lambda  - 2$.  Let $\eta$ be the flow line of $h$ and let $\eta_\theta$ be the flow line of $h$ with angle $\theta$.  Then $\eta \sim \SLE_\kappa(\rho^L;\rho)$ and the law of $\eta_\theta$ given $\eta$ is an $\SLE_\kappa(-\theta \chi/\lambda -2; \rho +  \theta \chi/ \lambda)$ from $x$ to $y$ in the right connected component $D_0$ of $D \setminus \eta$ (recall Figure~\ref{fig::monotonicity}).  Thus, assuming we chose $\rho$ large enough, Lemma~\ref{lem::two_force_point_non_boundary} implies $\CR(\eta) \sim \SLE_\kappa(\rho;\rho^L)$ from $y$ to $x$ in $D$ and the law of $\CR(\eta_\theta)$ given $\CR(\eta)$ is an $\SLE_\kappa(\rho+ \theta \chi/\lambda; - \theta \chi/\lambda -2)$ from $y$ to $x$ in $D_0$.  This implies that the law of $\CR(\eta)$ given $\CR(\eta_\theta)$ is an $\SLE_\kappa(\rho^R;\rho^L)$ from $y$ to $x$ in the left connected component $D_1$ of $D \setminus \eta_\theta$ (recall Figure~\ref{fig::monotonicity}).  The theorem follows since the law of $\eta$ given $\eta_\theta$ is an $\SLE_\kappa(\rho^L;\rho^R)$ in $D_1$ from $x$ to $y$ by Figure~\ref{fig::monotonicity}.
\end{proof}

\subsection{Whole plane SLE and variants}

If $\eta \colon [0,\infty] \to \C$ is a continuous path from $0$ to $\infty$, then we consider the map $\CR(\eta)(t) = [1/\overline{\eta(1/t)}]$ (up to monotone reparameterization of time) to be the ``time-reversal'' of $\eta$.  We state Theorem~\ref{thm::wholeplanetimereversal} below as a contingent result about whole-plane time-reversal: the conclusion will depend on the existence and uniqueness of a law for a random pair of paths that has certain properties.  In a subsequent work, we will explicitly construct pairs of paths with these properties as flow lines of the Gaussian free field (plus a multiple of the argument function) and establish the uniqueness of their laws an arguments similar to the one used for bichordal SLE processes here \cite{MS_INTERIOR}.  One result of this form was already proved by Zhan in \cite{2010arXiv1004.1865Z}.

\begin{theorem}
\label{thm::wholeplanetimereversal}
Suppose that $\eta_1$ and $\eta_2$ are random continuous paths in $\C$ from $0$ to $\infty$ for which the following hold:
\begin{enumerate}
\item Conditioned on $\eta_2$, the law of $\eta_1$ is that of an $\SLE_\kappa(\rho^L;\rho^R)$ from $0$ to $\infty$ in $\C \setminus \eta_2$.
\item Conditioned on $\eta_1$, the law of $\eta_2$ is that of an $\SLE_\kappa(\rho^R;\rho^L)$ from $0$ to $\infty$ in $\C \setminus \eta_1$.
\end{enumerate}
If $\eta_1$ and $\eta_2$ are the unique random paths that have these properties then the law of the pair $(\eta_1, \eta_2)$ is equal to the law of its time-reversal $(\CR(\eta_1), \CR(\eta_2))$ (up to monotone reparameterization).
\end{theorem}

\begin{proof}
The law of the pair of paths is invariant under the operation of successively resampling the two paths as $\SLE_\kappa(\rho^L;\rho^R)$ and $\SLE_\kappa(\rho^R;\rho^L)$ processes, which have time-reversal symmetry by Theorem~\ref{thm::all_reversible}.  We conclude  $(\CR(\eta_1), \CR(\eta_2))$ satisfies the properties described in the theorem statement, which were assumed to uniquely characterize the law of $(\eta_1, \eta_2)$.  Thus the pair $(\CR(\eta_1), \CR(\eta_2) )$ must have the same law as the pair $(\eta_1, \eta_2)$.
\end{proof}

In particular this will imply that each of $\eta_1$ and $\eta_2$ separately has time-reversal symmetry.  As we show in \cite{MS_INTERIOR}, there is a natural one-parameter family of laws for $\eta_1$ that can be constructed in this way, by considering the Gaussian free field couplings of Section~\ref{sec::multiple_force_points} but replacing the strip with a cylinder on which a multi-valued version of the Gaussian free field is defined (so that height changes by a constant as one makes a revolution around the cylinder), and taking $\eta_1$ and $\eta_2$ to be flow lines of different angles.

\section{Multiple force points}
\label{sec::multiple_force_points}

\subsection{Time reversal and free field perturbations}

\begin{figure}[h!]
\begin{center}
\includegraphics[scale=0.85]{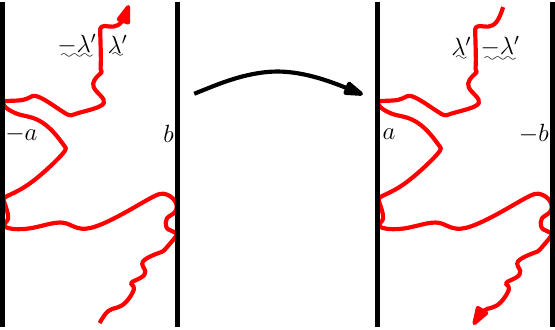}
\end{center}
\caption{\label{fig::abswap}
Consider a GFF $h$ on the infinite vertical strip $\vstrip = [-1,1] \times \R$ whose boundary values are depicted on the left side above.  The flow line $\eta$ of $h$ from the bottom to the top is an $\SLE_\kappa(\rho^L;\rho^R)$ process (we will calculate the values of $\rho^L,\rho^R$ in Figure~\ref{fig::abrhocalculation}).  To go from the left side of the figure to the right side, we add to $h$ on the left side of $\eta$ the harmonic function whose boundary values are equal to $2\lambda'$ along left side of $\eta$ and $2a$ along the left strip boundary $\vstripleft$.  On the right side of $\eta$, we add to $h$ the harmonic function whose boundary values are equal to $-2\lambda'$ along the right side of $\eta$ and $-2b$ along the right strip boundary $\vstripright$.  By the reversibility of $\SLE_\kappa(\rho^L;\rho^R)$ processes (Theorem~\ref{thm::all_reversible}), the resulting field is a GFF with the boundary data indicated on the right side and the time-reversal $\CR(\eta)$ of $\eta$ is the flow line of this new field from the top to the bottom.}
\end{figure}

\begin{figure}[h!]
\begin{center}
\includegraphics[scale=0.85]{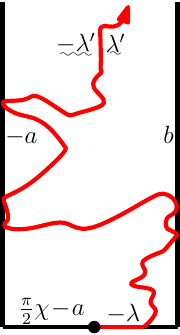}
\end{center}
\caption{\label{fig::abrhocalculation} Extending the black and red paths horizontally at the bottom (as though the domain were the half strip $\vhstrip =[-1,1] \times \R_+$) we find that the gap between the heights (red height minus black height) along the lower boundary is $-\lambda +a - \tfrac{\pi}{2}\chi = a - \lambda - \lambda (4-\kappa)/4 = a - (2 - \kappa/4) \lambda$.  Dividing by $\lambda$ gives a force point strength of $\rho^L = a/\lambda - 2 + \kappa/4$ (see Figure~\ref{fig::conditional_boundary_data}).  Symmetrically we find $\rho^R = b/\lambda + \kappa/4 - 2$.  In particular, if $a = b = \lambda'$ we find $\rho^L = \rho^R = \tfrac{\kappa}{2}-2$.  }
\end{figure}

In this section, we will make a series of observations that will lead to a proof of Theorem~\ref{thm::multiple_force_points}.  We will now illustrate what Theorem~\ref{thm::all_reversible} tells us about the couplings of $\SLE_\kappa(\rho^L;\rho^R)$ processes with the GFF described in Section~\ref{subsec::imaginary} and \cite{MS_IMAG}.  Consider the infinite vertical strip $\vstrip = [-1,1] \times \R$ with boundary conditions $-a$ on the left side $\vstripleft$ and $b$ on the right side $\vstripright$, as in the left panel of Figure~\ref{fig::abswap}.

\begin{figure}[h!]
\begin{center}
\includegraphics[scale=0.85]{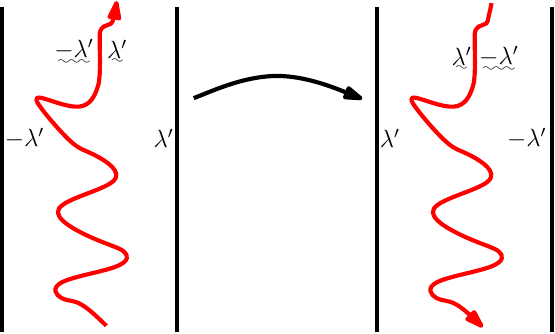}
\end{center}
\caption{\label{fig::simpleswap} To go from the left figure to the right figure, add the constant function $2\lambda'$ to left side of $\vstrip$ minus the path and $-2\lambda'$ to the right side.}
\end{figure}

By Figure~\ref{fig::conditional_boundary_data} the flow line path corresponds to an $\SLE_\kappa(\rho^L;\rho^R)$ process from the bottom of $\vstrip$ to the top with the $\rho^L$ and $\rho^R$ values described in Figure~\ref{fig::abrhocalculation} and initial force points at the seed of the path.
If we consider the boundary conditions on the right side of Figure~\ref{fig::abswap} and draw a path in the reverse direction, then the path corresponds to an $\SLE_\kappa(\rho^R;\rho^L)$ process from the top of $\vstrip$ to the bottom, with initial force points at the seed of the path.  Theorem~\ref{thm::all_reversible} implies that (up to time-reversal) the law of the path on the left and the law of the path on the right are exactly the same.  This implies that the transformation depicted in Figure~\ref{fig::abswap} is a {\em measure preserving map} from the space of GFF configurations with the boundary conditions on the left (recall from \cite[Theorem~1.2]{MS_IMAG} that the GFF determines the path almost surely) to the space of GFF configurations with the boundary conditions on the right (which again determine the path shown).  That is, if we sample an instance $h_1$ of the GFF with the left boundary conditions, and we then transform it according to the illustrated rule, then we obtain an instance $h_2$ of the GFF with the right boundary conditions.  This produces a {\em coupling} of the two Gaussian free fields $h_1$ and $h_2$ in such a way that the flow line on the left is the time-reversal of the flow line on the right and the two fields agree (up to harmonic functions) in the complement of the paths.  Indeed, $h_1 - h_2$ is the harmonic extension of the function that is $-2a$ on $\vstripleft$, $-2\lambda'$ on the left side of the path $2\lambda'$ on the right side of the path and $2b$ on $\vstripright$.

\begin{figure}[h!]
\begin{center}
\includegraphics[scale=0.85]{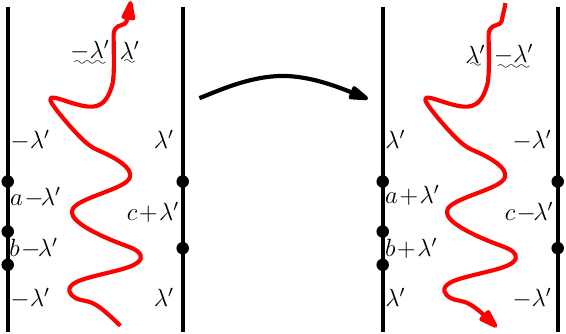}
\end{center}
\caption{\label{fig::simpleswapphi} In this case, the boundary conditions (on the left side) are equal to $-\lambda'$ on $\vstripleft$ and $\lambda'$ on $\vstripright$ except along compact intervals where they disagree.  To go from the left figure to the right figure, add $2\lambda'$ to left side of $\vstrip$ minus the path and $-2\lambda'$ to right side.}
\end{figure}

This construction takes a particularly simple form when we take $a = b = \lambda'$ as shown in Figure~\ref{fig::simpleswap}.  In this case, we have $\rho^L = \rho^R = -\pi\chi/\lambda = -(4-\kappa)/2 = \tfrac{\kappa}{2}-2$ (i.e., the $\rho$ value corresponding to a ``half-turn'').  The $\SLE_\kappa(\rho^L;\rho^R)$ is non-boundary intersecting in this case, but these values are {\em critical} for that to be the case (i.e., if the $\rho^q$ were any smaller then we would have boundary intersection; see \cite[Figure 4.1]{MS_IMAG}).  What makes this case simpler than the general one in Figure~\ref{fig::abswap} is that the functions one adds to the left and right sides of the path, in order to map from one type of boundary condition to another, are {\em constant}, as explained in Figure~\ref{fig::simpleswap}.  In this case $h_1 - h_2$ is simply $-2\lambda'$ to the left of the path and $2\lambda'$ to the right of the path.

\begin{figure}[h!]
\begin{center}
\includegraphics[scale=0.85]{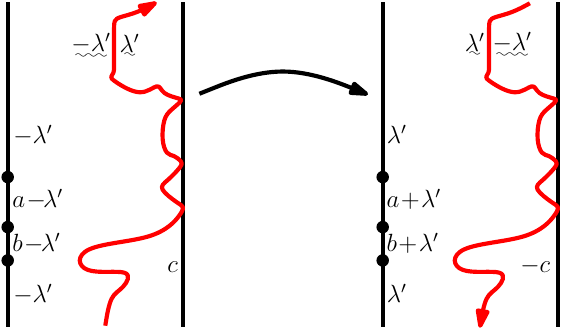}
\end{center}
\caption{\label{fig::abswapphi} In this case, the boundary conditions (on the left side) are equal to $-\lambda'$ on $\vstripleft$ except along a compact intervals where they disagree.  On $\vstripright$, the boundary conditions are given by a constant $c$.  To go from the left figure to the right figure, add $2\lambda'$ to left side of $\vstrip$ minus the path.  On the right side, add the harmonic function that is equal to $-2c$ on $\vstripright$ and $-2\lambda'$ on the right side of the path.}
\end{figure}

\begin{figure}[hbt!!]
\begin{center}
\includegraphics[scale=0.85]{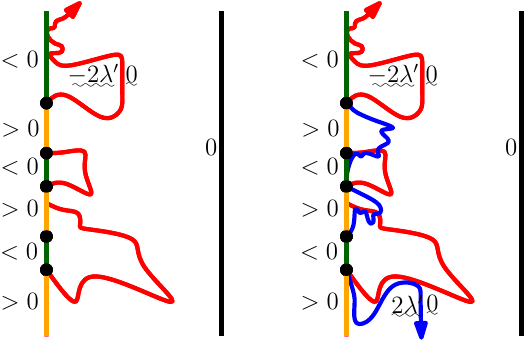}
\end{center}
\caption{\label{fig::sidematching} Procedure for producing a matching height condition on the opposite side of a strip.  In the left figure, one draws a certain flow line $\eta_1$ from the bottom to the top of the strip as shown, but whenever the flow line reaches a boundary location where the height is greater than or equal to zero (i.e., it reaches the continuation threshold) it is allowed to ``skip'' along the boundary of the strip to the first point where this is no longer the case and then start again at that point.  The path $\eta_1$ disconnects all of the left boundary intervals on which the height is {\em strictly less than zero} from the right boundary of the strip.  As shown on the right, we can repeat the same procedure with a path $\eta_2$ from top to bottom, which necessarily disconnects all intervals on which the height is strictly greater than zero from the right side of the strip.  Both $\eta_i$ avoid the right boundary of the strip almost surely so that the strip minus the pair of paths contains a ``large component'' that has the right boundary of the strip as its right boundary.  The angles are chosen so that conditioned on the paths the GFF boundary conditions on the left boundary of this component are $0$ (plus winding).}
\end{figure}

\begin{figure}[hbt!!]
\begin{center}
\includegraphics[scale=0.85]{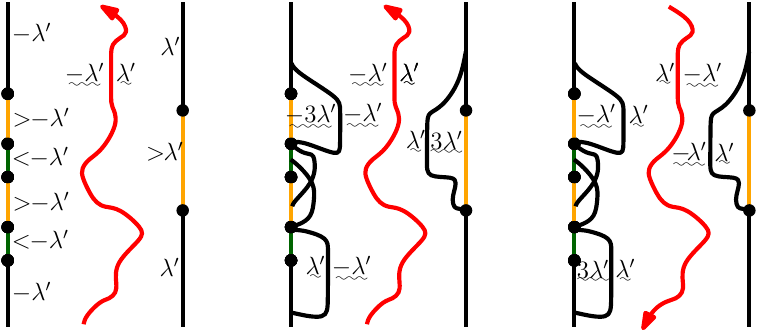}
\end{center}
\caption{\label{fig::shields} This illustrates an alternate proof of Lemma~\ref{lem::multiple_force_special_case}.  In the left figure, the red flow line avoids the boundary. In the middle figure, the {\em shields} (drawn using the procedure described in Figure~\ref{fig::sidematching}) are ``open'' in the sense that their union does not contain a crossing of the strip.  Given that the shields are open, the red path is almost surely boundary avoiding.  Conversely, given that the red path is boundary avoiding, the shields are almost surely open, so these events are equivalent.  One can fix the shields and resample the GFF in between as a GFF (producing a new red flow line) or fix the red flow line and resample the GFF on either side (producing new shields).  The argument used to prove Theorem~\ref{thm::bi_chordal} shows that this resampling invariance characterizes the joint law of the shields and the red path.  The third figure above is obtained by adding $2\lambda'$ to the left of the red path and subtracting $2\lambda'$ from the right, reversing the orientation of the red path.  We can view this as a reversed red path conditioned to be boundary avoiding (with the modified boundary conditions).  The resampling characterization mentioned above and the time reversal symmetry of $\SLE_\kappa(\tfrac{\kappa}{2}-2; \tfrac{\kappa}{2}-2)$ imply that the law is the same as in the second picture, which implies Lemma~\ref{lem::multiple_force_special_case}.}
\end{figure}

Now, what happens if we consider the simple setting of Figure~\ref{fig::simpleswap} but we replace $h$ with $h + \phi$ where $\phi$ is a smooth function whose Laplacian is supported on a subset $S$ contained in the interior of $\vstrip$?  This is equivalent to weighting the law of $h$ by (a constant times) $e^{(h, \phi)_\nabla}$ (this follows, for example, by~\eqref{eqn::gff_definition}).  By integrating by parts, we note that $(h, \phi)_\nabla = (h, -\Delta \phi)$.  Write $A = \int_S -\Delta \phi(z) dz$.  We thus observe:

\begin{proposition} \label{prop::radonforwardreverse}
The Radon-Nikodym derivative of the forward flow line $\eta$ of $h_1 + \phi$ with respect to the reverse flow line of $h_2 + \phi$, as depicted in Figure~\ref{fig::abswap}, is given by $e^{(h_1 - h_2, -\Delta \phi)}$ conditioned on the event that $\eta$ does not hit the boundary.
\end{proposition}

\begin{figure}[htb!!]
\begin{center}
\includegraphics[scale=0.85]{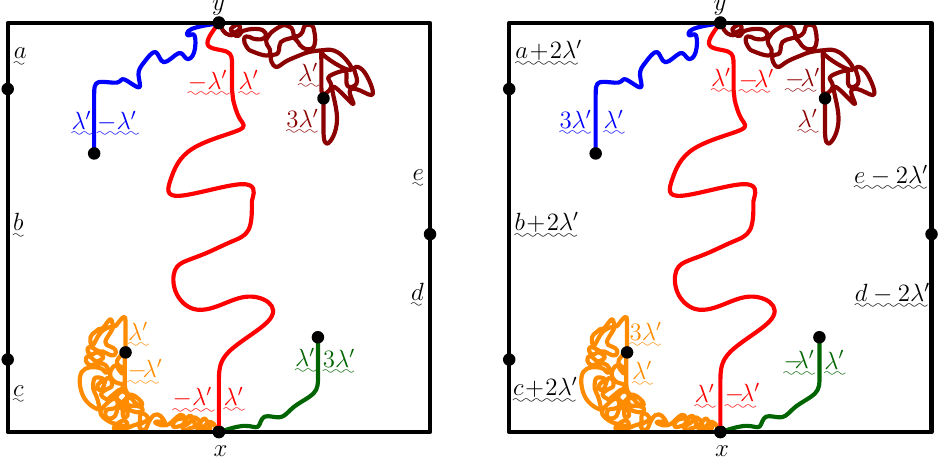}
\end{center}
\caption{\label{fig::reverse_sampling} Consider a version of Figure~\ref{fig::shields} in which the non-matching intervals (where height is not $\pm \lambda'$) extend all the way to infinity on each side.  In this case the ``shields'' go all the way to the top and bottom of the strip; each shield is potentially made up of two types of flow lines, as illustrated in Figure~\ref{fig::sidematching}, but within each shield there is only one type reaching the lower endpoint and one type reach the upper endpoint.  Depending on the type we can trace small pieces of the shield near the endpoints of the red path (conditioned on the red path itself) by using either flow lines or counterflow lines (with the shield as a boundary) as shown.  Definition~\ref{def::avoiding} then admits an alternate but equivalent formulation: instead of resampling the entire shield given the red path, and then resampling the red path given the shield, we can sample these four stubs given the red path (up to some small stopping times), and then resample the red path given the four stubs as a path {\em conditioned} on the (positive probability) event that it avoids the boundary and the stubs.
}
\end{figure}

One interesting observation is that on the event that the path goes completely to the right of $S$ we have deterministically $(h_1, -\Delta \phi)_\nabla - (h_2, -\Delta \phi)_\nabla = -2\lambda' A$.  Since this is a constant, we find that the Radon-Nikodym derivatives of the two weighted measures with respect to each other are constant on this event.  In particular, this means that the law of a forward flow line of $h_1 + \phi$ is equal to the law of a reverse flow line of $h_2 + \phi$ {\em on the event} that the path goes to the left of $S$.  A similar result holds when we condition on the event that the path goes to the right of $S$, or when we write $S = S_1 \cup S_2$ and condition the path to go left of $S_1$ and right of $S_2$.

Now suppose that $\wt{\phi}$ is a function that is harmonic on $\vstrip$ (with non-zero boundary conditions on some compact subset $I$ of the boundary, zero boundary conditions elsewhere).  We can then find a function $\phi$ that agrees with $\wt{\phi}$ outside of a neighborhood $U_I$ of $I$, has Laplacian supported inside of $U_I$, has finite Dirichlet energy, and extends continuously to zero on $\partial \vstrip$.  Using the arguments above for $\phi$ allows us to conclude that the time-reversal of the flow line of $h_1 + \wt{\phi}$ has a law that agrees with that of the flow line of $h_2 + \wt{\phi}$ run in the reverse direction on the event that the path does not hit $I$.  Indeed, this can be written as a limit of the functions discussed above, which allows us to deduce similarly that the flow line of $h_1 + \phi$ has the same law as the reverse flow line of $h_2 + \phi$ (as in Figure~\ref{fig::simpleswapphi}) if we {\em condition} on the path staying outside of $U_I$.  Since this holds for arbitrarily small choices of $U_I$, this proves the following:

\begin{lemma}
\label{lem::multiple_force_special_case}
Theorem~\ref{thm::multiple_force_points} holds in the setting depicted in Figure~\ref{fig::simpleswapphi}, that is, provided the boundary conditions are $-\lambda'$ and $\lambda'$ outside of the finite interval.
\end{lemma}
Note that in this setting there is no difficulty in making the statement of Theorem~\ref{thm::multiple_force_points} precise because the event that the path avoids the boundary has positive probability.  The same argument also applies in the setting of Figure~\ref{fig::abswapphi}, where the boundary conditions are equal to some constant $c$ on the right side and on the left side are equal to $-\lambda'$ outside of a finite interval.


\subsection{Time reversals and shields}

In this section, we describe another way to prove Lemma~\ref{lem::multiple_force_special_case}.  First we observe that if we are given zero boundary conditions on the right side of a strip and piecewise constant boundary conditions on the left side of a strip (with finitely many pieces), then we can draw a collection of flow lines on the left side of the strip (which necessarily avoid the right side) so that conditioned on these flow lines the boundary conditions are $0$ (plus winding) on the left boundary of the original strip minus this collection of flow lines).  This is illustrated and explained in Figure~\ref{fig::sidematching}.  Using this construction, the remainder of the alternate proof of Lemma~\ref{lem::multiple_force_special_case} is sketched in Figure~\ref{fig::shields}.

The resampling characterization in Figure~\ref{fig::sidematching} suggests a way to make sense of Theorem~\ref{thm::multiple_force_points} in a more general setting.

\begin{definition}
\label{def::avoiding}
Given piecewise constant boundary conditions on $\vstrip$ (with finitely many pieces), we say that a random path $\eta$ from the bottom to top of the strip has the law of a flow line {\em conditioned to avoid the boundary} if its law is invariant under the following resampling procedure: first resample the GFF off $\eta$ and generate the corresponding shields, as explained in Figure~\ref{fig::sidematching} and Figure~\ref{fig::shields}.  Then resample the GFF in the region between the shields and take $\eta$ to be the new flow line going from the bottom to the top of the in-between region.
\end{definition}

When we use Definition~\ref{def::avoiding}, Theorem~\ref{thm::multiple_force_points} is immediate from the argument in Figure~\ref{fig::shields}.  However, Definition~\ref{def::avoiding} is not entirely satisfying because we have not shown that there exists a unique law for $\eta$ with this resampling property.  To remedy this, we need the following:

\begin{proposition} \label{prop::uniqueboundaryavoidinglaw}
Given any set of piecewise constant boundary conditions on $\vstrip$ (with finitely many pieces), there exists a unique law for a random path $\eta$ that satisfies the conditions of Definition~\ref{def::avoiding}.
\end{proposition}

\begin{proof} In the special case that the boundary data is constant on each of the two sides of $\vstrip$, Proposition~\ref{prop::uniqueboundaryavoidinglaw} is an obvious consequence of our reversibility and bi-chordal results.  The general proof essentially builds on this observation.  The proof is rather similar to other proofs given in this paper, so we will only sketch the argument here.

First, we note that Definition~\ref{def::avoiding} can be equivalently formulated using ``stubs'' instead of entire shields, as illustrated in Figure~\ref{fig::reverse_sampling}.  Once we condition on the four stubs up to any positive stopping times, we know the conditional laws of the entire shields: namely, they can be drawn using flow or counterflow lines of the Gaussian free field conditioned on a (conditionally positive probability) event that the union of the two shields leaves a space for flow line to get from one endpoint to the other without hitting the boundary.  Given this formulation, we will prove that there is {\em at most one} law satisfying the conditions of Definition~\ref{def::avoiding} using a variant of the proof of Proposition~\ref{prop::conditionalunique2} that involves four stubs instead of two.  To explain this, recall that in the proof of Proposition~\ref{prop::conditionalunique2}, we started with a measure $\nu$ that was (in a natural sense) the law of two GFF flow lines conditioned not to ever intersect each other (where the boundary conditions for this GFF are constant on the two sides of the strip) with the law restricted to the portion of the paths that occurs before the first time they exit some small ball.  Precisely, $\nu$ was just a coupled pair of stopped $\SLE_\kappa(\rho_1, \rho_2)$ curves with parameters depending on the boundary values of the field near the corresponding endpoint.  We then showed that the joint law of the paths in the measure $\mu$ we wanted to understand must be absolutely continuous with respect to $\nu$, with a Radon-Nikodym derivative we could explicitly write down.

In this proof, we replace $\nu$ with $\nu_1 \otimes \nu_2$, where $\nu_1$ corresponds to the law of a pair of flow-or-counterflow lines started from the lower endpoint $x$, conditioned not to intersect each other, and $\nu_2$ is a law (independent of $\nu_1$) of a pair of flow-or-counterflow lines started from the lower endpoint $x$, conditioned not to intersect each other.  Again, each $\nu_i$ has precisely the law of the $\nu$ in the proof Proposition~\ref{prop::conditionalunique2}, given the boundary data near the corresponding endpoint.  As before, it is the law we get by conditioning the paths not to intersect each other while pretending that their are no other force points.

Then we let $\mu$ be the law of the quadruple of stubs in Figure~\ref{fig::reverse_sampling} (where stubs are stopped when they exit the same ball).  The exact arguments used in the proof of Proposition~\ref{prop::conditionalunique2} and Remark~\ref{rem::conditionalunique2_interval_exit} then give us a formula for the Radon-Nikodym derivative of $\mu$ with respect to $\nu$ (up to multiplicative constant).

The {\em existence} half of Proposition~\ref{prop::uniqueboundaryavoidinglaw} can then be established by showing that the $\nu$ expectation of the Radon-Nikodym derivative discussed above is finite.  In other words, we must show that the $\mu$ defined by this Radon-Nikodym derivative is not an infinite measure.  One way to prove this is to start with a configuration as in Figure~\ref{fig::shields} (but with shields running all the way to the endpoints as in Figure~\ref{fig::sidematching}) and repeatedly resample.  In the middle diagram of Figure~\ref{fig::shields}, the Gibbs properties assumed and the bi-chordal theorem imply that we can fix all of the ``downward going'' paths in the shields and resample the upward ones and red flow line as flow lines of the same GFF without any additional conditioning.  (Alternatively, we can just forget about the red flow line and sample the upward going lines of the shield from this law.)  We can also alter boundary conditions to those of the right diagram of Figure~\ref{fig::shields} and resample all of the downward going paths of the shields as ordinary GFF flow lines without additional conditioning.  It is now not hard to show that, uniformly in the choice of downward-going lines, the probability that the upward-going lines go below some level $-R$ tends to zero as $R \to \infty$.  A similar result holds for downward going lines going above level $R$.  Using Remark~\ref{rem::rate_of_convergence}, one can then see that no matter how many times we alternate between resampling the downward and upward lines, the law of paths below $-R$ and above $R$ is still reasonably ``nice'' and in particular, the analog of the sets $\Omega_\epsilon$ in the proofs Proposition~\ref{prop::conditionalunique} and Proposition~\ref{prop::conditionalunique2} have measures that do not tend to zero.  Since the Radon-Nikodym derivative of $\mu$ with respect to $\nu$ on these sets is bounded (as in the proofs mentioned above) it follows that $\mu$ cannot be an infinite measure.
\end{proof}

\bibliographystyle{hmralphaabbrv}
\bibliography{reversibility}

\begin{thebibliography}{{Zha}10a}

\bibitem[BAD96]{BAD96}
G.~Ben~Arous and J.-D. Deuschel.
\newblock The construction of the {$(d+1)$}-dimensional {G}aussian droplet.
\newblock {\em Comm. Math. Phys.}, 179(2):467--488, 1996. \MR{1400748
  (97h:60047)}

\bibitem[Bef08]{BEF_DIM}
V.~Beffara.
\newblock The dimension of the {SLE} curves.
\newblock {\em Ann. Probab.}, 36(4):1421--1452, 2008.
\newblock \arxiv{math/0211322}. \MR{2435854 (2009e:60026)}

\bibitem[CN06]{CN06}
F.~Camia and C.~M. Newman.
\newblock Two-dimensional critical percolation: the full scaling limit.
\newblock {\em Comm. Math. Phys.}, 268(1):1--38, 2006.
\newblock \arxiv{math/0605035}. \MR{2249794 (2007m:82032)}

\bibitem[CS12]{CS10U}
D.~Chelkak and S.~Smirnov.
\newblock Universality in the 2{D} {I}sing model and conformal invariance of
  fermionic observables.
\newblock {\em Invent. Math.}, 189(3):515--580, 2012.
\newblock \arxiv{0910.2045}. \MR{2957303}

\bibitem[Dub05]{DUB_MG_DUALITY}
J.~Dub{\'e}dat.
\newblock {${\rm SLE}(\kappa,\rho)$} martingales and duality.
\newblock {\em Ann. Probab.}, 33(1):223--243, 2005.
\newblock \arxiv{math/0303128}. \MR{2118865 (2005j:60180)}

\bibitem[Dub09a]{DUB_DUAL}
J.~Dub{\'e}dat.
\newblock Duality of {S}chramm-{L}oewner evolutions.
\newblock {\em Ann. Sci. \'Ec. Norm. Sup\'er. (4)}, 42(5):697--724, 2009.
\newblock \arxiv{0711.1884}. \MR{2571956 (2011g:60151)}

\bibitem[Dub09b]{DUB_PART}
J.~Dub{\'e}dat.
\newblock S{LE} and the free field: partition functions and couplings.
\newblock {\em J. Amer. Math. Soc.}, 22(4):995--1054, 2009.
\newblock \arxiv{0712.3018}. \MR{2525778 (2011d:60242)}

\bibitem[HBB10]{HagendorfBauerBernard10}
C.~Hagendorf, D.~Bernard, and M.~Bauer.
\newblock The {G}aussian free field and {${\rm SLE}\sb 4$} on doubly connected
  domains.
\newblock {\em J. Stat. Phys.}, 140(1):1--26, 2010.
\newblock \arxiv{1001.4501}. \MR{2651436 (2011d:60243)}

\bibitem[IK13]{IzyurovKytola10}
K.~Izyurov and K.~Kyt{\"o}l{\"a}.
\newblock Hadamard's formula and couplings of {SLE}s with free field.
\newblock {\em Probab. Theory Related Fields}, 155(1-2):35--69, 2013.
\newblock \arxiv{1006.1853}. \MR{3010393}

\bibitem[Ken01]{KEN01}
R.~Kenyon.
\newblock Dominos and the {G}aussian free field.
\newblock {\em Ann. Probab.}, 29(3):1128--1137, 2001.
\newblock \arxiv{math-ph/0002027}. \MR{1872739 (2002k:82039)}

\bibitem[Lam72]{LAMP72}
J.~Lamperti.
\newblock Semi-stable {M}arkov processes. {I}.
\newblock {\em Z. Wahrscheinlichkeitstheorie und Verw. Gebiete}, 22:205--225,
  1972. \MR{0307358 (46 \#6478)}

\bibitem[Law05]{LAW05}
G.~F. Lawler.
\newblock {\em Conformally invariant processes in the plane}, volume 114 of
  {\em Mathematical Surveys and Monographs}.
\newblock American Mathematical Society, Providence, RI, 2005. \MR{2129588
  (2006i:60003)}

\bibitem[LSW03]{CONF_RES_CHORDAL}
G.~Lawler, O.~Schramm, and W.~Werner.
\newblock Conformal restriction: the chordal case.
\newblock {\em J. Amer. Math. Soc.}, 16(4):917--955 (electronic), 2003.
\newblock \arxiv{math/0209343}. \MR{1992830 (2004g:60130)}

\bibitem[LSW04]{LSW04}
G.~F. Lawler, O.~Schramm, and W.~Werner.
\newblock Conformal invariance of planar loop-erased random walks and uniform
  spanning trees.
\newblock {\em Ann. Probab.}, 32(1B):939--995, 2004.
\newblock \arxiv{math/0112234}. \MR{2044671 (2005f:82043)}

\bibitem[{Mil}10]{MillerSLE}
J.~{Miller}.
\newblock {Universality for SLE(4)}.
\newblock {\em ArXiv e-prints}, October 2010, \arxiv{1010.1356}.

\bibitem[Mil11]{MillerGLCLT}
J.~Miller.
\newblock Fluctuations for the {G}inzburg-{L}andau {$\nabla\phi$} interface
  model on a bounded domain.
\newblock {\em Comm. Math. Phys.}, 308(3):591--639, 2011.
\newblock \arxiv{1002.0381}. \MR{2855536}

\bibitem[MS10]{MakarovSmirnov09}
N.~Makarov and S.~Smirnov.
\newblock Off-critical lattice models and massive {SLE}s.
\newblock pages 362--371, 2010.
\newblock \arxiv{0909.5377}. \MR{2730811}

\bibitem[MS12a]{MS_IMAG}
J.~{Miller} and S.~{Sheffield}.
\newblock {Imaginary Geometry I: Interacting SLEs}.
\newblock {\em ArXiv e-prints}, January 2012, \arxiv{1201.1496}.
\newblock To appear in Probability Theory and Related Fields.

\bibitem[MS12b]{MS_IMAG3}
J.~{Miller} and S.~{Sheffield}.
\newblock {Imaginary geometry III: reversibility of SLE$_\kappa$ for $\kappa
  \in (4,8)$}.
\newblock {\em ArXiv e-prints}, January 2012, \arxiv{1201.1498}.
\newblock To appear in Annals of Math.

\bibitem[MS13]{MS_INTERIOR}
J.~{Miller} and S.~{Sheffield}.
\newblock {Imaginary geometry IV: interior rays, whole-plane reversibility, and
  space-filling trees}.
\newblock {\em ArXiv e-prints}, February 2013, \arxiv{1302.4738}.

\bibitem[NS97]{NS97}
A.~Naddaf and T.~Spencer.
\newblock On homogenization and scaling limit of some gradient perturbations of
  a massless free field.
\newblock {\em Comm. Math. Phys.}, 183(1):55--84, 1997. \MR{1461951
  (98m:81089)}

\bibitem[Phe01]{choquet_lectures}
R.~R. Phelps.
\newblock {\em Lectures on {C}hoquet's theorem}, volume 1757 of {\em Lecture
  Notes in Mathematics}.
\newblock Springer-Verlag, Berlin, second edition, 2001. \MR{1835574
  (2002k:46001)}

\bibitem[RS05]{RS05}
S.~Rohde and O.~Schramm.
\newblock Basic properties of {SLE}.
\newblock {\em Ann. of Math. (2)}, 161(2):883--924, 2005.
\newblock \arxiv{math/0106036}. \MR{2153402 (2006f:60093)}

\bibitem[RV07]{RV08}
B.~Rider and B.~Vir{\'a}g.
\newblock The noise in the circular law and the {G}aussian free field.
\newblock {\em Int. Math. Res. Not. IMRN}, (2):Art. ID rnm006, 33, 2007.
\newblock \arxiv{math/0606663}. \MR{2361453 (2008i:60039)}

\bibitem[RY99]{RY04}
D.~Revuz and M.~Yor.
\newblock {\em Continuous martingales and {B}rownian motion}, volume 293 of
  {\em Grundlehren der Mathematischen Wissenschaften [Fundamental Principles of
  Mathematical Sciences]}.
\newblock Springer-Verlag, Berlin, third edition, 1999. \MR{1725357
  (2000h:60050)}

\bibitem[Sch00]{S0}
O.~Schramm.
\newblock Scaling limits of loop-erased random walks and uniform spanning
  trees.
\newblock {\em Israel J. Math.}, 118:221--288, 2000.
\newblock \arxiv{math/9904022}. \MR{1776084 (2001m:60227)}

\bibitem[Sch11]{MR2334202}
O.~Schramm.
\newblock Conformally invariant scaling limits: an overview and a collection of
  problems [mr2334202].
\newblock In {\em Selected works of {O}ded {S}chramm. {V}olume 1, 2}, Sel.
  Works Probab. Stat., pages 1161--1191. Springer, New York, 2011.
\newblock \arxiv{math/0602151}. \MR{2883399}

\bibitem[She]{She_SLE_lectures}
S.~Sheffield.
\newblock Local sets of the {G}aussian free field: slides and audio.
  www.fields.utoronto.ca/0506/percolationsle/sheffield1,
  www.fields.utoronto.ca/audio/0506/percolationsle/sheffield2,
  www.fields.utoronto.ca/audio/0506/percolationsle/sheffield3.

\bibitem[She07]{SHE06}
S.~Sheffield.
\newblock Gaussian free fields for mathematicians.
\newblock {\em Probab. Theory Related Fields}, 139(3-4):521--541, 2007.
\newblock \arxiv{math/0312099}. \MR{2322706 (2008d:60120)}

\bibitem[She09]{SHE_CLE}
S.~Sheffield.
\newblock Exploration trees and conformal loop ensembles.
\newblock {\em Duke Math. J.}, 147(1):79--129, 2009.
\newblock \arxiv{math/0609167}. \MR{2494457 (2010g:60184)}

\bibitem[{She}10]{2010arXiv1012.4797S}
S.~{Sheffield}.
\newblock {Conformal weldings of random surfaces: SLE and the quantum gravity
  zipper}.
\newblock {\em ArXiv e-prints}, December 2010, \arxiv{1012.4797}.
\newblock To appear in Annals of Probability.

\bibitem[Smi01]{S01}
S.~Smirnov.
\newblock Critical percolation in the plane: conformal invariance, {C}ardy's
  formula, scaling limits.
\newblock {\em C. R. Acad. Sci. Paris S\'er. I Math.}, 333(3):239--244, 2001.
\newblock \arxiv{0909.4499}. \MR{1851632 (2002f:60193)}

\bibitem[Smi10]{S07}
S.~Smirnov.
\newblock Conformal invariance in random cluster models. {I}. {H}olomorphic
  fermions in the {I}sing model.
\newblock {\em Ann. of Math. (2)}, 172(2):1435--1467, 2010.
\newblock \arxiv{0708.0039}. \MR{2680496 (2011m:60302)}

\bibitem[SS05]{SS05}
O.~Schramm and S.~Sheffield.
\newblock Harmonic explorer and its convergence to {${\rm SLE}\sb 4$}.
\newblock {\em Ann. Probab.}, 33(6):2127--2148, 2005.
\newblock \arxiv{math/0310210}. \MR{2184093 (2006i:60013)}

\bibitem[SS09]{SS09}
O.~Schramm and S.~Sheffield.
\newblock Contour lines of the two-dimensional discrete {G}aussian free field.
\newblock {\em Acta Math.}, 202(1):21--137, 2009.
\newblock \arxiv{math/0605337}. \MR{2486487 (2010f:60238)}

\bibitem[SS13]{SchrammShe10}
O.~Schramm and S.~Sheffield.
\newblock A contour line of the continuum {G}aussian free field.
\newblock {\em Probab. Theory Related Fields}, 157(1-2):47--80, 2013.
\newblock \arxiv{1008.2447}. \MR{3101840}

\bibitem[SW05]{SW05}
O.~Schramm and D.~B. Wilson.
\newblock S{LE} coordinate changes.
\newblock {\em New York J. Math.}, 11:659--669 (electronic), 2005.
\newblock \arxiv{math/0505368}. \MR{2188260 (2007e:82019)}

\bibitem[Wer04a]{WER_GIR}
W.~Werner.
\newblock Girsanov's transformation for {${\rm SLE}(\kappa,\rho)$} processes,
  intersection exponents and hiding exponents.
\newblock {\em Ann. Fac. Sci. Toulouse Math. (6)}, 13(1):121--147, 2004.
\newblock \arxiv{math/0302115}. \MR{2060031 (2005b:60262)}

\bibitem[Wer04b]{W03}
W.~Werner.
\newblock Random planar curves and {S}chramm-{L}oewner evolutions.
\newblock In {\em Lectures on probability theory and statistics}, volume 1840
  of {\em Lecture Notes in Math.}, pages 107--195. Springer, Berlin, 2004.
\newblock \arxiv{math/0303354}. \MR{2079672 (2005m:60020)}

\bibitem[Wer05]{WER_CONF_RES}
W.~Werner.
\newblock Conformal restriction and related questions.
\newblock {\em Probab. Surv.}, 2:145--190, 2005.
\newblock \arxiv{math/0307353}. \MR{2178043 (2007g:60015)}

\bibitem[WW13]{WER_WU_CLE_SLE_KR}
W.~Werner and H.~Wu.
\newblock From {${\rm CLE}(\kappa)$} to {${\rm SLE}(\kappa,\rho)$}'s.
\newblock {\em Electron. J. Probab.}, 18:no. 36, 20, 2013.
\newblock \arxiv{1210.3264}. \MR{3035764}

\bibitem[Zha08]{Z_R_KAPPA}
D.~Zhan.
\newblock Reversibility of chordal {SLE}.
\newblock {\em Ann. Probab.}, 36(4):1472--1494, 2008.
\newblock \arxiv{0808.3649}. \MR{2435856 (2010a:60284)}

\bibitem[{Zha}10a]{2010arXiv1004.1865Z}
D.~{Zhan}.
\newblock {Reversibility of Whole-Plane SLE}.
\newblock {\em ArXiv e-prints}, April 2010, \arxiv{1004.1865}.

\bibitem[Zha10b]{Z_R_KAPPA_RHO}
D.~Zhan.
\newblock Reversibility of some chordal {${\rm SLE}(\kappa;\rho)$} traces.
\newblock {\em J. Stat. Phys.}, 139(6):1013--1032, 2010.
\newblock \arxiv{0807.3265}. \MR{2646499 (2011h:60174)}

\end{thebibliography}

\bigskip

\filbreak
\begingroup
\small
\parindent=0pt

\bigskip
\vtop{
\hsize=5.3in
Microsoft Research\\
One Microsoft Way\\
Redmond, WA, USA }

\bigskip
\vtop{
\hsize=5.3in
Department of Mathematics\\
Massachusetts Institute of Technology\\
Cambridge, MA, USA } \endgroup \filbreak
\end{document}